\documentclass[12pt]{article}

\usepackage{braket}
\usepackage{amsmath}
\usepackage{mathtools}
\usepackage{stackrel}
\usepackage{scalerel}
\usepackage{leftidx}
\usepackage{xfrac}
\usepackage[charter]{mathdesign}
\usepackage{euscript}
\let\mathbb\undefined
\usepackage{bbold}
\DeclareSymbolFont{usualmathcal}{OMS}{cmsy}{m}{n}
\DeclareSymbolFontAlphabet{\mathcal}{usualmathcal}

\usepackage{authblk}

\usepackage[shortlabels]{enumitem}

\usepackage{hyperref}
\usepackage{cleveref}
\hypersetup{colorlinks=true, linkcolor=red!50!black, citecolor=green!50!black, urlcolor=blue!80!black}

\usepackage{xcolor}
\usepackage{tcolorbox}
\usepackage{graphicx}
\usepackage{subfigure}
\usepackage{tikz}
\usepackage{tikz-3dplot}

\usepackage[all,2cell,arrow,matrix]{xy} \UseAllTwocells \SilentMatrices
\usepackage{tikz-cd}
\usepackage{tikz}
\usetikzlibrary{arrows,arrows.meta,%
decorations.markings,decorations.pathreplacing,%
decorations.pathmorphing,calc}

\usepackage{verbatim}
\usepackage{comment}

\usepackage[nottoc]{tocbibind} 
\usepackage{appendix}

\usepackage{geometry}
\tikzset{->-/.style={decoration={markings,mark=at position #1 with {\arrow{Stealth}}},postaction={decorate}},->-/.default=0.55}

\usepackage[amsmath,amsthm,thmmarks]{ntheorem}
\theoremstyle{definition}

\newtheorem{thm}{Theorem}[section]
\newtheorem{prop}[thm]{Proposition}

\newtheorem{cor}[thm]{Corollary}
\newtheorem{lem}[thm]{Lemma}

\newtheorem{conj}[thm]{Conjecture}

{
\theoremsymbol{\mbox{${\blacksquare}$}}
\newtheorem{defn}[thm]{Definition}
}
{
\theoremsymbol{\mbox{$\heartsuit$}}
\newtheorem{expl}[thm]{Example}
}
{
\theoremsymbol{\mbox{$\diamondsuit$}}
\newtheorem{rem}[thm]{Remark}
}
\qedsymbol{\mbox{$\square$}}

\numberwithin{equation}{section}
\numberwithin{thm}{section}

\newcommand\be            {\begin{equation}}
\newcommand\ee            {\end{equation}}
\newcommand\bea           {\begin{eqnarray}}
\newcommand\eea         {\end{eqnarray}}
\newcommand\bnu          {\begin{enumerate}}
\newcommand\enu          {\end{enumerate}}
\newcommand\bit          {\begin{itemize}}
\newcommand\eit          {\end{itemize}}

\newcommand{\pf}{\begin{proof}}
\newcommand{\epf}{\qed\end{proof}}

\makeatletter
\providecommand{\leftsquigarrow}{%
  \mathrel{\mathpalette\reflect@squig\relax}%
}
\newcommand{\reflect@squig}[2]{%
  \reflectbox{$\m@th#1\rightsquigarrow$}%
}
\makeatother

\newcommand\CA			{\EuScript{A}}
\newcommand\CB			{\EuScript{B}}
\newcommand\CC			{\EuScript{C}}
\newcommand\CD			{\EuScript{D}}
\newcommand\CE			{\EuScript{E}}

\newcommand\CL			{\EuScript{L}}
\newcommand\CM			{\EuScript{M}}
\newcommand\CN			{\EuScript{N}}

\newcommand\CS			{\EuScript{S}}

\newcommand\CY			{\EuScript{Y}}

\newcommand{\FZ}			{\text{\usefont{U}{euf}{m}{n}Z}}

\DeclareMathOperator{\Hom}{\mathcal{H}om}
\DeclareMathOperator{\End}{\mathcal{E}nd}

\DeclareMathOperator{\id}{id}
\DeclareMathOperator{\Tr}{Tr}

\DeclareMathOperator{\Ind}{Ind}

\DeclareMathOperator{\colim}{colim}

\DeclareMathOperator{\ev}{ev}
\DeclareMathOperator{\coev}{coev}

\DeclareMathOperator{\ob}{ob}

\DeclareMathOperator{\fun}{Fun}
\DeclareMathOperator{\Fun}{\EuScript{F}\mathrm{un}}

\DeclareMathOperator{\Vect}{Vect}

\DeclareMathOperator{\Alg}{Alg}

\DeclareMathOperator{\Mod}{\mathcal{M}od}
\DeclareMathOperator{\LMod}{LMod}
\DeclareMathOperator{\RMod}{RMod}
\DeclareMathOperator{\BMod}{BMod}

\DeclareMathOperator{\Sym}{Sym}

\newcommand{\op}			{\mathrm{op}}

\newcommand\fd			{\mathrm{fd}}

\newcommand\Rep			{\mathcal{R}\mathrm{ep}}
\newcommand\rep			{\mathrm{Rep}}

\newcommand\Irr			{\mathrm{Irr}}

\newcommand{\bscale}	{0.7}
\makeatletter
\newcommand{\ec}[2][]	{{\@ec{#1 |}{#2}}}
\newcommand{\bc}[2][]	{{\@ec{#1}{#2}}}
\newcommand{\@ec}[2]	{\mathchoice
  {\displaystyle \raise.9ex\hbox{$\scaleobj{\bscale}{#1}$} {#2}}%
  {\textstyle \raise.9ex\hbox{$\scaleobj{\bscale}{#1}$} {#2}}%
  {\scriptstyle \raise.55ex\hbox{$\scriptstyle \scaleobj{\bscale}{#1}$} {#2}}%
  {\scriptscriptstyle \raise.38ex\hbox{$\scriptscriptstyle \scaleobj{\bscale}{#1}$} {#2}}%
}
\makeatother

\newcommand\Sl		{\mathrm{SL}}
\newcommand\Gl		{\mathrm{GL}}

\newcommand\rmH		{\mathrm{H}}
\newcommand\HH		{\mathrm{HH}}
\newcommand\SSkalg		{\mathcal{S}\mathrm{kAlg}}

\newcommand\SSk		{\mathcal{S}\mathrm{k}}

\newcommand\Int		{\mathrm{int}}
\newcommand\Sk		{\mathrm{Sk}}
\newcommand\Skalg		{\mathrm{SkAlg}}
\newcommand\SSkcat		{\mathcal{S}\mathrm{kCat}}
\newcommand\Skcat		{\mathrm{SkCat}}

\newcommand\Ann		{\mathbb{A}\mathrm{nn}}
\newcommand\VVect		{\mathcal{V}\mathrm{ect}}
\newcommand\BBim		{\mathcal{B}\mathrm{im}}
\newcommand\PPr		{\mathcal{P}\mathrm{r}}
\newcommand\act		{\mathrm{act}}
\newcommand\DGcat		{\mathrm{DGCat}}
\newcommand\inthom		{\underline{\mathrm{Hom}}}
\newcommand\intend		{\underline{\mathrm{End}}}
\newcommand\Ac		{\mathcal{A}\mathrm{c}}
\newcommand\LLoc		{\mathcal{L}\mathrm{oc}}
\newcommand\HC		{\mathrm{HC}}

\newcommand\Loc	{\mathrm{Loc}}
\newcommand\BBar	{\mathrm{Bar}}
\newcommand\DDMod		{\mathcal{D}\mathrm{Mod}}
\newcommand\QQCoh		{\mathcal{Q}\mathrm{Coh}}

\newcommand\Tot	{\mathrm{Tot}}
\newcommand\Mfld	{\mathcal{M}\mathrm{fld}}
\newcommand\pt	{\mathrm{Pt}}

\newcommand\disk	{\mathrm{Disk}}

\newcommand\gl	{\mathrm{gl}}
\newcommand\ACat	{\mathcal{A}\mathrm{Cat}}

\newcommand\coinv	{\mathrm{coinv}}

\newcommand\hhom	{\mathrm{Hom}}
\newcommand\RHom	{\mathrm{RHom}}
\newcommand\ABim	{\mathcal{A}\mathrm{Bim}}
\newcommand\forg	{\mathrm{Forg}}
\newcommand\Bord	{\mathcal{B}\mathrm{ord}}
\newcommand\Shv	{\mathcal{S}\mathrm{hv}}
\newcommand\PGL	{\mathrm{PGL}}

\tikzset{
  arrowmark/.style 2 args={postaction={decorate},decoration={markings,mark=at position #1 with \arrow{#2}}}
}

	\title{\huge Derived skein module}
	\author{Chun-Yu Bai \thanks{Email: \href{mailto:cyb0375@gmail.com}{\tt C.Bai-2@sms.ed.ac.uk }}}
	\date{\vspace{-5ex}}
\begin{document}

	\maketitle

	\begin{abstract}
We propose a model‑independent axiomatic framework for the derived skein theory of oriented 3-manifolds with coefficients in a ribbon tensor category, especially focusing on input category being the finite-dimensional representations of a quantum group with quantum parameter not a root of unity. The axioms are designed so that the 0th homology recovers the ordinary skein module and gluing is governed by a bar construction. We establish several relationships between the derived skein theory and the ordinary skein theory. We show that this framework yields computable formulas in terms of ordinary internal skein modules and internal skein algebras. We also prove a Hochschild formula for $\Sigma\times S^1$. We also give the first computations of derived skein modules. We establish finiteness properties for a generic parameter using deformation quantization methods.
	\end{abstract}
	
	\tableofcontents
\newpage
\section{Introduction}

\subsection{Skein modules}

Skein modules, introduced by Przytycki \cite{Prz91} and Turaev \cite{Tur90}, package three-dimensional quantum topology  into linear-algebraic objects. Given an oriented 3-manifold $M$ and a reductive group $G$, one takes the vector space $\Sk_G(M)$ spanned by embedded $G$-colored ribbon graphs in an oriented 3-manifold $M$ and imposes local relations. Despite the elementary definition, structural questions about $\Sk_G(M)$ have been notoriously difficult. For example, it was a long-open question whether the dimension of $\Sk_G(M)$ was finite-dimensional for generic parameters. This question was first posed by Witten.

An answer to this question was given in \cite{GJS22} by Gunningham-Jordan-Safronov. Their key conceptual contribution is to replace ad hoc link-by-link manipulations with a computable algebraic model for skein modules built from a Heegaard splitting of $M$. Concretely, they upgrade the usual skein algebra $\Skalg_G(\Sigma)$ and the module $\Sk_G(H)$ of a handlebody $H$ to internal skein objects that retain the hidden internal $G$-equivariance. The ordinary skein algebra/module is then recovered by taking $G$-invariants.  This gives a powerful bridge between low-dimensional topology, combinatorics, tensor categories, higher algebra, quantum algebra, analysis, representation theory, algebraic geometry, and also quantum field theory.

With these internal objects in hand, they proved a relative tensor product formula: for a Heegaard decomposition $M=H_1\bigcup_\Sigma H_2$, the skein module of $M$ is identified with a relative tensor product of internal skein modules  $\Sk^\Int_G(H)$ over the internal skein algebra $\Skalg^\Int_G(\Sigma^\circ)$ of the punctured surface $\Sigma^\circ$. The finiteness theorem then reduces to a finiteness statement about that relative tensor product.

A decisive technical move is to interpret the internal skein module data as (holonomic) deformation quantization modules in the sense of Kashiwara-Schapira, and then invoke general finiteness properties of holonomic objects: the relevant relative tensor product is finite-dimensional for a generic quantization parameter, hence the skein module of a closed 3-manifold is finite-dimensional. This yields the headline result: for every closed oriented 3-manifold $M$ and reductive algebraic group $G$, the $G$-skein module is finite-dimensional over $\mathbb{Q}(A)$.

Placed in the broader framework of skeins, \cite{GJS22} is best read as the "foundational finiteness theorem" that makes subsequent theorems plausible: once skein modules are known to be finite-dimensional, one can pursue some self-duality properties by just comparing the dimensions from both sides. For example, in \cite{Jor23}, Ben-Zvi-Gunningham-Jordan-Safronov conjectured that the generic $G$-skein module is stable under Langlands duality. There are also some variants of the skein Langlands duality in \cite{Jor23}. 

To define the internal skein module, one needs a relative version $\Sk_G(M,X)$ of the skein module that is relative to a boundary condition $X$. There is a natural categorical enhancement for the relative skein module, called the skein category $\Skcat_G(\Sigma)$, whose Hom spaces are given by the relative skein module. This natural definition has been known for a long time \cite{Wal06}\cite{JF15}. However, until recently, the notion of skein category has not found concrete applications. There are several papers that have ignited the skein category research, for example \cite{BZBJ18} \cite{BZBJ18a} \cite{Coo19} \cite{GJS22} \cite{GJV24}. Those papers develop an "engine" from surface gluing to algebraic module to quantization of character varieties. Based on the skein category, one can directly compute $\Sk_G(\Sigma\times S^1)$ by the 0th Hochschild homology of $\Skcat_G(\Sigma)$. This leads to the first known dimension computation of the skein module in \cite{GJV24} in higher rank, beyond the $G=\Sl_2$ Kauffman bracket
skein module. Those computations lead to a highly nontrivial conjecture on skein category: the (heart) of the free cocompletion of  $\Skcat_G(\Sigma)$ (for generic parameter) of closed surface admits a \emph{compact}  \footnote{The most non-trivial point is "compact" here.} projective generator. This conjecture can be viewed as a categorical version of finiteness theorem for the skein module. If this conjecture is true, then it will, in turn, help with the computations of skein modules.

\subsection{Why derived skein modules?}
Now back to the topic of this work, one might ask: Why should such a derived version exist? Why do we want a derived version of skein theory?  Is there a criterion for what the correct notion of derived skein theory should be? Is there anything essentially new beyond ordinary skein theory?
There are several motivations to develop such a derived skein theory.
 
First, it is expected that the skein module is a deformation quantization of the character stack $\Loc_G(M)$ \cite{Bul97}\cite{PS00}\cite{GJS22}. On the other hand, the cohomological Donaldson-Thomas (DT) invariant of $\Loc_G(M)$ is proved to be a deformation quantization of the same stack \cite{Pri19}\cite{Pri26}. In \cite{GJS22} and \cite{GS23}, they relate the generic skein module to the DT sheaf, as in the following conjecture:
\begin{conj}\label{GSconj1}
(\cite{GJS22}\cite{GS23})
Let $q$ be generic, and $M$ connected closed oriented 3-manifold.
\[  \Sk_G(M)\cong \rmH^0(R_G(M),\phi_{R_G(M)})   \]
where the right hand side is sheaf-theoretic framed instanton Floer homology introduced by Abouzaid-Manolescu \cite{AM19}, also called the DT-sheaf on representation variety $R_G(M)$.
\end{conj}

A similar conjecture, by Pavel Safronov, has been made for the case where $q$ is a root of unity:
\begin{conj}
(\cite{Pav22}, see also \cite{Kin24}) Let $M$ be a closed 3-manifold. If $q$ is a good root of unity, there is a line bundle $\CL_q$ on $\Loc_{G^\vee}(M)$ such that 
\[ \Sk_{G,q}(M)\cong \rmH^0(\Loc_{G^\vee}(M),\CL_q)    \]
Here $G^\vee$ is the metaplectic dual group, and $\CL_q$ is a line bundle on the character stack $\Loc_{G^\vee}(M)$.
\end{conj}

It's therefore natural to expect a chain-level enhancement of the skein module whose full homology recovers the DT-sheaf cohomology. A derived skein module is precisely the natural setting for such a comparison.

Second, in \cite{GJS22}, they introduce an important insight that the skein module should be related to the Hilbert space of Kapustin-Witten theory \cite{KW07}, also see Du Pei's \cite{Pei26} for an intensive study. By S-duality of Kapustin-Witten theory, one would expect the skein module is stable under the Langlands duality. Therefore, in \cite{Jor23}, Ben-Zvi-Gunningham-Jordan-Safronov conjectured that 
\begin{conj}
(\cite{Jor23}) Let $q$ be generic, and $M$ connected closed oriented 3-manifold.
\[ \Sk_G(M)\cong \Sk_{G^L}(M) \]
where $G^L$ is Langlands dual of $G$.
\end{conj}
The desired definition for derived skein module should also be stable under the Langlands duality (see Conjecture \ref{TLconj}), which is an immensely more powerful statement and interfaces directly with the Geometric Langlands program. We note, however, that Du Pei \cite{Pei26} proposes a formula for skein dimensions based on physical considerations that is inconsistent with the Langlands duality conjecture, see Section \ref{Langlands}.

Third, the ordinary skein category $\Skcat(\Sigma)$ of a surface, while long defined, only recently found genuine computational use in the works of \cite{BZBJ18}\cite{BZBJ18a}\cite{Coo19}\cite{GJS22}\cite{GJV24}. A derived skein category should carry even richer information: for instance, the derived skein module of $\Sigma\times S^1$ is the Hochschild homology of the derived skein category (Theorem \ref{SkHH}), converting a topological gluing into a purely algebraic Hochschild complex. Such formulas reduce 3-dimensional topology to homological algebra in a systematic way.

There have been other pioneering attempts at derived skein theory. In \cite{Ho45}, Hochschild first studied the derived skein theory in dimension 1.
In \cite{BFK97}, Bullock-Frohman-Kania-Bartoszynska first attempted to generalize the skein module to a chain complex which is a prototype of the disjoint part of the Blob complex.
In \cite{BFN10}, they study the classical derived $G$-skein category of an oriented surface. Also, in \cite{AFMR17}, Ayala-Francis-Mazel-Gee-Rozenblyum study 1-dimensional derived skein theory by higher categorical methods. In \cite{SW21}, Schweigert-Woike study the mapping class group of  derived string-net. Among these dimension settings, the 3-dimensional case is special \footnote{One can see this point from \cite{AFRcor}.} and often leads to rich mathematics.

Despite the clear importance of the notion of derived skein modules (of 3-manifold), the only paper proposing a model for the derived skein module is \cite{MW12} \footnote{Actually Blob complex gives the desired derived skein module is confirmed by this work.}, which has a broader focus and does not give any computations of the derived skein module. However, the Blob complex is heavily model-dependent; it does not provide an intrinsic definition of the derived skein module. In particular, while it is widely expected, it is not proven in the literature that ribbon categories such as the category of representations of quantum groups provide local coefficient systems for blob homology.  See Remark \ref{tanglehypo}.
One reason for the absence of papers in this direction is exactly that  the precise axiomatic definition of derived skein theory is unclear -- for instance, it was unclear whether it should involve more than simply taking a derived functor. For example, in \cite{GJS22} Remark 1, they propose a derived analogue for skein module as a derived functor of ordinary skein module. However, as proposed in this paper Theorem \ref{Sk^intGM-B}, the derived functor of ordinary skein module should be the internal derived skein module $\SSk^\Int_G(M-\mathbb{B})$ of $M-\mathbb{B}$ instead of the derived skein module of $M$. It's a key point that derived theory is sensitive to the 3-handle which is unlike to ordinary skein module. Nevertheless, it is expected that its definition must involve the topology of $M$ in some non-trivial way. On the other hand, the Blob complex also does not build any nontrivial bridge back to ordinary skein theory. These facts make it quite difficult to compute the derived skein module from Blob complex. We want to emphasize that we do not really use the Blob complex in this paper.

We also need to mention that $\beta$-factorization homology, or factorization homology of $(\infty,n)$-categories (rather than $\mathbb{E}_n$-algebras) has been developed by Ayala-Francis-Rozenblyum \cite{AFR18}. We can view $(\infty,n)$-category as a $(\infty,n)$-category enriched in $\infty$-category of topological spaces which is a \emph{cartesian} symmetric monoidal $\infty$-category. However, what we really want, for example category $\Vect$ of vector spaces or $\infty$-category $\VVect$ of chain complexes, are  not \emph{cartesian}. Thus, the desired theory is the factorization homology of $(\infty,n)$-category enriched in \emph{arbitrary} symmetric monoidal $\infty$-category. For the $n=1$ case, it has already been established in \cite{AFMR17}\cite{AMR24}. One would expect the factorization homology of enriched $(\infty,3)$-categories over 3-manifold is (essentially) equivalent to derived skein module of 3-manifold, which means the desired factorization homology of enriched $(\infty,3)$-categories should satisfy the axioms in this paper. There is an essential change between $n=1,2$ and $n=3$ as explained in \cite{AFRcor} which highlights the special nature of three dimension. The slogan is that $\beta$-factorization homology is the derived skein module, while $\alpha$-factorization homology is derived skein category.

\subsection{This paper: an axiomatic framework}
This paper proposes an axiomatic framework for derived skein theory and develops its first formal consequences. We give the first axiomatic definition of the relative derived skein module $\SSk(M,X)$ that is model-independent, characterized by a small set of universal properties (Definition \ref{axiomdefn}). We show that the Blob complex satisfies our axioms, however we note that the axioms alone are sufficient to derive all results in this paper. In particular, we expect our axioms, or minor modifications of them will also be satisfied by an enriched version of $\beta$-factorization homology developing by Ayala and Francis. For doing computations, it is typically preferable to rely on a complete list of algebraic axioms as opposed to a concrete huge topological model.  This is certainly the case in classical algebraic topology, where excision theorems are far more useful than model-speicifc definitions involving singular/simplicial/de Rham co-chains.

A key subtlety in defining derived skein theory is that it is not a naive or straightforward generalization of ordinary skein theory. Instead, it builds a subtle bridge between the derived and underived worlds. Meanwhile, the derived skein module itself carries significant higher homological information, often of a geometric representation-theoretic nature (e.g. yielding the homology of $BG$ or $LBG$). This marriage allows us to use the vast algebraic machinery developed for ordinary skein theory to compute derived invariants, while simultaneously accessing new homotopical phenomena. 

However, we need to emphasize that we only give axioms that determine the derived skein module uniquely up to isomorphism, but that we do not attempt to address the more difficult problem of proving coherence of the isomorphisms between different presentations (Remark \ref{uniqueness}). In particular, we do not recover from our construction the action of the diffeomorphism group of the manifold on the derived skein module. Such equivariance is expected to follow from both the Blob complex and enriched $\beta$-factorization homology models, but is not captured by our axioms.

We do not consider more difficult problems concerning homotopy-coherent structures (Remark \ref{homotopycoherent} and Remark \ref{tanglehypo}). We expect the upcoming paper by Ayala-Francis on factorization homology of rigid braided $(\infty,1)$-category may give us an answer, which will rely on the Tangle Hypothesis \cite{AF24}.    

Our main results are:
\begin{itemize}
\item \textbf{Derived strong finiteness.} (Theorem \ref{finiteness1}) For generic $q$ and reductive group $G$, the  internal derived $G$-skein module $\SSk^\Int_G(M-\mathbb{B})$ is a \emph{bounded} complex with finite-dimensional homologies over $\mathbb{Q}(q^{1/d})$.
\item \textbf{Derived finiteness.} (Theorem \ref{finiteness}) 
Let $G$ be a reductive algebraic group, and $q$ generic. The derived $G$-skein module $\SSk_G(M)$ of a closed oriented 3-manifold is a complex with  finite-dimensional homologies over $\mathbb{Q}(q^{1/d})$.

\item \textbf{Algebraic formula.} (Theorem \ref{computableexcisionforrepqG}) For  $q$ not a root of unity and reductive group $G$, the derived $G$-skein module is isomorphic to
\[  \SSk_G(M)\cong (\Sk^\Int_G(H_1)\otimes^\mathbb{L}_{\Skalg^\Int_G(\Sigma^\circ)}\Sk^\Int_G(H_2))\otimes^\mathbb{L}_{\rmH_\bullet(G)}\mathbb{k}     \]
Here the right hand side uses only the ordinary internal skein modules and internal skein algebra, with derived tensor products taken in a derived $\infty$-category. This formula is the engine for computations and demonstrates that the derived theory is not an isolated abstract construction, but rather a natural roof that sits above the ordinary theory. 
\item \textbf{Hochschild formula.} (Theorem \ref{SkHH}) For any closed surface $\Sigma$ and ribbon tensor category $\CA$,
\[  \SSk_\CA(\Sigma\times S^1)\cong \HC(\SSkcat_\CA(\Sigma))    \]
the Hochschild chain complex of the derived skein category (not the full Hochschild complex of the ordinary skein category).
\item \textbf{Derived quantum Hamiltonian reduction.} (Theorem \ref{QHR} and \ref{DQHR}) The derived skein algebra can be computed as
\[ \SSkalg_\CA(\Sigma)\cong \Hom_{\CD(\CA)}(\mathbb{1},\SSkalg^\Int_\CA(\Sigma^\circ)\otimes^{\mathbb{L}}_{\SSkalg^\Int_\CA(\Ann)}\mathbb{1})      \]

\item \textbf{Explicit computations.} For  $q$ not a root of unity,
\[
\SSk_G(S^3)\cong C_\bullet(BG), \ \ \ \ \ \ \ \ \SSk_G(S^2\times S^1)\cong C_\bullet(G)\otimes C_\bullet(BG),\]
\[ \SSk_{\mathbb{G}_m}(\Sigma_g\times S^1)\cong C_\bullet(B\mathbb{G}_m)\otimes \Lambda^\bullet\mathbb{k}^{2g+1}   \]
As far as we know, these are the only known computations for derived skein modules so far. We actually know more in $\mathbb{G}_m$ case, see Remark \ref{BR26}. We can at least infer three things from those computations:
\begin{itemize}
\item Even when the input category $\CA$ is semisimple (as with not a root of unity $q$), the derived skein module is non-trivial and captures infinite-length homological information.
\item The homology of generic derived skein module can have \emph{infinite} length, with each level finite-dimensional.  In particular it can have infinite total dimension, in contrast to ordinary skein modules.
\item Derived $G$-skein module for $q$ not a root of unity should be stable under Langlands duality (Remark \ref{stableunderL}). 
\end{itemize}

\end{itemize}

\subsection{Outlook}

We give some outlooks here on the predictable potential effects on other areas: 
\begin{enumerate}
\item Our work opens a new door between skein theory and the geometric representation theory of quantum groups and loop spaces, building on the ideas that have been developed over the past two decades by Ben-Zvi, Drinfeld, Gaitsgory, Nadler, and many others. The derived $G$-skein category of $S^2$ provides a new model for the DG category $\DDMod$ of D-modules on $BG$. (Cor \ref{SkcatS2=DModBG}), which is invisible in ordinary skein theory. We also believe that the methods introduced here will serve as the basis for a "skein-theoretic six-functor formalism" analogous to the six-functor formalism for D-modules on stacks (Remark \ref{Skein6functor}). We would like to return to these questions in the future. By connecting Blob complex with the derived DT-skein conjecture, our work also provides a topological interpretation of cohomological DT invariants of the character stack of 3-manifold. Our work also opens a new door between skein theory and string topology. For example, we show that derived skein module $\SSk_G(S^2\times S^1)$ is equivalent to chains on loop space $LBG$ (Theorem \ref{SkGS^2timesS1}), or equivalently the Hochschild homology of the skein category of $S^2$ is equivalent to Hochschild homology of string topology category on $BG$ (Remark \ref{stringtopologycategoryonBG}). This is also invisible in ordinary skein theory.

\item One of the meanings of the paper \cite{GJV24} is that the dimension formulas of skein modules display interesting combinatorics. This interesting combinatorics suggests many generalizations and connections to type A algebraic combinatorics. The combinatorics appearing in derived skein modules should be far more interesting.

\item The 4-dimensional skein module, also called skein Lasagna module introduced in \cite{MWW22} based on Morrison-Walker's Blob complex \cite{MW12}, has become an important object of study in 4-manifold topology. For example, in Ren-Willis's fantastic paper \cite{RW24}, they show that it can be used to detect exotic structures on 4-manifolds in an analysis-free way. In dimension 3, Kalfagianni, Detcherry, Sikora, Przytycki, Belletti and others have done a lot of important work in this direction for example \cite{Prz98}\cite{BD25}\cite{Kal25}. We also expect derived skein module to work well in revealing deep topology or geometry of 3-manifold.

\item In \cite{BJ25}, Brown-Jordan develop a study of ordinary skein modules of 3-manifolds in the presence of codimension-one defects. The paper explains why cluster coordinates naturally appear in defect skein theory. Ideal triangulations give edge/cross-ratio coordinates. After quantization, these coordinates generate quantum tori. That is the cluster structure. They provide an effective elementary cluster-algebraic formalism for computing knot invariants and quantum A-ideals. 
There are related papers in this direction, such as classical Blob complex paper \cite{MW12}, Jordan-Le-Schrader-Shapiro's work \cite{JLSS21}, Dimofte-Garoufalidis-Yu's works \cite{Dim13}\cite{Gar04}\cite{DimGar13}\cite{GY24}\cite{GY26}, Panitch-Park's work \cite{PP24}, and also Bonahon–Wong's works \cite{BW10a}\cite{BW11}\cite{BW16}\cite{BW17}. For the finite category case, Carqueville-Runkel-Schaumann \cite{CRS19} \cite{CR16} also studied defects in skein theory using orbifold completion. In Julia Bierent's upcoming paper \cite{Bie}, she will discuss defect skein category which quantized quasi-coherent sheaves of \emph{wild} character varieties. In Matthias Vancraeynest's upcoming paper \cite{Van}, he will discuss the relation between weaves and defects in skein theory.  One can also add defects into the derived skein module to produce new knot invariants.

\item One could study the derived skein modules of 3-manifolds with boundary and their finite generation properties by combining this paper with Jordan-Romaidis \cite{JR25}.

\item If the input category $\CA$ is semisimple, then the internal derived skein algebra is just internal skein algebra (Theorem \ref{semisimpleintskalg}). One can try to figure out what algebras the derived skein algebra and internal derived skein algebra give when $\CA$ is non-semisimple. This may give us some interesting quantum algebras. Even for semisimple case, one would expect Theorem \ref{QHR} to give some interesting DG algebras as in Examples \ref{SkalgS^2=C(g)} and \ref{SkalgGm(Sigmag)}. Those DG algebras will be useful. For example, it will be useful for the computation of $\SSk_G(\Sigma\times S^1)$ by  \ref{SkHH} and \ref{GJVconj}.

\item We propose a derived skein $(3,2,1)$-TQFT (or mathematically speaking somewhat non-rigorously, a $(3,2,1,0)$-TQFT) (Remark \ref{skeinTQFT}). It would be interesting to understand what the assignment for a 4-manifold is. We know partially that it should be homologies of the derived skein module $\SSk(M)$ when assigned to $M\times S^1$. There are several nice works in this direction, by Costantino-Geer-Ha\"ioun-Patureau-Mirand \cite{Hai24}\cite{CGBPM26}, in which they use handle attachment and categorical bottom-to-top methods. Czenky-Negron's paper \cite{CN25} may also be relate to this direction.

\item The foundations laid in this paper make a host of previously unreachable conjectures well-posed and amenable to investigation. The derived DT-skein comparison conjecture, the derived Langlands duality conjecture, and the compact generation conjecture are now concrete statements that can be tested and proved.
\end{enumerate}

\subsection{Notations}
Usually when one consider derived things, one will say "everything is derived". But in this paper, underived notions will still work. That is, this paper builds a bridge between ordinary skein theory, derived skein theory, not simply consider everything to be derived (for example, simply replacing every tensor product by derived tensor product, every hom functor by derived hom functors, etc). If a notation begins by $\SSk$, that means it lives in derived skein world. If a notation begins by $\Sk$, that means it lies in ordinary skein world.
Fix:
\begin{itemize}
\item Let $\mathbb{k}$ be a $\mathbb{Z}[q,q^{-1}]$-algebra.

\item $q$ quantum parameter. We will say $q$ is \emph{generic} to mean $\mathbb{k}=\mathbb{Q}(q^{1/d})$. Here integer $d$ is divisible by the determinant of the Cartan
matrix, so it is typical to fix that minimal choice for $d$. We will say $q$ is \emph{not a root of unity} to mean either that $q$ is generic, or that $\mathbb{k}=\mathbb{C}$ and $q^l\ne 1$ for all non-zero integers $l$.

\item $\Sigma$: compact oriented topological surface. Hence $\Sigma\cong \Sigma_{g,r}$ where $\Sigma_{g,r}$ is the standard surface of genus $g$ and with $r$ punctures; We assume that $\Sigma$ has a chosen disk, then we denote the punctured surface by $\Sigma^\circ$.
\item $G$: connected reductive algebraic group over $\mathbb{k}$, e.g. $\mathbb{G}_m,\Sl_n,\Gl_n$;
\item $\CA$: ribbon tensor category, in particular it is $\mathbb{k}$-linear, e.g. $\rep^{\fd}_q(G)$;
\item $M$: oriented compact 3-manifold.
\item $\rep_q(G)$ is the heart $\widehat{\rep^\fd_q(G)^\heartsuit}$ of the free cocompletion $\widehat{(-)}$ of $\rep^{\fd}_q(G)$. $\Rep_q(G)$ is the derived $\infty$-category $\CD(\rep_q(G))$ of $\rep_q(G)$.
\item Let $\PPr$ be the $\infty$-category of stable presentable $\mathbb{k}$-linear $\infty$-categories, colimit-preserving functors.
\item Let $\PPr^\circ$ be the full subcategory of $\PPr$ on those that are generated by compact objects.
\item Let $\PPr_c$ be the $\infty$-category of compactly generated stable presentable $\mathbb{k}$-linear $\infty$-categories, compact functors.
\end{itemize}

\subsection{Outline}

Note that for the sake of brevity, this paper assumes a great deal of prior knowledge. The author hopes to add a more thorough background section in his PhD thesis.

This paper is organized as follows.  In section \ref{categoricalsetting} we set up the categorical language. We recall some basic facts of DG categories and stable presentable $\infty$-categories. We show an equivalence between the Morita category of DG categories and the $\infty$-category of compactly generated stable presentable $\infty$-categories and cocontinuous functors (Theorem \ref{Bim=Prcirc}). We introduce derived $\infty$-categories. We prove an equivalence between the derived $\infty$-category of a linear category and the free cocompletion of it (Lemma \ref{DC=hatC}). We show that if an algebra lies in heart, then taking modules of it commutes with taking derived category (Lemma \ref{DLMod=LModhat}).

Section \ref{Derivedskeintheory} is the heart of this paper. It contains the axiomatic definition of the derived skein modules and categories. We show the derived skein category satisfies the excision property, hence equivalent to factorization homology. We also introduce internal variants of these notions and build bridges to ordinary skein theory. 

Section \ref{computeableformulaandresults} derives computable formulas and explicit examples. 

Section \ref{DQmoduleandfinite} gives the proof of the derived finiteness theorem using DQ-module techniques. 

Section \ref{ConjandExpect} proposes some conjectures in derived skein theory.

\subsection{Acknowledgements}

I would like to greatly thank my supervisor David Jordan who guided me to skein theory; suggested to me to explore this direction; proposed two important expectations at the early stage of this project including that the derived skein category should be equivalent to factorization homology and the Theorem \ref{SkHH} should be true; The starting point of this paper is just to realize these two expectations. I would also want to thank him for suggesting me some very useful references;  a lot of useful feedback when I come up with ideas on this project; kindly answering my questions on his papers; giving me opportunities to present in talks and posters; always appreciating me and much encouragement throughout this work. 

I would like to thank Pavel Safronov for sharing with the author their expectation/conjecture (joint with Sam Gunningham) on ordinary skein module as 0th part Renormalized cohomology, and I especially thank him for carefully reading the draft and very helpful comments.

I would like to thank Sam Gunningham for pointing out Lemma \ref{SkalgS^2=C(g)} to the author, and helpful comments on the draft.

I would like to thank Benjamin Ha\"ioun and Qiuyu Ren for useful discussions on the axioms for the derived skein module.

I would like to thank Kevin Walker for confirming with the author that theorem 7.2.1 in \cite{MW12} is just derived coend property, and explaining to the author the role of nested balls.

I would like to thank Clark Barwick and David Ayala for useful comments on category theory.

I would like to thank John Francis for a helpful comment on Blob complex.

I would like to thank Jennifer Brown for helpful comments on the draft.

I would like to thank Adam Sikora for suggesting to me the paper \cite{BFK97} after I presented my poster in the conference "New perspectives on Skein modules".

I would also like to thank the University of Hamburg for two weeks' hospitality during preparation of this work, and especially thank Prof. Christopher Schweigert, Prof. Ingo Runkel and Paul Wedrich for their encouragements and remarks.

I gratefully acknowledge research travel support provided by Simons Foundation award 888988 as part of the Simons Collaboration on Global Categorical Symmetries.

\section{Categorical setting}\label{categoricalsetting}

\subsection{DG categories and stable presentable \texorpdfstring{$\infty$}{}-categories}

All categories and functors are assumed to be $\mathbb{k}$-linear. All colimits are homotopy colimits.

We briefly recall the theory of DG categories. Also see Appendix \ref{AinftyandDG} for a comparison between $A_\infty$-categories and DG categories.

A \emph{chain complex} $C^\bullet$ of $\mathbb{k}$-modules, or simply a \emph{complex} $C^\bullet$, is a sequence of $\mathbb{k}$-module maps
\[  \cdots\xrightarrow{d_{n+1}}C^n\xrightarrow{d_n}C^{n-1}\xrightarrow{d_{n-1}}\cdots\xrightarrow{d_2}C^1\xrightarrow{d_1}C^0\xrightarrow{d_0} C^{-1}\xrightarrow{d_{-1}}\cdots   \]
such that $d_{i}\circ d_{i+1}=0$. We call $d$ differential maps. We are mainly interested in non-negative chain complexes, i.e. we take $C^{-n}=0$ if $n>0$. An element $x\in C^n$ is a chain of degree $n$. We adopt the notation $|x|=n$. The \emph{homology groups} are denoted by $\rmH_\bullet(C)$. Let $(C^\bullet,d)$ be a complex and let $p$ be an integer. The \emph{p-shifted complex} $C[p]$ is the complex such that $C[p]^n=C_{n-p}$ with differential map $(-1)^p d$. In fact, it is the tensor product of $\mathbb{k}[p]$ (which is $\mathbb{k}$ concentrated in degree $p$) with $C$.

\begin{defn}
A \emph{DG category} is a $\mathbb{k}$-linear category $\CC$ whose Hom spaces $\Hom_\CC(V,W)$ are chain complexes and whose compositions
\[  \Hom_\CC(W,U)\otimes\Hom_\CC(V,W)\to \Hom_\CC(V,U)    \]
are chain maps satisfying a strongly associativity condition. The identity morphism $1_V\in \Hom_\CC(V,V)$ is closed of degree zero. We say \emph{cocomplete DG category} if the associated stable $\mathbb{k}$-linear $\infty$-category admits all colimits or equivalently arbitrary direct sums. 
\end{defn}

\begin{expl}
\begin{enumerate}
\item Let $\Mod_{\mathbb{k}}$ be the DG category of all complexes over $\mathbb{k}$. We denote the associated stable $\infty$-category as $\VVect$. It has a canonical t-structure whose heart $\Vect:=\VVect^\heartsuit$ is the abelian category of $\mathbb{k}$-modules.
\item The delooping $BA$ of a DG algebra $A$ is a DG category.
\end{enumerate}
\end{expl}

One defines for every DG category $\CC$ the \emph{opposite DG category} $\CC^\op$ with $\ob \CC^\op=\ob \CC$ and $\Hom_{\CC^\op}(V,W)=\Hom_\CC(W,V)$.

\begin{defn}
The \emph{homolopy category} $\rmH_0(\CC)$ of a DG category $\CC$ has the same objects whose morphism spaces are given by the zeroth homologies $\hhom_\CC(V,W):=\rmH_0(\Hom_\CC(V,W))\in \Vect$ of the Hom complexes $\Hom_\CC(V,W)$ in $\CC$.
\end{defn}

\begin{defn}
A \emph{DG functor} $F:\CC\to \CD$ is given by a map $F:\ob \CC\to \ob \CD$ and by chain maps
\[  F(V,W):\Hom_\CC(V,W)\to \Hom_\CC(F(V),F(W))     \]
compatible with the composition and the units.
\end{defn}

\begin{defn}
The \emph{tensor product $\CC \boxtimes \CD$ of DG categories} is the DG category whose
\begin{itemize}
\item the object class is $\ob \CC\times \ob \CD$, whose element $(V,W)$ is denoted by $V\boxtimes W$
\item the Hom complex is
\[ \Hom_{\CC\boxtimes \CD}(V\otimes W,V'\otimes W')=\Hom_\CC(V,V')\otimes \Hom_\CD(W,W')     \]
with 
\[d(f\otimes g)=d_\CC(f)\otimes g+(-1)^{|f|}f\otimes d_\CD(g)\]
and 
\[ (f\otimes g)\circ (f'\otimes g')=(-1)^{|g|\cdot |f'|}f\circ f'\otimes g\circ g'   \]
\end{itemize}

\end{defn}
\begin{expl}
Let $\CC$ be a DG category. Then $\CC\boxtimes \Mod_{\mathbb{k}}=\CC$.
\end{expl}

\begin{defn}
Let $\CC\in \DGcat$ be a DG category. A \emph{right $\CC$-module} is a DG functor $\CC^\op\to \Mod_{\mathbb{k}}$. Similarly, \emph{left $\CC$-module} is a DG functor $\CC\to \VVect$. A  \emph{DG $(\CC,\CD)$-bimodule}  is a DG functor $\CC\boxtimes\CD^\op\to \Mod_{\mathbb{k}}$.
\end{defn}

\begin{defn}
Let $\CC,\CD$ be two DG categories. The \emph{functor DG category} $\fun(\CC,\CD)$ is the DG category whose
\begin{itemize}
\item objects are DG functors from $\CC$ to $\CD$;
\item a morphism $\eta:F\to G$ of degree $p$ between two such functors is given by a family $\eta=(\eta_V)_{V\in \ob \CC}$ of morphisms $\eta_V:F(V)\to G(V)$ of degree $p$ in $\CD$ satisfying
\[ G(f)\circ \eta_V=(-1)^{|\eta|\cdot|f|}\eta_{V'}\circ F(f)    \]
for all $f:V\to V'$. These morphisms form the $p$-th component of the Hom-complex of $\fun(F,G)$. The differential is given by
\[ (d\eta)_V=d_\CD(\eta_V)   \]
\end{itemize}
\end{defn}

\begin{expl}
Let $\CC$ be a DG category. Then $\Mod_\CC:=\fun(\CC^\op,\Mod_{\mathbb{k}})$ is a (compactly generated pretriangulated cocomplete) DG category.
\end{expl}

We also call $\widehat{\CC}:=\fun(\CC^\op, \VVect)$ free cocompletion of $\CC$. The relation between $\widehat{\CC}$ and $\Mod_\CC$ is given in Lemma \ref{DC=hatC}.

\begin{rem}
A cocomplete DG category is automatically pretriangulated.
\end{rem}

For us, it's not harmful to replace cocomplete DG category by the notion of \emph{stable presentable $\mathbb{k}$-linear $\infty$-category} see \cite{Lur17}. Category pragmatists are recommended to read Chapter 1 of Gaitsgory-Rozenblyum's book \cite{GR19}. We will use these two names interchangeably.

\begin{defn}
An object $V$ in a cocomplete DG category $\CC$ is \emph{compact} if one (hence all) of the following equivalent conditions hold:
\begin{enumerate}
\item the functor 
\[\hhom_\CC(V,-):\CC\to \VVect^\heartsuit\] commutes with arbitrary direct sums.
\item the (a priori non-cocontinuous) functor
\[ \Hom_\CC(V,-): \CC\to \VVect    \]
commutes with all colimits.
\end{enumerate}
\end{defn}
\begin{rem}
Compared with the linear case, compact object in a presentable linear category means the hom functor preserves filtered colimits while a compact projective object means the hom functor preserves all colimits. 
\end{rem}

\begin{defn}
The \emph{derived coend} $\int^{d\in \CD} F(c,d)\otimes^\mathbb{L} G(d,e)$ of DG functor $F:\CC\boxtimes \CD^\op \to \VVect$ and DG functor $G:\CD\boxtimes \CE^\op\to \VVect$ is the geometric realization of simplicial chain complex:
\[ \int^{d\in \CD} F(c,d)\otimes^\mathbb{L} G(d,e):=|\BBar_{\bullet}(F,\CD,G)(c,e)|    \]
We also denote $\int^{d\in \CD}F(-,d)\otimes^\mathbb{L}G(d,-)$ by $F\otimes^\mathbb{L}_\CD G$ \footnote{Both these two notations are helpful. For $\int$-type notation will be helpful for coend computation, and $\otimes^\mathbb{L}$ is good for intuition}.
\end{defn}

\begin{rem}\label{geometricrelation=doublecomplex}
By \cite{ARA23} Prop 2.4, In $\VVect$, geometric realization of the simplicial chain complex is given by direct sum totalization of the double complex associated to the simplicial chain complex. 
\end{rem}

\begin{defn}
\cite{Toe07} The symmetric monoidal $(\infty,1)$-category $\DGcat$ is the $\infty$-nerve of the following category:
\begin{itemize}
\item objects are small DG categories; 
\item morphisms from $\CC$ to $\CD$ the DG functors $\CC\to \CD$; 
\item tensor product is given by tensor product $\boxtimes$ of DG categories.
\end{itemize}
\end{defn}

\begin{defn}
\cite{Tab07} The symmetric monoidal $(\infty,1)$-category $\BBim$ is the $\infty$-nerve of the following category:
\begin{itemize}
\item objects are small DG categories.
\item 1-morphisms from $\CC$ to $\CD$ are DG functors $F:\CC\boxtimes \CD^\op\to \VVect$.
\item The composition of $F:\CC\boxtimes \CD^\op\to \VVect$ and $G:\CD\boxtimes \CE^\op\to \VVect$ is the functor $F\otimes^\mathbb{L}_\CD G:\CC\boxtimes \CE^\op\to \VVect$ given by the derived coend \[(F\otimes^\mathbb{L}_\CD G)(c,e)=\int^{d\in \CD}F(c,d)\otimes^\mathbb{L} G(d,e)\]
\item Tensor product $\boxtimes$ is given by tensor product of DG categories.
\end{itemize}
We call two DG categories $\CC$ and $\CD$ are \emph{Morita equivalent} if they are equivalent in $\BBim$.
\end{defn}

\begin{expl}\label{widehatF=derivedfunctor}
In the case that $\CE=\VVect$ and ${}_{\CD}F_\CC=\Hom_\CD(-,F(-))$ which is induced by a DG functor $F:\CC\to \CD$, the derived coend ${}_{\CD}F_\CC\otimes^\mathbb{L}_\CC -$ is just derived left Kan extension (or so call derived extension of scalars) along $F$ and we denote it by $\widehat{F}$. In particular, when $\CC,\CD$ are linear categories viewed as DG categories, and $F$ is a linear functor viewed as DG functor, it also makes sense to consider ordinary coend ${}_\CD F_\CC\otimes_\CC -$ which can be also viewed as ordinary left Kan extension (denote by $\widehat{F}^\heartsuit$) along $F$. Then in this case, $\widehat{F}$ is left \footnote{The ordinary left Kan extension $\widehat{F}^\heartsuit$ is a left adjoint to precomposition, so it is right exact.} derived functor of $\widehat{F}^\heartsuit$.
\end{expl}

\begin{expl}
Let $\CE=\CC=\VVect$, $\CD$ is $\CD\boxtimes \CD^\op$ and $F=G={}_\CD\id_\CD$ induced by identity DG functor $\id:\CD\to \CD$. Then the derived coend ${}_\CD\id_\CD\otimes^\mathbb{L}_{\CD\otimes \CD^\op}{}_\CD\id_\CD$ is Hochschild chain complex $\HC(\CD)$  of $\CD$. Note that if $\CC$ is morita equivalent to $\CD$, then $\HC(\CC)\cong \HC(\CD)$.
\end{expl}

\begin{expl}\label{coendmorita}
Let $\CC=\CE=\VVect$, and $\CD$ is Morita equivalent to $BA$ for a DG algebra $A$. Then $F\otimes^{\mathbb{L}}_\CD G\cong F(\mathbb{1})\otimes^{\mathbb{L}}_A G(\mathbb{1})$ where the right side is the usual derived tensor product over DG algebra.
\end{expl}

\begin{defn}
Let $\CC\in \DGcat$ be a DG category with a distinguished object $\mathbb{1}\in \CC$ and let $F:\CC^\op\to \VVect$ be a left $\CC$-module. We say $F$ is \emph{generated by invariants} if the canonical map is a natural quasi-isomorphism
\[ \Hom_\CC(-,\mathbb{1})\otimes^\mathbb{L}_{\End_\CC(\mathbb{1})} F(\mathbb{1})\xrightarrow{\sim} F(-)    \]

\end{defn}
\begin{expl}\label{exampleofgeneratedbyinvariants}
By Example \ref{coendmorita} and coend form of Yoneda lemma, if $\CC$ is morita equivalent to $B\End_\CC(\mathbb{1})$, then any $\CC$-module is generated by invariants.
\end{expl}

The property of "generated by invariants" will give us a computable form of derived coend
\begin{prop}\label{G1End1F1=GCF}
Suppose $F:\CC^\op\to \VVect$ and $G:\CC\to \VVect$ are generated by invariants. Then the map 
\[ G(\mathbb{1})\otimes^\mathbb{L}_{\End_\CC(\mathbb{1})}F(\mathbb{1})\to G\otimes^\mathbb{L}_\CC F    \]
is an isomorphism. Here left-hand side is the usual derived relative tensor product over DG algebra.
\end{prop}
\begin{proof}
Let $A=\End_\CC(\mathbb{1})$. Consider the following simplicial chain complex
\[ \xymatrix{ \Hom_\CC(-,\mathbb{1})\otimes A\otimes A\otimes F(\mathbb{1})  \ar@<0.5ex>[r] \ar[r] \ar@<-0.5ex>[r]  & \Hom_\CC(-,\mathbb{1})\otimes A\otimes F(\mathbb{1}) \ar@<0.25ex>[r] \ar@<-0.25ex>[r] & \Hom_\CC(-,\mathbb{1})\otimes F(\mathbb{1}) }    \]
Its colimit is $\Hom_\CC(-,\mathbb{1})\otimes^\mathbb{L}_{\End_\CC(\mathbb{1})}F(\mathbb{1})$ by definition.
Since $F$ is generated by invariants, we can write $F$ as the homotopy colimit of above simplicial chain complex by Yoneda lemma. Similarly,
\[ \xymatrix{ G(\mathbb{1})\otimes A\otimes A\otimes \Hom_\CC(\mathbb{1},-)  \ar@<0.5ex>[r] \ar[r] \ar@<-0.5ex>[r]  & G(\mathbb{1})\otimes A\otimes \Hom_\CC(\mathbb{1},-) \ar@<0.25ex>[r] \ar@<-0.25ex>[r] & G(\mathbb{1})\otimes \Hom_\CC(\mathbb{1},-) \ar[r] & G(-)  }    \]
is a homotopy colimit diagram.

We have 
\begin{align*}
G\otimes^\mathbb{L}_\CC F&:=\int^{c\in \CC} G(c)\otimes^\mathbb{L} F(c) \\
&\cong \int^{c\in \CC} 
\colim(\xymatrix{  \ar@<0.25ex>[r] \ar@<-0.25ex>[r] & G(\mathbb{1})\otimes \Hom_\CC(\mathbb{1},c)  })\otimes^{\mathbb{L}}\colim(\xymatrix{  \ar@<0.25ex>[r] \ar@<-0.25ex>[r] & \Hom_\CC(c,\mathbb{1})\otimes F(\mathbb{1})  }) \\
&\cong \colim( \xymatrix{ \ar@<0.75ex>[r] \ar@<0.25ex>[r] \ar@<-0.25ex>[r] \ar@<-0.75ex>[r] & \int^{c\in \CC}G(\mathbb{1})\otimes \Hom_\CC(\mathbb{1},c)\otimes^\mathbb{L} \Hom_\CC(c,\mathbb{1})\otimes F(\mathbb{1})      }) \\
&\cong \colim( \xymatrix{ G(\mathbb{1})\otimes A\otimes A\otimes A \otimes F(\mathbb{1}) \ar@<0.75ex>[r] \ar@<0.25ex>[r] \ar@<-0.25ex>[r] \ar@<-0.75ex>[r] & G(\mathbb{1})\otimes A \otimes F(\mathbb{1})      })\\
&\cong G(\mathbb{1})\otimes^{\mathbb{L}}_A A\otimes^{\mathbb{L}}_A  F(\mathbb{1}) \\
& \cong G(\mathbb{1})\otimes^\mathbb{L}_A F(\mathbb{1})
\end{align*}

\end{proof}

\begin{defn}
The symmetric monoidal $\infty$-category $\PPr$ has:
\begin{itemize}
\item objects are stable presentable $\mathbb{k}$-linear $\infty$-categories;
\item morphisms are cocontinuous functors.
\end{itemize}
\end{defn}

We have symmetric monoidal functors
\begin{align*}
\widehat{(-)}:\DGcat & \to \PPr \\
\CC&\mapsto \fun(\CC^\op,\VVect)\\
F&  \mapsto  \widehat{F}
\end{align*}
where $\widehat{F}$ is an derived left Kan extension along $F$.

Denote $\PPr^\circ$ be the full subcategory of $\PPr$ on those compactly generated stable presentable $\mathbb{k}$-linear $\infty$-category.
\begin{thm}\label{Bim=Prcirc}
There is an equivalence of $\infty$-categories
\[ \widehat{(-)}:\BBim \xrightarrow{\sim} \PPr^\circ : (-)^c    \]
\end{thm}
\begin{proof}
The free cocompletion of a DG category is a compactly generated stable presentable $\mathbb{k}$-linear $\infty$-category \cite{Lur17}\cite{Toe07}. The full subcategory of compact objects of a compactly generated stable presentable $\infty$-category is an idempotent complete stable $\infty$-category. Given an DG category $\CC$, by definition $(\widehat{\CC})^c$ is the idempotent completion of $\CC$. Since two DG categories are equivalent in $\BBim$ (i.e they are Morita equivalent) iff their idempotent completions are equivalent, we have $\CC\simeq (\widehat{\CC})^c$. Given a $\CS$ in $\PPr^\circ$, by definition, we have $\widehat{(\CS^c)}\simeq \CS$.
\end{proof}

\begin{defn}
Let $\CC,\CD$ be $\infty$-categories. A functor $F:\CC\to \CD$ is said to be \emph{conservative} if for a morphism $f:x\to y$ in $\CC$, is such that $F(f)$ is an isomorphism, then $f$ is an isomorphism.
\end{defn}

Denote $\PPr_c\subset\PPr^\circ$ be the symmetric monoidal $\infty$-category of compactly generated stable presentable $\infty$-categories, compact-preserving functors. 

\begin{defn}
A \emph{tensor category} in $\PPr_c$ is an $E_1$-algebra in $\PPr_c$. Similarly a \emph{braided tensor category} in $\PPr_c$ is an $E_2$-algebra in $\PPr_c$.
\end{defn}
\begin{expl}
Let $\CA$ be a (ribbon) tensor category, then $\widehat{\CA}$ is a (ribbon) tensor category in $\PPr_c$.
\end{expl}

\begin{defn}
Let $\CC$ be a tensor category in $\PPr_c$, and $\CM\in \PPr_c$ be a $\CC$-module category. We say that $m\in \CM$ is:
\begin{itemize}
\item an \emph{$\CA$-generator} if for every object $n\in \CM$, there exists an object $a\in \CA$ and a map $\act_m(a)\to n$ in $\CM$ such that the induced map $\rmH_0 (\act_m(a))\to \rmH_0 (n)$ is an epimorphism in the abelian category $\CM^\heartsuit$.
\item \emph{$\CA$-compact} if $\inthom_\CA(m,-):\CM\to \CA$ preserves filtered colimits ( equivalently all small colimits or direct sums).
\end{itemize}
\end{defn}
\begin{defn}
We say a tensor category $\CC$ in $\PPr_c$ is \emph{compact-rigid} if all ($\VVect$-)compact objects in $\CC$ has dual.
\end{defn}
\begin{expl}
If $\CA$ is rigid, then $\widehat{\CA}$ is compact-rigid.
\end{expl}

\begin{defn}\label{generatingcategory}
Let $\CC$ be a stable $\infty$-category. A collection of objects $\{x_i\}$ is said to \emph{generate $\CC$} if $\Hom_\CC(x_i,y)\simeq 0$ for all $i$, implies that $y=0$, i.e. $\CC$ should be generated under colimits of $\{x_i\}$. 
\end{defn}

\begin{lem}\label{rightadjointdominantfunctorisconservative}
(\cite{GR19} lemma 5.4.3) Let $\CC,\CD$ be stable $\infty$-categories, and let $F:\CC\to \CD$ be a functor that admits a right adjoint. Then the essential image of $F$ generates $\CD$ (or say $F$ is dominant) if and only if its right adjoint $F^R$ is conservative.
\end{lem}

\begin{cor}\label{genrator=conservative}
Let $\CC$ be a tensor category in $\PPr_c$, and $\CM\in \PPr_c$ be a $\CC$-module category with t-structure. $m\in \CM$ is an $\CC$-generator if and only if $\inthom_\CC(m,-)$ is conservative.
\end{cor}
\begin{proof}
Here $F=\act_m:\CC\to \CM$, and $\inthom_\CC(m,-)$ is its right adjoint.
\end{proof}

\begin{lem}\label{FRmodule}
(\cite{BZN09} Lemma 3.5) Let $\CC$ be a compact-rigid tensor category in $\PPr_c$, let $\CM, \CN\in \PPr_c$ be $\CC$-module categories, and let $F:\CM\to \CN$ an $\CC$-module functor in $\PPr_c$ admitting a right adjoint $F^R$. Then $F^R$ admits a canonical $\CC$-module functor structure.
\end{lem}

Let $\CC$ be a monoidal DG category (i.e. a $E_1$-algebra in $\DGcat$) and $\CM,\CN$ right, respectively left $\CC$-module DG categories with module action $\odot$.
\begin{defn}
 We define the \emph{relative tensor product}, denoted by $\CM\boxtimes_\CC \CN$, to be the DG-category: 
\begin{itemize}
\item objects are pairs $(m,n)$ for $m\in \CM$ and $n\in \CN$;
\item the hom complex from $(m,n)$ to $(m',n')$ is defined by the derived coend
\[ \CM\boxtimes_\CC \CN((m,n),(m',n')):=\int^{a\in \CC}\Hom_\CM(m,m'\odot a)\otimes^\mathbb{L} \Hom_\CN(a\odot n,n')   \]
\end{itemize}
\end{defn}

\begin{lem}\label{coend=relativetensor}
 \[\CM\boxtimes_\CC \CN\simeq  \colim (\xymatrix{\cdots \CM\times \CC\times \CC\times \CN \ar@<-0.5ex>[r] \ar@<0.5ex>[r] \ar[r] & \CM\times \CC\times \CN \ar@<-0.5ex>[r] \ar@<0.5ex>[r]   &\CM\times \CN})\]
\end{lem}
\begin{proof}
This is well-known.
\end{proof}

Similarly, one can also use derived coend to define relative tensor product in $\PPr$ which we still denote it as $\boxtimes$.
\begin{lem}
There is an equivalence of stable presentable $\mathbb{k}$-linear $\infty$-categories.
\[ \widehat{\CM\boxtimes_\CC \CN}\simeq \widehat{\CM}\boxtimes_{\widehat{\CC}} \widehat{\CN}.  \]
\end{lem}
\begin{proof}
Free cocompletion is colimit-preserving.
\end{proof}

\begin{lem}\label{compactgeneratorsofrelativetensor}
(\cite{GR19} Cor 8.7.4) Let $\CC$ be a compact-rigid tensor category in $\PPr_c$. Let $\CM$ be a right $\CC$-module, and $\CN$ be a left $\CC$-module. If $\{W_i\}$ are all compact generators of $\CM$ and $\{V_i\}$ are all compact generators of $\CN$, then $\{W_i\boxtimes_{\CC} V_j\}$ are all compact generators of $\CM\boxtimes_\CC\CN$.
\end{lem}

\subsection{Barr-Beck theorem}
Let $\CC$ be an $(\infty,1)$-category. The functor category $\fun(\CC,\CC)$ has a natural structure of monoidal category, and $\CC$ that of $\fun(\CC,\CC)$-module. By definition, a monad acting on $\CC$ is an associative algebra $F\in \fun(\CC,\CC)$. Given a monad $F$, we can consider the category $\RMod_F(\CC)$. We denote by
\[ \forg:\RMod_F(\CC)\to \CC   \]
the forgetful functor. This forgetful functor admits a left adjoint, denoted
\[ \Ind_F:\CC\to \RMod_F(\CC)    \]
Then we have the composite functor
\begin{align*}
\forg\circ \Ind_F:\CC&\to \CC  \\    c&\mapsto F(c)
\end{align*}

For any $(\infty,1)$-category $\CD$, the $(\infty,1)$-category $\fun(\CD,\CC)$ is also  module over $\fun(\CC,\CC)$.

One can deduce from the construction that for a given $G\in \fun(\CD,\CC)$, a structure on $G$ of $F$-module, i.e. that of object in
\[ \RMod_F(\fun(\CD,\CC))   \]
is equivalent to that of factoring $G$ as
\[ \CD\to \RMod_F(\CC)\xrightarrow{\forg} \CC   \]

Let $G$ be a functor $\CD\to \CC$. It is easy to see that if $G$ admits a left adjoint, then the internal hom $\intend_{\fun(\CC,\CC)}(G,G)$ exists and identifies with $G\circ G^L$. Note that $G\circ G^L$ admits a structure of associative algebra.

By the above, the functor $G$ canonically factors as
\[ \CD\xrightarrow{\tilde{G}} \RMod_{G\circ G^L}(\CC)\xrightarrow{\forg} \CC   \]
\begin{defn}
Let $\CC,\CD$ be $\infty$-categories, and $G:\CD\to \CC$ be a right adjoint functor. We shall say that $G$ is \emph{monadic} if the above functor
\[ \tilde{G}:\CD\to \RMod_{G\circ G^L}(\CC)  \]
is an equivalence.
\end{defn}

We have Barr-Beck-Lurie monadicity theorem:
\begin{thm}\label{Barrbeck}
(\cite{DAG2} Theorem 3.4.5) Let $\CC,\CD$ be $\infty$-categories both contains geometric realizations, and $G:\CD\to \CC$ be a functor. Then the functor $G$ is monadic provided that the following two conditions hold:
\begin{itemize}
\item $G$ is conservative;
\item $G$ preserves geometric realizations.
\end{itemize}
\end{thm}

\begin{thm}
Let $\CC$ be a compact-rigid tensor category in $\PPr_c$, and let $\CM\in \PPr_c$ be a right $\CA$-module category. Let $m$ be an $\CC$-compact generator. Then there is an equivalence of right $\CC$-module categories
\[ \CM\simeq \LMod_{\intend_\CC(m)}(\CC)  \]
\end{thm}\label{Monadicityformodule}
\begin{proof}
Since $\CC$ is compact-rigid, by lemma \ref{FRmodule}, $\act^R_m$ has a canonical $\CC$-module structure. So we have $\act^R_m\circ \act_m=\inthom_\CC(m,\act_m(-))\simeq \intend_\CC(m)\otimes -$.

By assumption, we have $\act^R_m=\inthom_\CC(m,-):\CM\to \CC$ is conservative and colimit-preserving. Hence $\act^R_m$ is monadic. By Theorem \ref{Barrbeck}, we have $\CM\simeq \RMod_{\act^R_m\circ\act_m}(\CC)\simeq \RMod_{\intend_\CC(m)\otimes -}(\CC)\simeq\LMod_{\intend_\CC(m)}(\CC)$.
\end{proof}

\begin{thm}\label{Monadicityforrelativetensorproduct}(\cite{BZBJ18} Theorem 4.12)
 Let $\CC$ be a compact-rigid tensor category in $\PPr_c$, and let $\CM,\CN$ be right and left $\CC$-module categories with $\CC$-compact generator $m\in \CM$ and $n\in \CN$, respectively. Then we have equivalences
\[ \CM\boxtimes_\CC \CN\simeq \LMod_{\intend_\CC(m)}(\CN)      \]
\end{thm}

We give an stable $\infty$-version of \cite{DSPS19} Theorem 3.3(3).
\begin{prop}\label{DSS}
 Given a right module category $\CM$, and a left module category $\CN$ over a compact-rigid tensor category $\CC\in \PPr_c$, with $\CC$-compact generator $m,m'\in \CM$, and $n,n'\in \CN$, we have an isomorphism
\[    \Hom_{\CM\boxtimes_\CC \CN}(m\boxtimes_\CC n,m'\boxtimes_\CC n')\cong \Hom_\CC(\mathbb{1},\inthom(m,m')\otimes \inthom(n,n'))      \]
\end{prop}
\begin{proof}
By \ref{FRmodule}, we have internal hom functor is a $\CC$-module functor. So we have 
\begin{align*}
\Hom_\CC(\mathbb{1},\inthom(m,m')\otimes\inthom(n,n'))&\cong \Hom_\CC(\mathbb{1},\inthom(m,m'\otimes \inthom(n,n')))
\end{align*}
Then by definition of internal hom, we have
\[\Hom_\CC(\mathbb{1},\inthom(m,m'\otimes \inthom(n,n'))\cong \Hom_{\LMod_{\intend(m)}(\CC)}(m,m'\otimes \inthom(n,n'))\]
Again, since internal hom functor is a $\CC$-module functor, we have
\begin{align*}
\Hom_{\LMod_{\intend(m)}(\CC)}(m,m'\otimes \inthom(n,n'))&\cong\Hom_{\LMod_{\intend(m)}(\CC)}(m,\inthom(n,m'\otimes n')) \\
& \cong \Hom_{\LMod_{\intend(m)}(\RMod_{\intend(\CN)}(\CC))}(m\otimes n,m'\otimes n').
\end{align*}
By \ref{Monadicityforrelativetensorproduct}, we get the formula.
\end{proof}

\subsection{Derived categories}

Let $\CC$ be a small DG category. Denote by $\Ac(\CC)$ the full DG subcategory of $\Mod_\CC$ consisting of all acyclic DG modules. It is well-known that the homotopy category of DG modules $\rmH^0(\Mod_\CC)$ has a natural structure of a triangulated category and the homotopy category of acyclic complexes $\rmH^0(\Ac(\CC))$ forms a full triangulated subcategory in it.
\begin{defn}
The \emph{derived category} $D(\CC)$ of a small DG category $\CC$ is the Verdier quotient of $\rmH_0(\Mod_\CC)$ by the subcategory $\rmH_0(\Ac(\CC))$.
\end{defn}

\begin{lem}\label{DC=hatC} Let $\mathbb{k}$ be a field.
Let $\CC$ be a small $\mathbb{k}$-linear category viewed as a DG category. There are equivalences of stable $\mathbb{k}$-linear presentable $\infty$-categories
\[ \widehat{\CC}\simeq \CD(\CC)\simeq \CD(\widehat{\CC}^\heartsuit)   \]
Here $\CD(\CC)$ is the unique dg/$\infty$-enhancement of $D(\CC)$, and $\CD(\widehat{\CC}^\heartsuit)$ is the derived $\infty$-category of Grothendieck abelian category $\widehat{\CC}^\heartsuit$, defined in \cite{SAG18}. 
\end{lem}
\begin{proof}
Since $\CC$ is the full subcategory of $\widehat{\CC}$ and the $\ob(\CC)$ forms a set of compact generators of $\rmH^0(\widehat{\CC})$, by \cite{LO10} Proposition 1.17, we have an equivalence of triangluated categories
\[ D(\CC)\simeq \rmH_0(\widehat{\CC})   \]

By \cite{LO10} Proposition 2.6, $D(\CC)$ has a unique $\infty$-enhancement $\CD(\CC)$, so we have 
\[  \widehat{\CC}\simeq \CD(\CC) . \]

By \cite{LO10} proposition 1.15, we have $\CD(\CC)\simeq \CD(\widehat{\CC}^\heartsuit)$. 
\end{proof}
\begin{rem}
Since $\CD(\CC)$ has a canonical $t$-structure, the above lemma implies $\widehat{\CC}$ has a canonical $t$-structure. Therefore, one can keep track of $t$-structure, and upgrade our $\PPr$ to a refined category $\rm{Groth}^{cg}_\infty$ of compactly generated Grothendieck prestable $\infty$-categories and compact functors. 
\end{rem}

\begin{lem}\label{DLMod=LModhat}
Let $A$ be an algebra in $\widehat{\CC}^\heartsuit$. 
\[ \CD(\LMod_A(\widehat{\CC}^\heartsuit))\simeq \LMod_A(\widehat{\CC})    \]
\end{lem}
\begin{proof}
A complex of $A$-modules is an $A$-module object in complexes. Since forgetful functor $U:\LMod_A(\widehat{\CC}^\heartsuit)\to \widehat{\CC}^\heartsuit$ is exact (so its derived functor $\CD(\LMod_A(\widehat{\CC}^\heartsuit))\to \widehat{\CC}$ is same with $U$) and reflects isomorphisms, we can use Barr-Beck theorem to prove it, where the monad is $A\otimes -$.
\end{proof}

\section{Derived skein theory}\label{Derivedskeintheory}

\subsection{Derived skein module}\label{DSM}

We give an axiomatic definition for derived skein module.
\begin{defn}\label{axiomdefn}

Let $M$ be an oriented 3-manifold with corners, $\CA$ be a ribbon tensor (1-)category, and $X$ be a $\CA$-labeling of $\partial M$, i.e. $\CA$-labeled framed points of $\partial M$. The \emph{relative derived  $\CA$-skein module $\SSk_\CA(M,X)$} of $M$ is a non-negative chain complex over $\mathbb{k}$ such that
\begin{enumerate}
\item (Functorial property on $M$) Given an oriented embedding from $f:M\hookrightarrow  M'$ whose restriction is an oriented embedding $f|_{\partial M}:\partial M \hookrightarrow \partial M'$ between boundaries, there is a chain map as a part of data
\[  f_*:\SSk_\CA(M,X)\xrightarrow{} \SSk_\CA(M',f(X))    \]
such that if $f$ is a homeomorphism, $f_*$ is a chain homotopy equivalence.
\item (Disjoint union property)
There is a canonical homotopy equivalence as a part of data
\[  \SSk_\CA(M\sqcup M',X\cup X')\cong \SSk_\CA(M,X)\otimes \SSk_\CA(M',X')  \]
\item (Gluing property) If $X\cup X\cup X'$ is a $\CA$-labeling of  $\partial M=\Sigma\cup \Sigma^\op\cup \Sigma'$. Here the two copies of $\Sigma$ are disjoint from each other and $\partial (\Sigma\cup \Sigma^\op)=\partial \Sigma'$.  Gluing the two copies of $\Sigma$ together yields a new 3-manifold $M_{\gl}$ with boundary $\Sigma'_{gl}$, where $\Sigma'_{gl}$ denotes $\Sigma'$ glued to itself (without corners) along two copies of $\partial \Sigma$. There is a natural chain map as a part of data
\[ \gl:\bigoplus_{X} \SSk_\CA(M,X\cup X\cup X')\to \SSk_\CA(M_{\gl}, X' \footnote{\text{Here $Y$ is a $\CA$-labeling of $\Sigma_{gl}'$ which is not necessary equal to $\Sigma'$}.})   \]
such that if $\partial M=\Sigma\cup\Sigma^\op\cup\Sigma'\cup\Sigma'^\op\cup\Sigma''$, then the following diagram commutes
\[ \xymatrix{ & \SSk(M,X \cdots X'') \ar[r] \ar[d] & \bigoplus_X \SSk(M,X\cdots X'') \ar[r]^{\gl} &\SSk(M_{\gl},X'\cup X'\cup X'') \ar[d] \\ & \bigoplus_{X'}\SSk(M,X\cdots X'') \ar[d]_{\gl} & & \bigoplus_{X'}\SSk(M_{gl},X'\cup X'\cup X'') \ar[d]^{\gl} \\ & \SSk(M_{\gl'}, X\cup X\cup X'') \ar[r] & \bigoplus_{X} \SSk_\CA(M_{gl'}, X\cup X\cup X'') \ar[r]_-{\gl} &  \SSk((M_{\gl'})_{gl},X'')= \SSk((M_{\gl})_{\gl '}, X'')  }    \]

The sum is over all $\CA$-labelings $X$ of $\Sigma$ compatible at their 1-dimensional boundaries with $\Sigma'$. Here natural means natural with respect to property 1 functorial property on $M$.
\item (Skein property)
\[ \rmH_0(\SSk_\CA(M,X))\cong \Sk_\CA(M,X)    \]

\item (Excision property) Let $M=M_1\bigcup_\Sigma M_2$. Then there is a chain homotopy equivalence
\[ \SSk_\CA(M,X\cup Y)\cong \Tot(\BBar(\SSk_\CA(M_1,X\cup -),\SSk_\CA(\Sigma\times I,-\cup -),\SSk_\CA(M_2,-\cup Y)) )    \]
Here the right hand side is the total complex of the following simplicial complex:
\begin{itemize}
\item The n-th term:  
\[  \BBar_n=\bigoplus_{Z_0,\cdots Z_n} \SSk_\CA(M_1,X\cup Z_0)\otimes \SSk_\CA(\Sigma\times I,Z_0\cup Z_1)\otimes\cdots\otimes \SSk_\CA(\Sigma\times I,Z_{n-1}\cup Z_n)\otimes \SSk_\CA(M_2,Z_n\cup Y)   \]
an elementary tensor is written as $x\otimes f_1\otimes \cdots\otimes f_n\otimes y$
\item face maps 
\[ d_i:\BBar_n\to \BBar_{n-1}  \]
are as follows.

The first face uses the left action $\rho^L$:
\[ d_0(x\otimes f_1\otimes\cdots\otimes f_n\otimes y)=x\otimes f_1\otimes\cdots\otimes \rho^L(f_1\otimes y) \]
For $1\le i\le n-1$, the i-th face uses composition $\alpha$:
\[ d_i(x\otimes f_1\otimes\cdots\otimes f_{i}\otimes f_{i+1}\otimes \cdots\otimes f_n\otimes y)=x\otimes f_1\otimes\cdots \alpha(f_i\otimes  f_{i+1})\otimes\cdots\otimes f_n\otimes y   \]
the last face uses the right action $\rho^R$:
\[ d_n(x\otimes f_1\otimes\cdots\otimes f_n\otimes y)=\rho^R(x\otimes f_1)\otimes f_2\otimes\cdots f_n\otimes y   \]
\end{itemize}

with actions and compositions given by property 3 gluing maps as follows: For any $\CA$-labelings $Z,W$ of $\Sigma$
\begin{align*}
 \rho^R: \SSk_\CA(M_1,X\cup Z)\otimes \SSk_\CA(\Sigma\times I,Z\cup W)&\to \bigoplus_{Z} \SSk_\CA(M_1,X\cup Z)\otimes \SSk_\CA(\Sigma\times I, Z\cup W)\\
  &\cong \bigoplus_Z \SSk_\CA(M_1\sqcup \Sigma\times I,X\cup Z\cup Z\cup W) \\
  &\xrightarrow{\gl} \SSk_\CA(M_1, X\cup W)
\end{align*}
similarly we have left action $\rho^L$.
and 
\begin{align*}
\alpha:\SSk_\CA(\Sigma\times I,X\cup Z)\otimes \SSk_\CA(\Sigma\times I,Z\cup W)&\to \bigoplus_{Z} \SSk_\CA(\Sigma\times I,X\cup Z)\otimes \SSk_\CA(\Sigma\times I, Z\cup W)\\
  &\cong \bigoplus_Z \SSk_\CA(\Sigma\times I\sqcup \Sigma\times I,X\cup Z\cup Z\cup W) \\
  &\xrightarrow{\gl} \SSk_\CA(\Sigma\times I, X\cup W)
\end{align*}

\item (Unit property)
Consider a labeling $X\cup Y$ with disks $X=\{(x_1,V_1),\cdots,(x_n, V_n)\}$ embedded in $\mathbb{D}\times \{0\}$ and $Y=\{(y_1,W_1),\cdots,(y_m,W_m)\}$ embedded in $\mathbb{D}\times \{1\}$. Here $V_1,\cdots V_n,W_1,\cdots W_n$ are objects of $\CA$.
\[ \SSk_\CA(\mathbb{D}\times I,X\cup Y)\cong \Hom_\CA(V_1\otimes \cdots\otimes V_n,W_1\otimes\cdots\otimes W_m)\cong \rmH_0(\SSk_\CA(\mathbb{D}\times I,X\cup Y))    \]
\end{enumerate}
When $\CA=\rep^\fd_q(G)$ for not a root of unity $q$, we denote the \emph{relative derived $G$-skein module} as $\SSk_G(M,X)$. If $X=\emptyset$, we denote the \emph{derived skein module} by $\SSk_\CA(M)$.
\end{defn}
\begin{rem}\label{homotopycoherent}
We want to emphasize again that we do not consider more difficult problems like homotopy coherence, which are important for full functoriality. We expect Ayala-Francis's upcoming paper on factorization homology of rigid braided $(\infty,1)$-category can give us an answer, and we focus in this paper on a framework for doing computations.
\end{rem}

\begin{rem}\label{Ainftygluing}
Due to the commutative diagram in gluing property, our axioms above is a "DG"-type axioms rather than an $A_\infty$-type axioms. One can consider an $A_\infty$-type axioms. For example the gluing property will become like:
($A_\infty$-Gluing property) If $X_1\cup X_1\cup X_2\cup X_2\cdots X_{n-1}\cup X_{n-1}\cup X'$ is a $\CA$-labeling of  $\partial M=\Sigma_1\cup \Sigma_1^\op\cup\Sigma_2\cup \Sigma_2^\op\cup\cdots\cup \Sigma'$. Gluing the two identical ones together one by one yields an new 3-manifold $M_{\gl}$ with boundary $\Sigma'_{gl}$, where $\Sigma'_{gl}$ denotes $\Sigma'$ glued to itself (without corners) along two copies of $\partial \Sigma$. For each integer $n\ge 2$, There is a natural chain map of degree $n-2$
\[ \gl^n:\bigoplus_{i=1}^{n-1}\bigoplus_{X_i} \SSk_\CA(M,X_1\cup X_1\cup\cdots X_i\cup X_i\cdots\cup X')\to \SSk_\CA(M_{\gl}, X' \footnote{\text{Here $Y$ is a $\CA$-labeling of $\Sigma_{gl}'$ which is not necessary equal to $\Sigma'$}.})   \]
satisfies Stasheff identities. Here natural means natural with respect to property 1 functorial property on $M$. Also see Remark \ref{BlobcomplexAinfty}.
\end{rem}

\begin{rem}
One can try to formulate the above axioms as a symmetric monoidal $\infty$-functor with target $\VVect$. One key issue is how to define morphisms and compositions in the source category. 
\end{rem}
\begin{rem}\label{uniqueness}
We don't carefully prove that the above axiom gives a \emph{unique}, up to isomorphism, chain complex. But we believe it is the case. The reason is the following: every 3-manifold admits a handlebody decomposition. Each handle is homeomorphic to a 3-ball, so unit property tells us what we attach to a handle, and excision property tells us how to glue them together. So finally we should get a unique, up to isomorphism, answer. In other words, we believe the above axioms give the universal property of the derived skein module. Another reason is that our Theorem \ref{computableexcisionforrepqG} is actually a structural theorem for derived skein modules, which we obtained from those axioms.
\end{rem}

The existence of such a \footnote{See Remark \ref{Ainftygluing} and \ref{BlobcomplexAinfty}.} derived skein module is given by Morrison-Walker's Blob complex \cite{MW12} by Property 1.3.2, Theorem 7.2.1 and Proposition 3.2.1. We recall their definition in the following and focus on the 3-dimensional case.

\begin{rem}\label{tanglehypo}
We note that, while Blob complex satisfies our hypotheses, it is only a folklore expectation that a ribbon tensor category such as $\rep^{\fd}_q(G)$ defines a system of coefficients in Blob homology. 

For this reason, the computations in this paper are valid subject to either establishing this expectation, or by an alternative formulation such as the upcoming work on enriched $\beta$-factorization homology of Ayala-Francis.
\end{rem}
\begin{rem}\label{modelindependent}
We wish to emphasise that, while we draw inspiration from the construction of the Blob complex, no theorems in this paper rely on its technical details. Actually in derived case, it would be better to work in axiomatic algebraic way rather than with concrete huge topological model.

The enriched $\beta$-factorization homology, as a developing construction that aims to have the same target category and the same inputs as the skein TQFT, by the cobordism hypothesis \cite{Lur08}, should also give us an instance of our axioms. For framed 3-manifolds, they have a forthcoming model of a derived skein module using enriched $\beta$-factorization homology. Since we are working in a model-independent way, our results should also work for their (oriented variant) construction.
\end{rem}

Now let's say something on low dimensions first. Let $X$ be a $\CA$-labeling of $\partial M$. For $\SSk^0_\CA(M,X)$ is the vector space of all $\CA$-labeled ribbon graphs compatible with $X$. Consider a 3-ball $\mathbb{B}\cong \mathbb{D}\times I\subset M$, and consider a labeling $Y=Y_0\cup Y_1$ with disks $Y_0=\{(x_1,V_1),\cdots, (x_n,V_n)\}$ embedded in $\mathbb{D}\times \{0\}$ and $Y_1=\{(y_1,W_1),\cdots, (y_m, W_m) \}$ embedded in $\mathbb{D}\times \{1\}$. We have a well-defined surjection
\[ \phi_{B,Y}:\SSk^0_\CA(\mathbb{B},Y)\to \Hom_\CA(V_1\otimes \cdots\otimes V_n,W_1\otimes\cdots\otimes W_n)    \]
We call an element in $\ker \phi_{\mathbb{B},Y}$ an \emph{$\CA$-skein relation}.

We say a 1-Blob diagram consists of
\begin{itemize}
\item a closed ball $\mathbb{B}\subset X$;
\item a $\CA$-labeling $Y$ of $\partial \mathbb{B}$;
\item a ribbon graph $r$ relative to $Y$;
\item a skein relation $s\in \ker \phi_{\mathbb{B},Y}$.
\end{itemize}

Then define
\[\SSk^1_\CA(M,X):=\bigoplus_{\mathbb{B}}\bigoplus_Y \ker \phi_{\mathbb{B},Y}\otimes \SSk^0_\CA(M-\mathbb{B},Y)\]
The first direct sum is indexed by all closed balls $\mathbb{B}\subset M$, and the second by all $\CA$-labelings of $\partial \mathbb{B}$. Note that $\SSk^1_\CA(M)$ is spanned by $(\mathbb{B},s,r)$. Here $s$ is a skein relation in $\mathbb{B}$, and $r$ is a $\CA$-labeled ribbon graph in $M-\mathbb{B}$ relative to $X$.

Define the boundary map 
\begin{align*}
d:\SSk^1_\CA(M,X)&\to \SSk^0_\CA(M,X)   \\
(\mathbb{B},s,r)&\mapsto s* r
\end{align*}
where $s*r$ denotes the $\CA$-labeled ribbon graph obtained by gluing $s$ to $r$. 

A 2-Blob diagram comes in one of two types, disjoint and nested. A disjoint 2-Blob diagram consists of 
\begin{itemize}
\item a pair of closed balls $\mathbb{B}_1,\mathbb{B}_2\subset M$ with disjoint interiors;
\item a $\CA$-labeled ribbon graph $r\in \SSk^0_\CA(M-\mathbb{B}_1-\mathbb{B}_2,Y_1\cup Y_2)$; Here $Y_i$ is a $\CA$-labeling of $\partial \mathbb{B}_i$. 
\item two skein relations $s_i\in \ker\phi_{\mathbb{B}_i,Y_i}$.
\end{itemize}
We also identify $(\mathbb{B}_1,\mathbb{B}_2,s_1,s_2,r)$ with $-(\mathbb{B}_2,\mathbb{B}_1,s_2,s_1,r)$. Define
\[ d(\mathbb{B}_1,\mathbb{B}_2,s_1,s_2,r)=(\mathbb{B}_2,s_2,s_1* r)-(\mathbb{B}_1,s_1,s_2*r) \in \SSk^1_\CA(M,X)      \]
It's easy to check that $d^2=0$.

A nested 2-Blob diagram consists of 
\begin{itemize}
\item a pair of nested balls $\mathbb{B}_1\subseteq \mathbb{B}_2\subseteq M$;
\item a $\CA$-labeled ribbon graph $r'\in \SSk^0_\CA(\mathbb{B}_2-\mathbb{B}_1,Y_1\cup Y_2)$;
\item a $\CA$-labeled ribbon graph $r\in \SSk^0_\CA(M-\mathbb{B}_2,Y_2)$;
\item a skein relation $s\in \ker \phi_{\mathbb{B}_1,Y_1}$.
\end{itemize}

Define $\partial(\mathbb{B}_1,\mathbb{B}_2,s,r',r)=(\mathbb{B}_2,s* r',r)-(\mathbb{B}_1,s,r'* r)$. It is again easy to check that $d^2=0$. Now we have 
\begin{align*} \SSk^2_\CA(M,X):=\bigoplus_{\mathbb{B}_1\cap \mathbb{B}_2=\emptyset}\bigoplus_{Y_1,Y_2} \ker\phi_{\mathbb{B}_1,Y_1}\otimes \ker\phi_{\mathbb{B}_2,Y_2}\otimes \SSk^0_\CA(M-\mathbb{B}_1-\mathbb{B}_2,Y_1\cup Y_2)  \\  
\oplus\bigoplus_{\mathbb{B}_1\subset \mathbb{B}_2}\bigoplus_{Y_1,Y_2}\ker\phi_{\mathbb{B}_1,Y_1}\otimes \SSk^0_\CA(\mathbb{B}_2-\mathbb{B}_1,Y_1\cup Y_2)\otimes \SSk^0_\CA(M-\mathbb{B}_2,Y_2).
\end{align*}

\begin{defn}
Let $M$ be an oriented 3-manifold. A \emph{configuration of $k$ Blobs} in $M$ is an ordered collection of $k$ subsets $\{\mathbb{B}_1,\cdots, \mathbb{B}_k\}$ of $M$ such that there exists a relative handlebody decomposition $M_{-1}=\partial M\times I \subset M_0\subset \cdots\subset M$ with the property that for each subset $\mathbb{B}_i$ there is some $0\le l\le m$ and some connected component $M_l'$ of $M_l$ which is a 3-ball, such that $\mathbb{B}_i$ is the image of $M'_l$ in $M$. We say that such a relative handlebody decomposition is  compatible with the configuration. A Blob $\mathbb{B}_i$ is a \emph{twig Blob} if no other Blob $\mathbb{B}_j$ is a strict subset of it.
\end{defn}

The above definition implies that for any two Blobs in a configuration of Blobs in $M$, they either have disjoint interiors, or one Blob is contained in the other. We describe these as disjoint Blobs and nested Blobs.

\begin{defn}
Let $\CA$ be a ribbon tensor category, $M$ an oriented 3-manifold and $X$ is a $\CA$-labeling on $\partial M$. A \emph{k-Blob diagram} on $M$ consists of 
\begin{itemize}
\item a configuration $\{\mathbb{B}_1,\cdots, \mathbb{B}_k\}$ of $k$ Blobs in $M$;
\item a $\CA$-ribbon graph compatible with $X$, which is splittable along some gluing decomposition compatible with that configuration,
\end{itemize}
such that the restriction $u_i$ of $r$ to each twig Blob $\mathbb{B}_i$ is a skein relation in $\mathbb{B}_i$. 
\end{defn}

\begin{defn}
Let $M$ be an oriented 3-manifold, and $\CA$ be a ribbon tensor category. The $k$-th vector space $\SSk^k_\CA(M,X)$ of the derived skein module of $M$ is the direct sum over all configurations of $k$ Blobs in $X$, modulo identifying the vector spaces for configurations that only differ by a permutation of the Blobs by the sign of that permutation. The differential $\SSk^k_\CA(M,X)\to \SSk^{k-1}_\CA(M,X)$ is, as above, the signed sum of ways of forgetting one Blob from the configuration, preserving the $\CA$-labeled ribbon graph $r$:
\[ \partial(\{\mathbb{B}_1,\cdots,\mathbb{B}_k\},r)=\sum\limits^k_{i=1}(-1)^{i+1}(\{\mathbb{B}_1,\cdots,\hat{\mathbb{B}_i},\cdots,\mathbb{B}_k\},r)   \]
\end{defn}

\begin{rem}\label{higherDGinput}
One would be want the input category has derived information, e.g a DG category or $A_\infty$-category. Then one needs to be careful on unit property. For ribbon tensor DG category $\CA$ (There are two possible ways to define a ribbon tensor DG category. The first is to define it as homotopy $SO(3)$-fixed point of 3-dualizable objects in $\Alg_2(\BBim)$ such that its homotopy category is usual ribbon tensor category. One can also define it as a symmetric monoidal functor $\disk^{or}_2\to \BBim$ with additional rigid properties.), the unit property would become $\SSk_\CA(\mathbb{B},X\cup Y)\cong \Hom_\CA(V_1\otimes\cdots\otimes V_n,W_1\otimes\cdots\otimes W_m)$ where right hand side is Hom complex of $\CA$. 

One can also consider input category as a higher category, like a "ribbon" $3$-category. A ribbon tensor category can be then viewed as a ribbon 3-category with only one object and only one morphism.

Due to the long-term view of Morrison-Walker, both  cases have already been considered in Blob complex \cite{MW12} in the setting of disk-like $A_\infty$ n-category. But disk-like $A_\infty$ n-category is not a standard mathematical notion. Just like what we did above, one can modify it into a standard $A_\infty$ n-category, or enriched $(\infty,n)$-category setting. The same ideas also appeared in Ayala-Francis-Rozenblyum's works \cite{AF17}\cite{AFR18} and Fuchs-Schweigert-Yang's work \cite{FSY23}.
\end{rem}

\subsection{Derived skein categories}

\begin{defn}\label{derivedskeincat}
Let $\Sigma$ be an oriented surface. Let $\CA$ be a ribbon tensor category. The \emph{derived skein category} $\SSkcat_\CA(\Sigma)$  of $\Sigma$ is the following DG category:
\begin{itemize}
\item Objects are $\CA$-labelings of $\Sigma$.
\item Hom complex from $X$ to $Y$ is given by $\SSk_\CA(\Sigma\times[0,1],X,Y)$.
\item Compositions are given by gluing map of cylinders based on the axioms of relative derived skein module.
\item Strict associativity is given by the commutative diagram in Gluing property in where $M$ is disjoint union of 3 copies of $\Sigma\times I$.
\end{itemize}

Sometimes, our $\CA$ is a ribbon tensor category in $\PPr_c$, then the notation $\SSkcat_\CA(\Sigma)$ actually means $\SSkcat_{\CA^c}(\Sigma)$. Here $\CA^c$ is the full subcategory of compact objects of $\CA$.
\end{defn}

\begin{defn}
Let $M:\Sigma'\to \Sigma$ be an oriented 3-cobordism. The \emph{derived skein bimodule} of $M$ is the following DG functor
\begin{align*}
\underline{\SSk}_\CA(M):\SSkcat_\CA(\Sigma)\boxtimes\SSkcat_\CA(\Sigma')^\op&\to \VVect\\
(X,Y)&\mapsto \SSk_\CA(M,X\cup Y)
\end{align*}
\end{defn}

Due to Remark \ref{geometricrelation=doublecomplex}, we can from now denote the excision property in Definition \ref{axiomdefn} of derived skein module as follows for convenience
\[  \SSk_\CA(M,X\cup Y)\cong \int^{Z\in \SSkcat_\CA(\Sigma)} \SSk_\CA(M_1,X\cup Z)\otimes^\mathbb{L}\SSk_\CA(M_2,Z\cup Y):=\underline{\SSk}_\CA(M_1)\otimes^\mathbb{L}_{\SSkcat_\CA(\Sigma)}\underline{\SSk}_\CA(M_2).     \]

\begin{rem}\label{BlobcomplexAinfty}
Actually if we use the version of Blob complex in Section \ref{DSM} as a model, what we get on derived skein category should be an $A_\infty$-category rather than a DG category. But it's not a big issue since every $A_\infty$-category is $A_\infty$-equivalent to a DG category; The usual categories they form are equivalent; the Morita categories they form are equivalent; see Appendix \ref{AinftyandDG} for more comparisons. This implies that there should exists another variant for Blob complex, which give us a honest DG category. Choose DG categories as our definition can help us avoid unnecessary homotopy issues, and also DG category is well-developed. In any case, this will not affect the main results of this paper, see Remark \ref{modelindependent}.
\end{rem}

\begin{rem}
In \cite{HRW24}, they use handle attachments and categorical bottom-to-top approach, which follows the classical spirit of Walker \cite{Wal06}\cite{Walker22} and Lurie \cite{Lur08}, to construct a DG category for a punctured surface with a specific monoidal 2-category input. This may provide a potential example for derived skein category of punctured surface with a monoidal 2-category input. Also see Remark \ref{higherDGinput}.
\end{rem}

\begin{lem}
$\SSkcat_\CA(-):\Mfld_2^{or}\to \DGcat$ is a symmetric monoidal functor.
\end{lem}
\begin{proof}
Given an oriented embedding $f|_{\Sigma}:\Sigma\hookrightarrow \Sigma'$, we have an oriented embedding from $f:\Sigma\times I\to \Sigma'\times I$. By Functorial property on $M$ in \ref{axiomdefn}, we have a map $f_*:\SSk_\CA(\Sigma\times I,X)\to \SSk_\CA(\Sigma'\times I, f(X))$.

Then we can define the following DG functor
\begin{align*}
(f|_{\Sigma})_*:\SSkcat_\CA(\Sigma)&\to \SSkcat_\CA(\Sigma')\\
X& \mapsto \iota\circ X\\
a &\mapsto f_*(a)
\end{align*}

The disjoint union property make it a symmetric monoidal functor.
\end{proof}

The excision of ordinary skein categories was proven by Cooke \cite{Coo19} and Brown-Haioun \cite{BH24}. This still holds for derived skein category.
\begin{prop}\label{excisionlskcat}
Let $\Sigma=\Sigma_1\bigcup_{P\times I} \Sigma_2$. There is an equivalence of DG categories
\[ \SSkcat_\CA(\Sigma)\simeq \SSkcat_\CA(\Sigma_1)\boxtimes_{\SSkcat_\CA(P\times I)}\SSkcat_\CA(\Sigma_2)    \]
\end{prop}
\begin{proof}
Essentially surjective on objects: the proof is same as ordinary case proved in \cite{Coo19}.

For the fully faithful part: A key observation is that $\Sigma\times I=(\Sigma_1\times I)\cup_{P\times \mathbb{D}}(\Sigma_2\times I)$. Then combining this with the excision property of derived skein module and \ref{coend=relativetensor}, we get the proof.
\end{proof}

\begin{prop}
We have an equivalence of DG categories
\[ \SSkcat_\CA(\mathbb{D})\simeq \CA     \]
\end{prop}
\begin{proof}
It's due to unit property of derived skein module.
\end{proof}
\begin{cor}
\[ \SSkcat_\CA(-)\simeq \int^{\DGcat}_{(-)}\CA    \]
\end{cor}
\begin{proof}
Factorization homology is uniquely characterized by the unit property and excision property \cite{AF15}.
\end{proof}

\begin{lem}\label{ExcisionFreecoofskeincat}
Let $\Sigma=\Sigma_1\bigcup_{P\times I}\Sigma_2$. There is an equivalence of stable presentable $\mathbb{k}$-linear $\infty$-categories
\[ \widehat{\SSkcat_\CA(\Sigma)}\simeq \widehat{\SSkcat_\CA(\Sigma_1)}\boxtimes_{\widehat{\SSkcat_\CA(P\times I)}} \widehat{\SSkcat_\CA(\Sigma_2)}  \]
\end{lem}
\begin{proof}
Free cocompletion preserves colimit.
\end{proof}
\begin{cor}
\[ \widehat{\SSkcat_\CA(-)}\simeq \int^{\PPr}_{(-)}\widehat{\CA}   \]
\end{cor}
\begin{rem}
A more subtle choice for the target category is $\BBim$ or $\PPr^\circ$. In \cite{AF20} Prop 3.9, factorization homology was proved to exist when target category is a $\otimes$-sifted cocomplete symmetric monoidal $\infty$-category. We don't know if $\BBim$ and $\PPr^\circ$ admit all (sifted) colimits. The linear case has already been proved by Brown-Haioun \cite{BH24} that factorization homology exists in those targets.
\end{rem}

\begin{rem}\label{skeinTQFT}
There is a $(3,2,1)$-TQFT, which we called \emph{derived skein $(3,2,1)$-TQFT}, given by the following:
\begin{align*}
\Bord^{or}_{3,2,1}&\to \Alg_1(\PPr)\\
C&\mapsto \widehat{\SSkcat_\CA(C\times I)}\\
\Sigma&\mapsto \widehat{\SSkcat_\CA(\Sigma)}\\
M&\mapsto \underline{\SSk}_\CA(M)
\end{align*}

Note that giving a $(2,1,0)$-TQFT and $(3,2)$-TQFT where they coincides on dimension 2, does not rigorously give a $(3,2,1,0)$-TQFT. That's why here we don't assign $\widehat{\CA}$ to a point.
\end{rem}

\subsection{Skeins of punctured surfaces}

Let $\Sigma^\circ$ be the punctured surface of $\Sigma$. Let $\act_\emptyset:\CA\to \SSkcat_\CA(\Sigma)$ be the DG functor given by the action of $\CA$ on $\emptyset\in \SSkcat_\CA(\Sigma)$. Applying the symmetric monoidal functor $\widehat{(-)}:\DGcat\to \PPr$, we have $\widehat{\act_\emptyset}:\widehat{\CA}\to \widehat{\SSkcat_\CA(\Sigma^\circ)}$. In particular, $i_*$ in \ref{i*} is a special case of this functor.  By the adjoint functor theorem, the right adjoint $\widehat{\act_\emptyset}^R:\widehat{\SSkcat_\CA(\Sigma)}\to \widehat{\CA}$ exists. But note that in general $\widehat{\act_\emptyset}$ is not a tensor functor.

\begin{lem}\label{emptyasgenerator}
$\emptyset=\mathbb{1}\in \widehat{\SSkcat_\CA(\Sigma^\circ)}$ is a $\widehat{\CA}$-compact generator.
\end{lem}
\begin{proof}
We want to prove that $\emptyset$ is a $\widehat{\CA}$-generator, i.e. any object in $\widehat{\SSkcat_\CA(\Sigma^\circ)}$can be obtained from the action of some $\CA$-labeling of $\mathbb{D}$ on $\emptyset$:
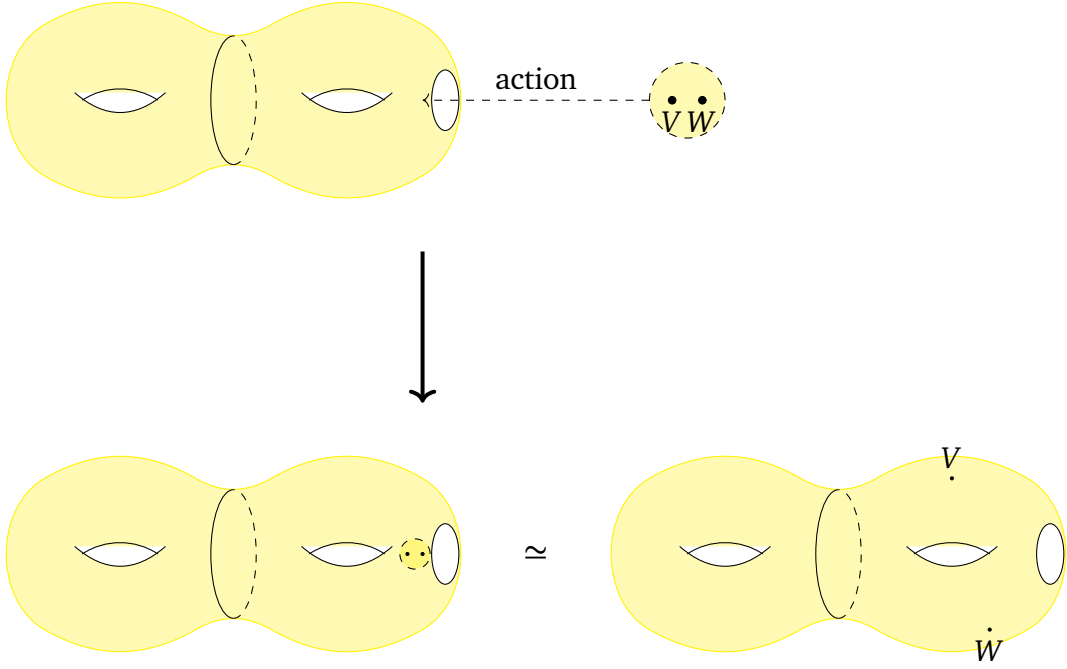
\begin{center}
\begin{tikzpicture}
\draw[smooth,color=yellow, fill opacity=0.3,fill] (0,1) to[out=30,in=150] (2,1) to[out=-30,in=210] (3,1) to[out=30,in=150] (5,1) to[out=-30,in=30] (5,-1) to[out=210,in=-30] (3,-1) to[out=150,in=30] (2,-1) to[out=210,in=-30] (0,-1) to[out=150,in=-150] (0,1);
\draw[smooth,fill=white] (0.4,0.1) .. controls (0.8,-0.25) and (1.2,-0.25) .. (1.6,0.1);
\draw[smooth] (0.5,0) .. controls (0.8,0.2) and (1.2,0.2) .. (1.5,0);
\draw[smooth,fill=white] (3.4,0.1) .. controls (3.8,-0.25) and (4.2,-0.25) .. (4.6,0.1);
\draw[smooth] (3.5,0) .. controls (3.8,0.2) and (4.2,0.2) .. (4.5,0);

\draw (2.5,-0.85) arc(270:90:0.3 and 0.85);
\draw[dashed] (2.5,-0.85) arc(270:450:0.3 and 0.85);

\draw[fill=white] (5.3,0) ellipse (0.18 and 0.4);

\draw[fill=yellow,fill opacity=0.3,dashed] (8.5,0) circle [radius=0.5];

\draw[fill] (8.3,0) circle [radius=0.05];
\draw[fill] (8.7,0) circle [radius=0.05];
\node[below] at (8.3,0) {$V$};
\node[below] at (8.7,0) {$W$};

\draw[dashed,->] (8,0) -- (5,0);
\node[above] at (6.5,0) {action};

\draw[line width=1.5,->] (5,-2) -- (5,-4);

\draw[smooth,color=yellow,fill, fill opacity=0.3] (0,-5) to[out=30,in=150] (2,-5) to[out=-30,in=210] (3,-5) to[out=30,in=150] (5,-5) to[out=-30,in=30] (5,-7) to[out=210,in=-30] (3,-7) to[out=150,in=30] (2,-7) to[out=210,in=-30] (0,-7) to[out=150,in=-150] (0,-5);
\draw[smooth,fill=white] (0.4,-5.9) .. controls (0.8,-6.25) and (1.2,-6.25) .. (1.6,-5.9);
\draw[smooth] (0.5,-6) .. controls (0.8,-5.8) and (1.2,-5.8) .. (1.5,-6);
\draw[smooth,fill=white] (3.4,-5.9) .. controls (3.8,-6.25) and (4.2,-6.25) .. (4.6,-5.9);
\draw[smooth] (3.5,-6) .. controls (3.8,-5.8) and (4.2,-5.8) .. (4.5,-6);

\draw (2.5,-6.85) arc(270:90:0.3 and 0.85);
\draw[dashed] (2.5,-6.85) arc(270:450:0.3 and 0.85);

\draw[fill=white] (5.3,-6) ellipse (0.18 and 0.4);

\draw[fill=yellow,fill opacity=0.3,dashed] (4.9,-6) circle [radius=0.2];
\draw[fill] (4.8,-6) circle [radius=0.02];
\draw[fill] (5,-6) circle [radius=0.02];

\draw[smooth,color=yellow,fill opacity=0.3,fill] (8,-5) to[out=30,in=150] (10,-5) to[out=-30,in=210] (11,-5) to[out=30,in=150] (13,-5) to[out=-30,in=30] (13,-7) to[out=210,in=-30] (11,-7) to[out=150,in=30] (10,-7) to[out=210,in=-30] (8,-7) to[out=150,in=-150] (8,-5);
\draw[smooth,fill=white] (8.4,-5.9) .. controls (8.8,-6.25) and (9.2,-6.25) .. (9.6,-5.9);
\draw[smooth] (8.5,-6) .. controls (8.8,-5.8) and (9.2,-5.8) .. (9.5,-6);
\draw[smooth,fill=white] (11.4,-5.9) .. controls (11.8,-6.25) and (12.2,-6.25) .. (12.6,-5.9);
\draw[smooth] (11.5,-6) .. controls (11.8,-5.8) and (12.2,-5.8) .. (12.5,-6);

\draw (10.5,-6.85) arc(270:90:0.3 and 0.85);
\draw[dashed] (10.5,-6.85) arc(270:450:0.3 and 0.85);

\draw[fill=white] (13.3,-6) ellipse (0.18 and 0.4);

\draw[fill] (12,-5) circle [radius=0.02];

\node[above] at (12,-5) {$V$};

\draw[fill] (12.5,-7) circle [radius=0.02];

\node[below] at (12.5,-7) {$W$};

\node[] at (6.5,-6) {$\simeq$};

\end{tikzpicture}
\end{center}

Now we prove $\emptyset \in \widehat{\SSkcat_\CA(\Sigma^\circ)}$ is a $\widehat{\CA}$-compact object. Since $\widehat{\CA}$ is compactly generated and $\widehat{\act_{\emptyset}}:\widehat{\CA}\times \widehat{\SSkcat_\CA(\Sigma^\circ)}\to \widehat{\SSkcat_\CA(\Sigma^\circ)}$ is a compact functor, every compact object in $\widehat{\SSkcat_\CA(\Sigma^\circ)}$ is a $\widehat{\CA}$-compact object. By Theorem \ref{Bim=Prcirc}, we have $\emptyset\in \SSkcat_\CA(\Sigma^\circ)\subset \widehat{\SSkcat_\CA(\Sigma^\circ)}$ is a compact object in $\widehat{\SSkcat_\CA(\Sigma^\circ)}$. Thus $\emptyset$ is a $\widehat{\CA}$-compact object.
\end{proof}

\begin{defn}
The \emph{internal derived skein algebra} is the DG functor
\begin{align*}
\SSkalg^\Int_\CA(\Sigma^\circ):\CA^\op&\to \VVect\\
V&\mapsto \Hom_{\SSkcat_\CA(\Sigma^\circ)}(\act_\emptyset(V),\emptyset)
\end{align*}
\end{defn}

\begin{lem}\label{intskalg=RL}
$\SSkalg^\Int_\CA(\Sigma^\circ)\cong \widehat{\act_\emptyset}^R\circ\widehat{\act_\emptyset}(\mathbb{1})$. 
\end{lem}
\begin{proof}
This is by definition. Note that $\widehat{\act_\emptyset}^R=\inthom(\emptyset,-)$.
\end{proof}

\begin{thm}\label{hatSkcat=LModhat}
Let $\CA$ be a ribbon tensor category, $\Sigma$ be an oriented surface, and $\Sigma^\circ$ its punctured surface. There is an equivalence 
\[ \widehat{\SSkcat_\CA(\Sigma^\circ)}\simeq \LMod_{\SSkalg^{\Int}_\CA(\Sigma^\circ)}(\widehat{\CA})    \]
\end{thm}
\begin{proof}
Following the theorem \ref{Monadicityformodule}, Lemma \ref{emptyasgenerator}, Lemma \ref{intskalg=RL}. 
\end{proof}

\begin{lem}\label{skcatannisrigid}
$\widehat{\SSkcat_\CA(\Ann)}$ is a compact-rigid tensor category in $\PPr_c$.
\end{lem}
\begin{proof}
Topologically, the natural embedding of the stack $\Ann\sqcup \Ann\hookrightarrow \Ann$ gives a the monoidal structure \[\widehat{\SSkcat_\CA(\Ann)}\boxtimes \widehat{\SSkcat_\CA(\Ann)}\to \widehat{\SSkcat_\CA(\Ann)}\]
by functoriality of $\SSkcat_\CA(-)$. The tensor unit is given by empty skein.

Algebraically, the tensor product of $\widehat{\SSkcat_\CA(\Ann)}\simeq \LMod_{\SSkalg^\Int_\CA(\Ann)}(\widehat{\CA})$ is given by $M,N\mapsto \tau_{lr}(M)\otimes^{\mathbb{L}}_{\SSkalg^\Int_\CA(\Ann)} N$ where  $\tau_{lr}:\LMod_{\SSkalg^\Int_\CA(\Ann)}(\widehat{\CA})\xrightarrow{\sim}\RMod_{\SSkalg^\Int_\CA(\Ann)}(\widehat{\CA})$ is the identification provided by field goal transform. Here we take derived tensor product is due to \ref{DC=hatC}. And tensor unit is given by $\SSkalg^\Int_\CA(\Ann)$.

$\SSkcat_\CA(\Ann)$ is rigid, so $\widehat{\SSkcat_\CA(\Ann)}$ is compact-rigid.
\end{proof}

Thanks to $\widehat{\SSkcat_\CA(\Ann)}$ is compact-rigid, this helps us to simplify the study of $\widehat{\SSkcat_\CA(\Sigma)}$.
\begin{lem}\label{1boxtimesVarecompactgenerators}
$\{\mathbb{1}_{\Sigma^\circ}\boxtimes_{\widehat{\SSkcat_\CA(\Ann)}}V\}_{V\in \widehat{\CA}^c}$ are all compact generators of $\widehat{\SSkcat_\CA(\Sigma)}$. Here $\mathbb{1}_{\Sigma^\circ}$ is the distinguished object in $\widehat{\SSkcat_\CA(\Sigma^\circ)}$.
\end{lem}
\begin{proof}
By \ref{compactgeneratorsofrelativetensor}, \ref{skcatannisrigid} and \ref{hatSkcat=LModhat}, we know 
\[ \{(\mathbb{1}_{\Sigma^\circ}\otimes W)\boxtimes_{\widehat{\SSkcat_\CA(\Ann)}}V\}=\{\mathbb{1}_{\Sigma^\circ}\boxtimes_{\widehat{\SSkcat_\CA(\Ann)}}W\otimes V\}     \]
are all compact generators of $\widehat{\SSkcat_\CA(\Sigma)}$.
\end{proof}

\begin{lem}\label{imageofactgenerates}
The essential image of $\widehat{\act_\emptyset}:\widehat{\CA}\to \widehat{\SSkcat_\CA(\Sigma)}$ generates $\widehat{\SSkcat_\CA(\Sigma)}$.
\end{lem}
\begin{proof}
This is a direct result from \ref{1boxtimesVarecompactgenerators}.
\end{proof}

\begin{prop}\label{semisimpleintskalg}
If $\CA$ is semisimple, we have 
\[ \SSkalg^\Int_\CA(\Sigma^\circ)\cong \Skalg^\Int_\CA(\Sigma^\circ)    \]
\end{prop}
\begin{proof}
$\widehat{\act_\emptyset}$ (resp $\widehat{\act_\emptyset}^R$) is (left) derived functor of $\widehat{\act_\emptyset}^\heartsuit$ (resp $(\widehat{\act_\emptyset}^\heartsuit)^R$). Since $\CA$ is semisimple, $\widehat{\act_\emptyset}^\heartsuit:\widehat{\CA}^\heartsuit\to \widehat{\SSkcat_\CA(\Sigma^\circ)^\heartsuit}$ is exact. Since $\widehat{\CA}^\heartsuit$ is compactly generated and $\widehat{\act}^\heartsuit:\widehat{\CA}^\heartsuit\times \widehat{\SSkcat_\CA(\Sigma^\circ)^\heartsuit}\to \widehat{\SSkcat_\CA(\Sigma^\circ)^\heartsuit}$ is a compact functor, every compact object in $\widehat{\SSkcat_\CA(\Sigma^\circ)^\heartsuit}$ is a $\widehat{\CA}^\heartsuit$-compact object. By definition of free cocompletion, $\emptyset\in \widehat{\SSkcat_\CA(\Sigma^\circ)^\heartsuit}$ is a compact object, so is a $\widehat{\CA}^\heartsuit$-compact object. Since $\emptyset\in \widehat{\SSkcat_\CA(\Sigma^\circ)^\heartsuit}$ is a $\widehat{\CA}^\heartsuit$-compact, the right adjoint $(\widehat{\act_\emptyset}^\heartsuit)^R$ is also colimit-preserving, so in particular right exact. A right adjoint is left exact, so $(\widehat{\act_\emptyset}^\heartsuit)^R$ is also exact. Since $\mathbb{1}\in \widehat{\CA}^\heartsuit\subset \widehat{\CA}$, we have 
\[\SSkalg^\Int_\CA(\Sigma^\circ)\cong\widehat{\act_\emptyset}^R\circ \widehat{\act_\emptyset} (\mathbb{1})\simeq (\widehat{\act_\emptyset}^\heartsuit)^R\circ \widehat{\act_\emptyset}^\heartsuit(\mathbb{1})\cong \Skalg^\Int_\CA(\Sigma^\circ) \]
\end{proof}

\begin{rem}\label{internal skein algebra as derived coend}
We can write the internal derived skein algebra as a derived coend,
\[ \SSkalg^\Int_\CA(\Sigma^\circ)\cong\int^{X\in \CA}\Hom_{\SSkcat_\CA(\Sigma^\circ)}(\act_\emptyset(X),\mathbb{1})\otimes^\mathbb{L} X     \]
In particular, when $\CA$ is semisimple, 
\[  \SSkalg^\Int_\CA(\Sigma^\circ)\cong \bigoplus_{X\in \Irr(\CA)} \Hom_{\SSkcat_\CA(\Sigma^\circ)}(\act_\emptyset(X),\mathbb{1})\otimes X  \]
Here $\Irr(\CA)$ is the set of all isomorphism classes of simple objects.
\end{rem}

an oriented surface is classified by genus $g$ and punctured $r$, so we can denote an oriented surface $\Sigma$ by $\Sigma_{g,r}$. In particular, $\Sigma_{g,0}$ are closed surfaces. Combine Proposition \ref{semisimpleintskalg} and \cite{BZBJ18} Section 6, we have
\begin{expl}
Let's consider closed surface $\Sigma=\Sigma_{g,0}$. Let $\CA=\rep_q^{\fd}(G)$ where $q$ is not a root of unity. We have \[\SSkalg^\Int_G(\Sigma_{g,0}^\circ)\cong\mathcal{D}_q(G)^{\otimes g}\]
where $\mathcal{D}_q(G)$ is the algebra of quantum differential operators on $G$. Here $\otimes$ denotes the braided tensor product of $\widehat{\CA}$.
\end{expl}
\begin{expl}
Let's consider punctured disk $\Sigma=\Sigma_{0,r}$ $(r\ge 1)$. Let $\CA=\rep^{\fd}_q(G)$ where $q$ is not a root of unity. We have 
\[ \SSkalg^\Int_G(\Sigma^\circ)\cong\mathcal{O}_q(G)^{\otimes r}  \]
where $\mathcal{O}_q(G)$ is the quantum coordinate algebra of $G$. Here $\otimes$ denotes the braided tensor product of $\widehat{\CA}$.
\end{expl}
\begin{expl}\label{LskalgintG}
In general, let's consider $\Sigma=\Sigma_{g,r}$. Let $\CA=\rep^{\fd}_q(G)$ where $q$ is not a root of unity. We have
\[ \SSkalg^\Int_\CA(\Sigma^\circ_{g,r})\cong \mathcal{D}_q(G)^{\otimes g}\otimes \mathcal{O}_q(G)^{\otimes r}    \]
Here $\otimes$ denotes the braided tensor product of $\widehat{\CA}$.
\end{expl}

We have the following non-trivial bridge between ordinary skein category and derived skein category:
\begin{thm}
Let $\CA$ be semisimple ribbon tensor category, $\Sigma$ an oriented surface, and $\Sigma^\circ$ its punctured surface.
\[ \CD(\Skcat_\CA(\Sigma^\circ))\simeq \widehat{\SSkcat_\CA(\Sigma^\circ)}   \]
\end{thm}
\begin{proof}
Following  \ref{DLMod=LModhat}, \ref{semisimpleintskalg}, \ref{hatSkcat=LModhat}.
\end{proof}

\begin{rem}
Note that the above theorem is not a corollary of lemma \ref{DC=hatC} in which the input category is a $\mathbb{k}$-linear category but $\SSkcat_\CA(\Sigma^\circ)$ is a non-trivial DG category usually.
\end{rem}

\subsection{Derived skein category of Annulus}

In particular, derived skein category of the Annulus is specifically important. It is one of elementary ingredients for derived skein category of a closed surface which we will analysis in Section \ref{skeincatofclosed}. It also has rich structure which leads directly a lot of miracles in skein theory. For example, in the previous Section Lemma \ref{skcatannisrigid}, we show that $\widehat{\SSkcat_\CA(\Ann)}$ is a compact-rigid tensor category in $\PPr_c$. This leads a lot of interesting results in last section. For ordinary case, this was already studied deeply in \cite{BZBJ18}\cite{BZBJ18a}\cite{Saf19}\cite{GJS22}. In this section, we view $\widehat{(-)}$ as a cocomplete DG category for simplicity.

Let $\FZ_1(\widehat{\CA})$ be the Drinfeld center of the monoidal DG category $\widehat{\CA}$. Since $\widehat{\CA}$ is braided, we have a natural braided monoidal functor $\widehat{\CA}\otimes \widehat{\CA}^{\sigma op}\to \FZ_1(\widehat{\CA})$, see \cite{EGNO15} Prop 8.6.1., given by the left and right action of $\widehat{\CA}$ on itself. Unpacking this braided monoidal functor, we have a monoidal functor $T:\widehat{\CA}\otimes \widehat{\CA}^{\sigma op}\to \widehat{\CA}$ and natural isomorphisms $V\otimes T(x)\to T(x)\otimes V$ for all $x\in \widehat{\CA}\otimes \widehat{\CA}^{\sigma op}$ and $V\in \widehat{\CA}$. Let $T^R$ be the right adjoint of $T$. For instance,  we obtain the \emph{field goal transform}
\[ \tau_V:V\otimes T(T^R(\mathbb{1}))\to T(T^R(\mathbb{1}))\otimes V    \]

\begin{prop}
There is an isomorphism
\[  T(T^R(\mathbb{1}))\cong \SSkalg^{\Int}_\CA(\Ann)      \]
\end{prop}
\begin{proof}
Left hand side is the value of the  derived functor of $T^\heartsuit\circ (T^R)^\heartsuit$ on $\mathbb{1}$, see Example \ref{widehatF=derivedfunctor}. The right hand side is the value of derived functor of $(\widehat{\act_\emptyset}^\heartsuit)^R\circ (\widehat{\act_\emptyset})^\heartsuit$ on $\mathbb{1}$. By \cite{BZBJ18} Cor 6.4, we have $T^\heartsuit\circ (T^R)^\heartsuit(\mathbb{1})\cong (\widehat{\act_\emptyset}^\heartsuit)^R\circ (\widehat{\act_\emptyset})^\heartsuit(\mathbb{1})\cong \Skalg^\Int_\CA(\Ann)$. Taking derived functors on both sides, we get the above equivalence.
\end{proof}

\begin{defn}\label{i*}
Let $D^2\to \Ann$ be an inclusion given by the including a disk $D^2$ into an annulus along some small band, as depicted in Fig. By functoriality of derived skein category, this inclusion induces a functor $i_*:\widehat{\CA}\to \widehat{\SSkcat_\CA(\Ann)}$ in $\PPr_c$ Note that it's same as $\widehat{\act_\emptyset}:\widehat{\CA}\to \widehat{\SSkcat_\CA(\Sigma^\circ)}$ where $\Sigma$ is a closed disk see Section \ref{SectionDSA}.
\end{defn}
\begin{lem}\label{i*dominant}
The functor $i_*$ is a tensor functor whose essential image generates $\widehat{\SSkcat_\CA(\Ann)}$ see Definition \ref{generatingcategory}.
\end{lem}
\begin{proof}
Since $\emptyset$ is a $\widehat{\CA}$-compact generator see \ref{emptyasgenerator}, then by \ref{rightadjointdominantfunctorisconservative} and \ref{genrator=conservative}, we get the claim. 
\end{proof}

\begin{rem}
Note that there is another functor which has the same source and target as $i_*$. Consider a functor 
\begin{align*}
\coinv_r:\widehat{\SSkcat_\CA(\Ann)}&\to \widehat{\CA} \\
M&\mapsto M\otimes^\mathbb{L}_{\SSkalg^\Int_\CA(\Ann)}\mathbb{1}
\end{align*}
given by derived coinvariants on the right. Its right adjoint $\widehat{\CA}\to \widehat{\SSkcat_\CA(\Ann)}$ is given by sending $V\in \widehat{\CA}$ to the trivial right $\SSkalg^\Int_\CA(\Ann)$-module.
\end{rem}

By Theorem \ref{hatSkcat=LModhat}, we know that $\widehat{\SSkcat_\CA(\Ann)}\simeq \LMod_{\SSkalg^\Int_\CA(\Ann)}(\widehat{\CA})$. If we use this form to represent $\widehat{\SSkcat_\CA(\Ann)}$, we have a more algebraic definition for $i_*$ which is a free module functor $V\mapsto \SSkalg^\Int_\CA(\Ann)\otimes V$. Then the right adjoint of $i_*$ identifies simply with the forgetful functor from $\LMod_{\SSkalg^\Int_\CA(\Ann)}(\widehat{\CA})\to \widehat{\CA}$.

As we mentioned in the beginning of this section, our goal is to understand derived skein category of closed surfaces $\widehat{\SSkcat_\CA(\Sigma)}\simeq \widehat{\SSkcat_\CA(\Sigma^\circ)}\boxtimes_{\widehat{\SSkcat_\CA(\Ann)}}\widehat{\CA}$. Thus we also need to understand module structure over $\widehat{\SSkcat_\CA(\Ann)}$.
\begin{defn}
A \emph{braided module category} over $\widehat{\CA}$ is a module over $\widehat{\SSkcat_{\CA}(\Ann)}$.
\end{defn}

\begin{expl}\label{explBMC}
\begin{enumerate}
\item $\widehat{\CA}$ is a braided module category over $\widehat{\CA}$.
\item $\widehat{\SSkcat_\CA(\Sigma^\circ)}$ is a braided module category over $\widehat{\CA}$. Here the module structure is given by insertion of annuli along the puncture. \footnote{ Meanwhile, $\widehat{\act_\emptyset}:\widehat{\CA}\to \widehat{\SSkcat_\CA(\Sigma^\circ)}$ also gives an natural right $\widehat{\CA}$-module structure on $\widehat{\SSkcat_\CA(\Sigma^\circ)}$. These two right action are compatible.}
\end{enumerate}
\end{expl}

\begin{defn}\label{QMM}
Let $A$ be an algebra in cocomplete DG category $\widehat{\CA}$. A \emph{quantum moment map} is a homomorphism of DG algebras in $\widehat{\CA}$
\[ \mu:\SSkalg^\Int_\CA(\Ann)\to A   \]
such that the following diagram commutes in cocomplete DG category $\widehat{\CA}$
\[ \xymatrix{ & A\otimes \SSkalg^\Int_\CA(\Ann) \ar[r]^-{\id\otimes \mu} \ar[dd]_{\tau} & A\otimes A \ar[rd]^{m} & \\ & & & A \\ & \SSkalg^\Int_\CA(\Ann)\otimes A \ar[r]_-{\mu\otimes \id} & A\otimes A \ar[ur]_m  }   \]
\end{defn}
\begin{lem}
Let $A$ be an algebra in $\widehat{\CA}$, and $\mu:\SSkalg^\Int_\CA(\Ann)\to A$ be a quantum moment map. Then $\LMod_A(\widehat{\CA})$ becomes a $\widehat{\SSkcat_\CA(\Ann)}\simeq \LMod_{\SSkalg^\Int_\CA(\Ann)}(\widehat{\CA})$-module category.
\end{lem}
\begin{proof}
We want to define an action
\[ \LMod_A(\widehat{\CA})\otimes \LMod_{\SSkalg^\Int_\CA(\Ann)}(\widehat{\CA})\to \LMod_A(\widehat{\CA})    \]

Given $M\in \LMod_A(\widehat{\CA})$, $\mu$ gives $M$ a left $\SSkalg^\Int_\CA(\Ann)$-module structure. The point is that $M$ already has a left $A$-action and now also a left $F$-action. To tensor $M$ over $\SSkalg^\Int_\CA(\Ann)$ with a left $\SSkalg^\Int_\CA(\Ann)$-module $V$, we need a right $\SSkalg^\Int_\CA(\Ann)$-action on $M$. This is given by field goal transformation. Now we must check this right $F$-action is compatible with the left $A$-action, i.e. that $M$ is an $(A,\SSkalg^\Int_\CA(\Ann))$-bimodule. This is exactly the quantum moment map equation.

Now let $V\in \LMod_{\SSkalg^\Int_\CA(\Ann)}(\widehat{\CA})$. Since $M$ is a right $\SSkalg^\Int_\CA(\Ann)$-module and $V$ is a left $\SSkalg^\Int_\CA(\Ann)$-module, we can define $M\otimes^\mathbb{L}_{\SSkalg^\Int_\CA(\Ann)}V$. Since the left $A$-action on $M$ commutes with the right $\SSkalg^\Int_\CA(\Ann)$-action, the $A$-action descends to the colimit defining $M\otimes^\mathbb{L}_{\SSkalg^\Int_\CA(\Ann)}V$. Hence
\[ M\otimes^\mathbb{L}_{\SSkalg^\Int_\CA(\Ann)} V \in \LMod_A(\widehat{\CA})  \]
So we obtain the functor
\begin{align*}
\LMod_A(\widehat{\CA})\boxtimes \LMod_{\SSkalg^\Int_\CA(\Ann)}(\widehat{\CA})&\to \LMod_A(\widehat{\CA}) \\ 
M\boxtimes V&\mapsto M\otimes^\mathbb{L}_{\SSkalg^\Int_\CA(\Ann)}V
\end{align*}

\end{proof}

\begin{expl} 
Conversely, an important class of quantum moment maps is given by braided module category. The tensor functor $i_*:\widehat{\CA}\to \widehat{\SSkcat_\CA(\Ann)}$ gives a braided module category a  $\widehat{\CA}$-module structure. By definition, given a braided module category $\CM$, and $M\in \CM$ be a $\widehat{\CA}$-compact generator i.e. $i_*^R\circ \act^R_M:\CM\to \widehat{\CA}$ is conservative and cocontinuous, we can consider its internal endomorphism algebra $\intend_{\widehat{\CA}}(M)$ in $\widehat{\CA}$. 

We want to construct a DG algebra map $\SSkalg^\Int_\CA(\Ann):=\intend_{\widehat{\CA}}(\mathbb{1})\to \intend_{\widehat{\CA}}(M)$. This map comes from a right $\SSkalg^\Int_\CA(\Ann)$-action on $\intend_{\widehat{\CA}}(M)$, $\intend_{\widehat{\CA}}(M)\otimes \intend_{\widehat{\CA}}(\mathbb{1})\to \intend_{\widehat{\CA}}(M)$. By the defining adjuncation of internal hom, it's enough to construct a map $M\odot_{\widehat{\CA}} \intend_{\widehat{\CA}}(M)\otimes\intend_{\widehat{\CA}}(\mathbb{1})=M\odot_{\widehat{\SSkcat_\CA(\Ann)}} (\intend_{\widehat{\CA}}(\mathbb{1})\otimes\intend_{\widehat{\CA}}(M)\otimes \intend_{\widehat{\CA}}(\mathbb{1}))\to M$ in $\CM$. This map is given by the following:
\begin{align*}
M\odot (\intend_{\widehat{\CA}}(\mathbb{1})\otimes\intend_{\widehat{\CA}}(M)\otimes \intend_{\widehat{\CA}}(\mathbb{1}))&\xrightarrow{\tau} M\odot \intend_{\widehat{\CA}}(\mathbb{1})\otimes \intend_{\widehat{\CA}}(\mathbb{1})\otimes \intend_{\widehat{\CA}}(M)\\
&\xrightarrow{\mu} M\odot \intend_{\widehat{\CA}}(\mathbb{1})\otimes \intend_{\widehat{\CA}}(M):=M\odot \intend_{\widehat{\CA}}(M)\to M
\end{align*}
The last map corresponds to the $\id_{\intend_{\widehat{\CA}}(M)}\in \Hom_{\widehat{\CA}}(\intend_{\widehat{\CA}}(M),\intend_{\widehat{\CA}}(M))$.

Note that the action map $\intend_{\widehat{\CA}}(M)\otimes \intend_{\widehat{\CA}}(\mathbb{1})\to \intend_{\widehat{\CA}}(M)$ commutes with the right $\intend_{\widehat{\CA}}(M)$-action on itself. This exactly means the algebra map $\mu_M:\SSkalg^\Int_\CA(\Ann)\to \intend_{\widehat{\CA}}(M)$ is a quantum moment map.
\end{expl}

By the examples in Example \ref{explBMC} for braided module category, we can also get examples for quantum moment map
\begin{expl}
\begin{enumerate}
\item The quantum moment map for $\widehat{\CA}$ attached to $\mathbb{1}$ is the counit homomorphism $\varepsilon:\SSkalg^\Int_\CA(\Ann)\to \mathbb{1}$.
\item The quantum moment map for $\widehat{\SSkcat_\CA(\Sigma^\circ)}$ attached to $\mathbb{1}$ is the canonical algebra map $\SSkalg^\Int_\CA(\Ann)\to \SSkalg^\Int_\CA(\Sigma^\circ)$.
\end{enumerate}
\end{expl}

\begin{prop}\label{quantummomentmapcommutes}
Let $A$ be an algebra of  $\widehat{\CA}$ in $\PPr_c$ with a quantum moment map $\SSkalg^\Int_\CA(\Ann)\to A$. Then we have 
\[ \LMod_{\SSkalg^\Int_\CA(\Ann)}(\LMod_A(\widehat{\CA}))\simeq \LMod_A(\LMod_{\SSkalg^\Int_\CA(\Ann)}(\widehat{\CA}))\simeq \LMod_A(\widehat{\SSkcat_\CA(\Ann)})     \]
\end{prop}
\begin{proof}
The commutation of the  $A$ and $\SSkalg^\Int_\CA(\Ann)$-module structures is exactly the
quantum moment map condition.
\end{proof}

In Lemma \ref{emptyasgenerator}, we show that $\mathbb{1}\in \widehat{\SSkcat_\CA(\Sigma^\circ)}$ is a $\widehat{\CA}$-compact generator, we also want to know if it is a $\widehat{\SSkcat_\CA(\Ann)}$-compact generator. This leads the following lemma:
\begin{lem}\label{Acompactgenerator->Anncompactgenerator}
Let $\CM$ be a braided module category, and $M\in \CM$ be a $\widehat{\CA}$-compact generator. Then $M$ is also a $\widehat{\SSkcat_\CA(\Ann)}$-compact generator, i.e. $ \act^R_M:\CM\to \widehat{\SSkcat_\CA(\Ann)}$ is conservative and cocontinuous. 
\end{lem}
\begin{proof}
By \ref{i*dominant} and \ref{rightadjointdominantfunctorisconservative}, we know $i^R_*$ is conservative and cocontinuous. Since conservative and cocontinuous functors themselves reflects conservativity and cocontinuity, we have $\act^R_M$ is conservative and cocontinuous.
\end{proof}

Thus we have \[\CM\simeq \LMod_{\intend_{\widehat{\SSkcat_\CA(\Ann)}}(M)}(\widehat{\SSkcat_\CA(\Ann)})\simeq \LMod_{\intend_{\widehat{\CA}}(M)}(\widehat{\CA})\]
and we have $i_*^R(\intend_{\widehat{\SSkcat_\CA(\Ann)}}(M))=\intend_{\widehat{\CA}}(M)$ where $i_*^R$ is forgetful functor. This leads the following proposition
\begin{prop}\label{algebrainAnn=algebrawithqmm}
An algebra $A$ in $\widehat{\SSkcat_\CA(\Ann)}$ is the same with its underlying algebra $A$ in $\widehat{\CA}$ with a quantum moment map $\SSkalg^\Int_\CA(\Ann)\to A$.
\end{prop}

one may be curious about what $\LMod_{i_*(\intend_{\widehat{\CA}}(M))}(\widehat{\SSkcat_\CA(\Ann)})$ is?
\begin{prop}
Let $\CM$ be a braided module category, and $M\in \CM$ be a $\CA$-compact generator.
\[   \LMod_{i_*(\intend_{\widehat{\CA}}(M))}(\widehat{\SSkcat_\CA(\Ann)})\simeq \CM\boxtimes_{\widehat{\CA}}\widehat{\SSkcat_\CA(\Ann)}     \]
\end{prop}
\begin{proof}
This is due to \ref{i*dominant} and corollary 4.13 in \cite{BZBJ18}.

\end{proof}

\subsection{Derived skein algebra}\label{SectionDSA}
\begin{defn}\label{DSA}
Let $\Sigma$ be an oriented surface, and $\CA$ be a ribbon tensor category. The \emph{derived skein algebra} is defined as
\[ \SSkalg_\CA(\Sigma):=\End_{\widehat{\SSkcat_\CA(\Sigma)}}(\emptyset)=\End_{\SSkcat_\CA(\Sigma)}(\emptyset).   \]
\end{defn}
\begin{rem}
The derived skein algebra $\SSkalg_\CA(\Sigma^\circ)$ of $\Sigma^\circ$ is the value of internal derived skein algebra $\SSkalg^\Int_\CA(\Sigma^\circ)$ on $\mathbb{1}\in \CA$. 
\end{rem}

By \ref{1boxtimesVarecompactgenerators}, if we curious on compact generation of $\widehat{\SSkcat_\CA(\Sigma)}$, we would be in particular interested in the endomorphism spaces of those compact generators $\mathbb{1}\boxtimes_{\widehat{\SSkcat_\CA(\Ann)}}V$. This motivates the following theorem which gives a derived version for theorem 5.4 in \cite{BZBJ18a}: 
\begin{thm}\label{QHR}
Let $\Sigma$ be an oriented surface, and $\CA$ ribbon tensor category. 
\[ \Hom_{\widehat{\SSkcat_\CA(\Sigma)}}(\mathbb{1}_{\Sigma^\circ}\boxtimes_{\widehat{\SSkcat_\CA(\Ann)}} V,\mathbb{1}_{\Sigma^\circ}\boxtimes_{\widehat{\SSkcat_\CA(\Ann)}} V)\simeq \Hom_{\widehat{\CA}}(\mathbb{1}, \SSkalg^\Int_\CA(\Sigma^\circ)\otimes^\mathbb{L}_{\SSkalg^\Int_\CA(\Ann)}\intend_{\widehat{\CA}}(V))   \]
\end{thm}
\begin{proof}
Since $\widehat{\SSkcat_\CA(\Ann)}$ is a compact-rigid tensor category (Lemma \ref{skcatannisrigid}), by Prop \ref{DSS}, we have 
\begin{align*}
\End_{\widehat{\SSkcat_\CA(\Sigma)}}(\mathbb{1}_{\Sigma^\circ}\boxtimes_{\widehat{\SSkcat_\CA(\Ann)}} V)&\cong \Hom_{\widehat{\SSkcat_\CA(\Ann)}}(\SSkalg^\Int_\CA(\Ann),\intend(\mathbb{1}_{\Sigma^\circ})\otimes \intend(V) ) \\
&
\cong \Hom_{\widehat{\SSkcat_\CA(\Ann)}}(\SSkalg^\Int_\CA(\Ann),\SSkalg^\Int_\CA(\Sigma^\circ)\otimes^\mathbb{L}_{\SSkalg^\Int_\CA(\Ann)}\intend(V))\\
& \cong \Hom_{\widehat{\SSkcat_\CA(\Ann)}}(i_*(\mathbb{1}_\mathbb{D}),\SSkalg^\Int_\CA(\Sigma^\circ)\otimes^\mathbb{L}_{\SSkalg^\Int_\CA(\Ann)}\intend(V)) 
\\&
\cong \Hom_{\widehat{\CA}}(\mathbb{1}_\mathbb{D},i^R_*(\SSkalg^\Int_\CA(\Sigma^\circ)\otimes^\mathbb{L}_{\SSkalg^\Int_\CA(\Ann)}\intend(V))\\
&\cong \Hom_{\widehat{\CA}}(\mathbb{1}_\mathbb{D},\SSkalg^\Int_\CA(\Sigma^\circ)\otimes^{\mathbb{L}}_{\SSkalg^\Int_\CA(\Ann)}\intend(V))
\end{align*}
Here $i_*$ is from Definition \ref{i*}.
\end{proof}
In particular, we have
\begin{thm}\label{DQHR} If $V=\mathbb{1}$, then
\[  \SSkalg_\CA(\Sigma)\cong \Hom_{\widehat{\CA}}(\mathbb{1},\SSkalg^\Int_\CA(\Sigma^\circ)\otimes^{\mathbb{L}}_{\SSkalg^\Int_\CA(\Ann)}\mathbb{1})   \]
\end{thm}

\begin{lem}\label{SkalgS^2=C(g)}
Let $G$ be connected reductive group over a field $\mathbb{k}$, and $\mathbb{k}\hookrightarrow \mathbb{C}$ (e.g. $q$ is not a root of unity) 
\[ \SSkalg_G(S^2)\cong C_\bullet(\mathfrak{g};\mathbb{k})\cong \rmH_\bullet(G_{\mathbb{C}}(\mathbb{C});\mathbb{k})\]
where $C_\bullet(\mathfrak{g};\mathbb{k})$ is the Lie algebra chain complex, and $\rmH_\bullet(G_{\mathbb{C}}(\mathbb{C});\mathbb{k})$ is singular homology of $G_{\mathbb{C}}(\mathbb{C})$.
\end{lem}
\begin{proof}
For the first $\cong$, $\SSkalg_G(S^2)\cong\mathbb{k}\otimes^{\mathbb{L}}_{\mathcal{O}_q(G)}\mathbb{k}\cong \mathbb{k}\otimes^{\mathbb{L}}_{U(\mathfrak{g})}\mathbb{k}=: C_\bullet(\mathfrak{g};\mathbb{k})$. 

For the second $\cong$, by definition, the Lie algebra $\mathfrak{g}_{\mathbb{C}}$ of algebraic group $G_\mathbb{C}:=G\otimes_{\mathbb{k}}\mathbb{C}$ is just given by usual lie algebra of the lie group $G_{\mathbb{C}}(\mathbb{C})$ of $\mathbb{C}$-points of $G$. By the Iwasawa decomposition, a connected Lie group $G_{\mathbb{C}}(\mathbb{C})$ deformation retracts onto its maximal compact Lie subgroup $\mathbb{S}G$. Since G is reductive, we have $\mathfrak{t}\otimes_{\mathbb{R}}\mathbb{C}\cong \mathfrak{g}_{\mathbb{C}}$ where $\mathfrak{t}$ is the Lie algebra of $\mathbb{S}G$. Then we have $C^\bullet(\mathfrak{t};\mathbb{R})\otimes_\mathbb{R}\mathbb{C}\cong C^\bullet(\mathfrak{g}_{\mathbb{C}};\mathbb{C})$. $C^\bullet(\mathfrak{t};\mathbb{R})$is quasi-isomorphic to de Rham cochain complex $\Omega^\bullet(\mathbb{S}G)$. By de Rham theorem, $\Omega^\bullet(\mathbb{S}G)$ is quasi-isomorphic to singular cochains $C^\bullet(\mathbb{S}G;\mathbb{R})$. By universal coefficient theorem,  $C^\bullet(\mathbb{S}G);\mathbb{C})\cong C^\bullet(\mathbb{S}G;\mathbb{R})\otimes_\mathbb{R}\mathbb{C}$, . So we have $C^\bullet(\mathfrak{g}_{\mathbb{C}};\mathbb{C})\cong C^\bullet(\mathbb{S}G;\mathbb{C})$. Since Lie group is formal, we have $C^\bullet(\mathfrak{g}_{\mathbb{C}};\mathbb{C})\cong \rmH^\bullet(\mathfrak{g}_{\mathbb{C}};\mathbb{C}) \cong \rmH^\bullet(\mathbb{S}G;\mathbb{C})\simeq \rmH^\bullet(G_{\mathbb{C}}(\mathbb{C});\mathbb{C})$.  Then by Poincar\'e duality, we get $C_\bullet(\mathfrak{g}_{\mathbb{C}};\mathbb{C})\cong \rmH_\bullet(G_{\mathbb{C}}(\mathbb{C});\mathbb{C})$. Note that since G is reductive, it is unimodular, so the twist is trivial. Thus we have $C_\bullet(\mathfrak{g};\mathbb{k})\otimes_{\mathbb{k}}\mathbb{C}\cong C_\bullet(\mathfrak{g}_{\mathbb{C}};\mathbb{C})\cong \rmH_\bullet(G_{\mathbb{C}}(\mathbb{C});\mathbb{C})\cong \rmH_\bullet(G_\mathbb{C}(\mathbb{C});\mathbb{k})\otimes_{\mathbb{k}}\mathbb{C}$. Since $G_\mathbb{C}(\mathbb{C})$ has finite-dimensional homologies in each degree and graded dimensions match, the comparison over $\mathbb{C}$ descends to $\mathbb{k}$. Hence $\rmH_\bullet(\mathfrak{g};\mathbb{k})\cong \rmH_\bullet(G_\mathbb{C}(\mathbb{C});\mathbb{k})$.
\end{proof}

\begin{rem}
Using \ref{DQHR}, we also compute $\SSkalg_{\mathbb{G}_m}(\Sigma_g)\cong\mathcal{D}_q(\mathbb{G}_m)^{\otimes g}\otimes \mathbb{k}[\varepsilon]/(\varepsilon^2)$ where $\varepsilon$ has homological degree 1, see Equation \ref{SkalgGm(Sigmag)}.
\end{rem}

\subsection{Derived skein category of closed surfaces}\label{skeincatofclosed}

In this section, we study derived skein category of a closed surface $\Sigma$. All surfaces $\Sigma$ in this section are closed. By excision theorem, this means we want to understand the relative tensor product $\widehat{\SSkcat_\CA(\Sigma^\circ)}\boxtimes_{\widehat{\SSkcat_\CA(\Ann)}} \widehat{\CA}$. There are two important things to understand: $\widehat{\SSkcat_\CA(\Ann)}$ itself, and the action of $\widehat{\SSkcat_\CA(\Ann)}$ on both left and right. Notice that there is not just an action of $\widehat{\CA}$ on $\widehat{\SSkcat_\CA(\Ann)}$, but conversely, also an action of $\widehat{\CA}$ on $\widehat{\SSkcat_\CA(\Ann)}$ induced by a tensor functor.

\begin{thm}\label{Skcat=BMod}
Let $\Sigma$ be closed, and $\CA$ be a ribbon tensor category. There is an equivalence of stable presentable $\mathbb{k}$-linear $\infty$-categories
\[ \widehat{\SSkcat_\CA(\Sigma)}\simeq \BMod_{\SSkalg^{\Int}_\CA(\Sigma^\circ)|\mathbb{1}_\CA}(\widehat{\SSkcat_\CA(\Ann)})
\]
Here although $\SSkalg^\Int_\CC(\Sigma^\circ),\mathbb{1}_\CA$ are algebras in $\widehat{\CA}$ but they both admit a quantum moment map, by \ref{algebrainAnn=algebrawithqmm}, we can regard them as algebras in $\widehat{\SSkcat_\CA(\Ann)}$, so the above notation makes sense.
\end{thm}
\begin{proof}
We have $\widehat{\SSkcat_\CA(\Sigma)}\simeq \widehat{\SSkcat_\CA(\Sigma^\circ)}\boxtimes_{\widehat{\SSkcat_\CA(\Ann)}}\widehat{\CA}$. So our goal is to prove
\[ \widehat{\SSkcat_\CA(\Sigma^\circ)}\boxtimes_{\widehat{\SSkcat_\CA(\Ann)}}\widehat{\CA}\simeq \BMod_{\SSkalg^\Int_\CA(\Sigma^\circ)|\mathbb{1}_\CA}(\widehat{\SSkcat_\CA(\Ann)})  \]

By \ref{quantummomentmapcommutes}, we have that
\begin{align*}
\BMod_{\SSkalg^\Int_\CA(\Sigma^\circ)|\mathbb{1}}(\widehat{\SSkcat_\CA(\Ann)})&\simeq \BMod_{\SSkalg^\Int_\CA(\Sigma^\circ)|\mathbb{1}}(\LMod_{\SSkalg^\Int_\CA(\Ann)}(\widehat{\CA}))\\
&\simeq  \LMod_{\SSkalg^\Int_\CA(\Ann)}(\BMod_{\SSkalg^\Int_\CA(\Sigma^\circ)|\mathbb{1}}(\widehat{\CA}))\\
&\simeq \LMod_{\SSkalg^\Int_\CA(\Ann)}(\widehat{\SSkcat_\CA(\Sigma^\circ)})
\end{align*}

By \ref{Acompactgenerator->Anncompactgenerator}, we know that $\mathbb{1}$ is a $\widehat{\SSkcat_\CA(\Ann)}-$compact generator, so we can use theorem \ref{Monadicityforrelativetensorproduct} to get the claim.
\end{proof}

\begin{lem}\label{SkcatS^2=ModA}
Let $\CA$ be a ribbon tensor category. 
\[ \widehat{\SSkcat_\CA(S^2)}\simeq\BMod_{\mathbb{1}_\mathbb{D}|\mathbb{1}_\mathbb{D}}(\widehat{\SSkcat_\CA(\Ann)}) \simeq \Mod_{\mathbb{1}_{\mathbb{D}}\otimes^{\mathbb{L}}_{\SSkalg^\Int_\CA(\Ann)}\mathbb{1}_{\mathbb{D}}}(\hat{\CA})   \]
\end{lem}
\begin{proof}
The relative tensor product $\widehat{\SSkcat_\CA(S^2)}\simeq \widehat{\CA}\boxtimes_{\widehat{\SSkcat_\CA(\Ann)}}\widehat{\CA}$ is obtained as the geometric realization of the simplicial object \[\xymatrix{\widehat{\CA} & \widehat{\CA}\boxtimes_{\widehat{\CA}}\widehat{\SSkcat_\CA(\Ann)}\simeq \widehat{\SSkcat_\CA(\Ann)}\simeq \RMod_{\SSkalg^{\Int}_\CA(\Ann)}(\widehat{\CA}) \ar@<-0.5ex>[l] \ar@<0.5ex>[l] & \ar@<0.5ex>[l] \ar[l] \ar@<-0.5ex>[l]  } \]
in $\PPr$. Since $\widehat{\SSkcat_\CA(\Ann)}$ is compact-rigid (Lemma \ref{skcatannisrigid}), this diagram admits right adjoints which satisfy the Beck-Chevalley conditions. Therefore, the right adjoint to the projection
\[ \widehat{\CA}\to \widehat{\CA}\boxtimes_{\widehat{\SSkcat_\CA(\Ann)}}\widehat{\CA}   \]
is monadic. This monad is given by
\[ \widehat{\CA}\to \widehat{\SSkcat_\CA(\Ann)}\xrightarrow{\coinv_r} \widehat{\CA}   \]
Here the first functor equips an object in $\widehat{\CA}$ with the right $\SSkalg^\Int_\CA(\Ann)$-module structure coming from the quantum moment map $\SSkalg^\Int_\CA(\Ann)\to \mathbb{1}$.
\end{proof}

\begin{lem}\label{skcatS^2=Modskalg}
Let $\CC=\rep^{\fd}_q(G)$ for $q$ not a root of unity. We have
\[ \widehat{\SSkcat_G(S^2)}\simeq \Mod_{\SSkalg_G(S^2)}(\VVect)    \]
i.e. $\SSkcat_G(S^2)$ is Morita equivalent to $B\SSkalg_G(S^2)$.
\end{lem}
\begin{proof}
Since $\{\mathbb{1}\boxtimes_{\widehat{\SSkcat_\CA(\Ann)}}V\}_{V\in \rep_q(G)}$ are generators of $\widehat{\SSkcat_G(S^2)}$, it's enough to prove that the one framed point labeling $X_V$ of $S^2$ labeled by $V$ is 0 in $\widehat{\SSkcat_G(S^2)}$ if $V\notin \Vect=\FZ_2(\rep_q(G))$. To show $X_V\simeq0$, it's also enough to show that $\id_{X_V}\simeq0\in \widehat{\SSkcat_G(S^2)}$. 

By Schur's Lemma, if $V\notin \FZ_2(\rep_q(G))$, then there is a nonzero simple $W\in \rep_q(G)$ and $\lambda\ne 1$ such that
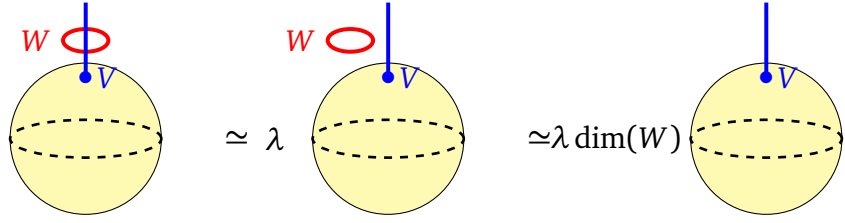
\begin{center}
\begin{tikzpicture}
\draw[fill=yellow,fill opacity=0.3] (0,0) circle [radius=1];
\draw[line width=1.5,color=red] (0,1.3) ellipse (0.3 and 0.15);
\node[left,color=red] at (-0.3,1.3) {$W$};

		\draw[line width=1,dashed] (0,0) ellipse (1 and 0.25);
		\node[color=blue] at (0,0.8) {$\bullet$};
	   \node[right,color=blue] at (0,0.8) {$V$};
       \draw[line width=1.5,color=blue] (0,0.8) -- (0,1.8);

\draw[line width=1.5,color=red] (3.5,1.3) ellipse (0.3 and 0.15);
\node[left,color=red] at (3.2,1.3) {$W$};
\draw[fill=yellow,fill opacity=0.3] (4,0) circle [radius=1];
\draw[line width=1,dashed] (4,0) ellipse (1 and 0.25);
		\node[color=blue] at (4,0.8) {$\bullet$};
	   \node[right,color=blue] at (4,0.8) {$V$};
       \draw[line width=1.5,color=blue] (4,0.8) -- (4,1.8);

\node[] at (2,0) {$\simeq$};
\node[] at (2.5,0) {$\lambda$};

\draw[fill=yellow,fill opacity=0.3] (9,0) circle [radius=1];
\draw[line width=1,dashed] (9,0) ellipse (1 and 0.25);
		\node[color=blue] at (9,0.8) {$\bullet$};
	   \node[right,color=blue] at (9,0.8) {$V$};
       \draw[line width=1.5,color=blue] (9,0.8) -- (9,1.8);

\node[] at (6,0) {$\simeq$};
\node[right] at (6,0) {$\lambda\dim(W)$};
\end{tikzpicture}
\end{center}

On the other hand,
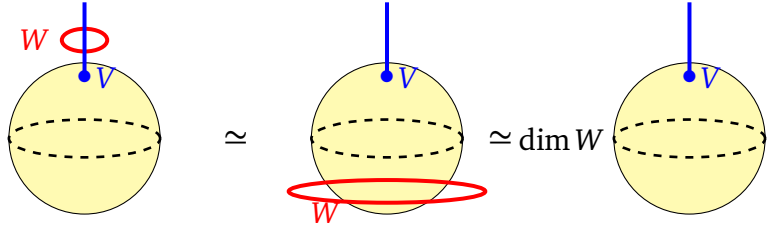
\begin{center}
\begin{tikzpicture}
\draw[fill=yellow,fill opacity=0.3] (0,0) circle [radius=1];
\draw[line width=1.5,color=red] (0,1.3) ellipse (0.3 and 0.15);
\node[left,color=red] at (-0.3,1.3) {$W$};

		\draw[line width=1,dashed] (0,0) ellipse (1 and 0.25);
		\node[color=blue] at (0,0.8) {$\bullet$};
	   \node[right,color=blue] at (0,0.8) {$V$};
       \draw[line width=1.5,color=blue] (0,0.8) -- (0,1.8);

\draw[fill=yellow,fill opacity=0.3] (4,0) circle [radius=1];
\draw[line width=1,dashed] (4,0) ellipse (1 and 0.25);
		\node[color=blue] at (4,0.8) {$\bullet$};
	   \node[right,color=blue] at (4,0.8) {$V$};
       \draw[line width=1.5,color=blue] (4,0.8) -- (4,1.8);

\draw[line width=1.5,color=red] (4,-0.7) ellipse (1.3 and 0.15);
\node[below,color=red] at (3.2,-0.7) {$W$};

   \node[] at (2,0) {$\simeq$};

\draw[fill=yellow,fill opacity=0.3] (8,0) circle [radius=1];
\draw[line width=1,dashed] (8,0) ellipse (1 and 0.25);
		\node[color=blue] at (8,0.8) {$\bullet$};
	   \node[right,color=blue] at (8,0.8) {$V$};
       \draw[line width=1.5,color=blue] (8,0.8) -- (8,1.8);
\node[] at (5.5,0) {$\simeq$};
 \node[right] at (5.6,0) {$\dim W$};  
\end{tikzpicture}
\end{center}
Then we have $(1-\lambda)\dim W\id_{X_V}\simeq0$. Since $\rep_q(G)$ is semisimple, $\dim W\ne 0$. So $\id_{X_V}\simeq0$ if $V\notin \FZ_2(\rep_q(G))=\Vect$.
\end{proof}
\begin{rem}
One can also prove the above purely algebraically. By \ref{QHR}, its enough to show that $\mathbb{k}\otimes^\mathbb{L}_{O_q(G)}V^*\otimes V\simeq 0$ is not exact if $V\notin \FZ_2(\rep_q(G))$. If $q$ is not a root of unity, all differentials will become isomorphisms. See \ref{Gmskcatmorskalg} for a proof when $G=\mathbb{G}_m$.
\end{rem}

\begin{cor}\label{Hom(1,-)useless}
For $q$ is not a root of unity and Heegaard splitting $M=H_1\cup_{\Sigma} H_2$, we have
\[  \Hom_{\Rep_q(G)}(\mathbb{1},\Sk^\Int_G(H_1)\otimes^\mathbb{L}_{\Skalg^\Int_G(\Sigma^\circ)}\Sk^\Int_G(H_2))\cong \Sk^\Int_G(H_1)\otimes^\mathbb{L}_{\Skalg^\Int_G(\Sigma^\circ)}\Sk^\Int_G(H_2)     \]
\end{cor}
\begin{proof}
This is due to $\Sk^\Int_G(H_1)\otimes^{\mathbb{L}}_{\Skalg^\Int_G(\Sigma^\circ)}\Sk_G^\Int(H_2)\in \widehat{\SSkcat_G(S^2)}\cong \Mod_{\rmH_\bullet(G)}(\Rep_q(G))\cong \Mod_{\rmH_\bullet(G)}(\VVect)$.
\end{proof}

\begin{cor}\label{SkcatS2=DModBG}
\[ \widehat{\SSkcat_G(S^2)}\simeq \Mod_{\rmH_\bullet(G)}(\VVect)\simeq \DDMod(BG):=\QQCoh((BG)_{dR})   \]
Here $(BG)_{dR}$ is the de Rham stack of $BG$.
\end{cor}
\begin{proof}
The first $\simeq$ is given by \ref{SkalgS^2=C(g)}, \ref{skcatS^2=Modskalg}.

The second $\simeq$ is well-known.
\end{proof}
\begin{rem}
This is one of the obvious differences between derived skein theory and ordinary skein theory. In ordinary skein theory, we have $\widehat{\Skcat_G(S^2)^\heartsuit}\simeq \Vect$.
\end{rem}
\begin{rem}\label{stringtopologycategoryonBG}
One can show that $\widehat{\SSkcat_G(S^2)}$ is Morita equivalent to the string topology category $\mathcal{S}_{BG}$ constructed by Blumberg-Cohen-Teleman \cite{BCT09}.
\end{rem}

\begin{rem}\label{Skein6functor}
Combining \ref{SkcatS^2=ModA} and \ref{skcatS^2=Modskalg}, there are similarities with \cite{DG13} Section 7.2.2 and \cite{BGR22} Lemma 11. In \cite{DG13} Section 7.2.2, they use six functor formalism for D-modules. Their natural forgetful functor Equation 7.7 seems to correspond to the natural forgetful functor \[\widehat{\SSkcat_G(S^2)}\simeq \Mod_{\mathbb{1}\otimes^\mathbb{L}_{\mathcal{O}_q(G)}\mathbb{1}}(\Rep_q(G))\to \Rep_q(G)\]
although it is not straightforward that Equation 7.7 in $\cite{DG13}$ is a "forgetful" functor. Do we have a skein theoretical six-functor formalism? We would like to return to this question in the future.
\end{rem}

\subsection{Internal derived skein module}

Suppose now $N$ is a compact oriented 3-manifold with $\partial N=\Sigma$. The relative derived skein module defines a DG functor $\underline{\SSk}_\CA(N):\SSkcat_\CA(\Sigma)^\op\to \VVect$ which we can restrict to a DG functor $\SSkcat_\CA(\Sigma^\circ)^\op\to \VVect$. By Theorem \ref{hatSkcat=LModhat}, we obtain a $\SSkalg^\Int_\CA(\Sigma^\circ)$-module.
\begin{defn}
Let $N$ be a compact oriented 3-manifold with $\partial N=\Sigma$. The \emph{internal derived skein module} of $N$ is the DG functor
\begin{align*}
\SSk^\Int_\CA(N):\CA^\op&\to \VVect \\  V&\mapsto \SSk_\CA(N,\act_\emptyset(V))
\end{align*}
\end{defn}

It is a left $\SSkalg^\Int_\CA(\Sigma^\circ)$-module via the map
\[ \Hom_{\SSkcat_\CA(\Sigma^\circ)}(\act_\emptyset(V),\mathbb{1})\otimes \SSk_\CA(N,\act_\emptyset(W))\to \SSk_\CA(N,\act_\emptyset(V\otimes W))   \]
given by gluing map. In particular, the derived skein module is recovered as 
\[ \SSk_\CA(N)\cong \SSk^\Int_\CA(N)(\mathbb{1})    \]
Similarly, one can also define a \emph{right} $\SSkalg^\Int_\CA(\Sigma^\circ)$-module if $N$ is a 3-manifold with $\partial N=\overline{\Sigma}$.

The goal of what follows is to give an explicit description of $\SSk^\Int(\Sigma^\circ)$ in terms of $\SSkalg^\Int(\Sigma^\circ)$. Consider a handlebody $H$ with $\partial H=\Sigma_g$. Punctured surfaces have a gluing pattern presentation. Each gluing pattern presentation gives us a system of $a$ and $b$ cycles in $\Sigma^\circ$, i.e. $2g$ embeddings 
\[ a_1,b_1,\cdots,a_g,b_g:\Ann\to \Sigma^\circ    \]
such that for all $i=1,\cdots, n$ the images of the $a_i$ (respectively $b_i$) are pairwise disjoint, and the intersection of $a_i$ and $b_i$ is a single disk. We choose the $b$-cycles to be contractible in $H$. This induces a map
\[  \SSkalg^\Int_\CA(\Ann)^{\otimes 2g}\to \SSkalg^\Int_\CA(\Sigma^\circ)    \]
of objects in $\widehat{\CA}$. This is an isomorphism. In \cite{BZBJ18}, they proved the ordinary case but the proof still holds for derived case.
\begin{lem}\
(\cite{BZBJ18} Theorem 5.14)\label{2gSkalgAnn=SkalgSigmacirc} The map defined above 
\[ \SSkalg^\Int_\CA(\Ann)^{\otimes 2g}\cong \SSkalg^\Int_\CA(\Theta)\otimes\SSkalg^\Int_\CA(\Theta) \xrightarrow{a\otimes b} \SSkalg^\Int_\CA(\Sigma^\circ)   \]
is an isomorphism of right $\SSkalg^\Int_\CA(b(\Theta))$-modules.
\end{lem}
\begin{rem}
Note that the above not an isomorphism of DG algebras.
\end{rem}

 Let $\Theta$ denote a disk with $g$ smaller disks removed from its interior. The $a$-cycles and the $b$-cycles can be combined to form two embeddings
 \[  a,b:\Theta\to \Sigma   \]

Note that $\Theta$ naturally carries the structure of a right $\mathbb{D}$-module, by inserting disks inside the "outer" annulus in $\Theta$. We can choose the embeddings $a$ and $b$ to be compatible with the right $\mathbb{D}$-module structure on $\Theta$ and $\Sigma^\circ$. We obtain the following maps on internal skeins:
\begin{itemize}
\item The embeddings $a,b:\Theta\hookrightarrow \Sigma^\circ$ determine maps of  internal derived skein algebras
\[ a_*,b_*:\SSkalg^\Int_\CA(\Theta)\to \SSkalg^\Int_\CA(\Sigma^\circ)    \]
\item The embedding of $\Theta\hookrightarrow \mathbb{D}$ determines a map of internal derived skein algebras
\[ \varepsilon:\SSkalg^\Int_\CA(\Theta)\to \SSkalg^\Int_\CA(\mathbb{D})=\mathbb{1}   \]
\item The composite $\Sigma^\circ\times I\hookrightarrow \Sigma\times I\hookrightarrow (\Sigma\times I)\cup_{\Sigma\times \{1\}}H\cong H$ determines a map of left $\SSkalg^\Int_\CA(\Theta)$-modules
\[ \SSkalg^\Int_\CA(\Sigma^\circ)\to \SSk^\Int_\CA(H).   \]
\end{itemize}

\begin{lem}\label{mapg}
The composition
\[  \SSkalg^\Int_\CA(\Theta)\xrightarrow{a_*}\SSkalg^\Int_\CA(\Sigma^\circ)\to \SSk^\Int_\CA(H)    \]
is an isomorphism of left $\SSkalg^\Int_\CA(\Theta)$-modules in $\widehat{\CA}$.
\end{lem}
\begin{proof}
This is due to the embedding $\Theta\times I\hookrightarrow H$ is a deformation retract.
\end{proof}

\begin{lem}\label{mapf}
The following diagram commutes in $\LMod_{\SSkalg^\Int_\CA(\Sigma^\circ)}(\widehat{\CA})$:
\[  \xymatrix{ & \SSkalg^\Int_\CA(\Sigma^\circ) \ar[r] &\SSk^\Int_\CA(H) \\ & \SSkalg^\Int_\CA(\Theta) \ar[u]_b \ar[r]_-{\varepsilon} & \mathbb{1} \ar[u] }   \]
Thus, by universal property of pushout, there is a unique morphism of left $\SSkalg^\Int_\CA(\Sigma^\circ)$-modules from pushout
\[ \SSkalg^\Int_\CA(\Sigma^\circ)\sqcup_{\SSkalg^\Int_\CA(\Theta)}\mathbb{1}\to \SSk^\Int_\CA(H)      \]
Then by universal property of relative tensor product, we have a well-defined morphism of left $\SSkalg^\Int_\CA(\Sigma^\circ)$-modules
\[  f:\SSkalg^\Int_\CA(\Sigma^\circ)\otimes^\mathbb{L}_{b,\SSkalg^\Int_\CA(\Theta),\varepsilon} \mathbb{1}\to \SSkalg^\Int_\CA(\Sigma^\circ)\sqcup_{\SSkalg^\Int_\CA(\Theta)}\mathbb{1}\to  \SSk^\Int_\CA(H)   \]
\end{lem}
\begin{proof}
This is due to the $b$-cycles embedding
\[ \Theta\hookrightarrow \Sigma\hookrightarrow H  \]
factors through the inclusion of a disk since $b$-cycles are contractible.
\end{proof}

\begin{lem}\label{freea-action}
The composition
\[ \SSkalg^\Int_\CA(\Theta)\cong \SSkalg^\Int_\CA(\Theta)\otimes \SSkalg^\Int_\CA(\Theta)\otimes^\mathbb{L}_{b,\SSkalg^\Int_\CA(\Theta),\varepsilon}\mathbb{1}\xrightarrow{a}\SSkalg^\Int_\CA(\Sigma^\circ)\otimes^\mathbb{L}_{b,\SSkalg^\Int_\CA(\Theta),\varepsilon}\mathbb{1}    \]
is an isomorphism in $\widehat{\CA}$
\end{lem}
\begin{proof}
This is due to \ref{2gSkalgAnn=SkalgSigmacirc}.
\end{proof}

\begin{thm}\label{SkintH=derivedtensorproduct}
There is an isomorphism of left $\SSkalg^\Int_\CA(\Sigma^\circ)$-modules in $\widehat{\CA}$
\[ \SSk^\Int_\CA(H)\cong \SSkalg^\Int_\CA(\Sigma^\circ)\otimes^\mathbb{L}_{b,\SSkalg^\Int_\CA(\Theta),\varepsilon}\mathbb{1}\cong \SSkalg^\Int_\CA(\Sigma^\circ)\otimes_{b,\SSkalg^\Int_\CA(\Theta),\varepsilon}\mathbb{1}     \]
\end{thm}
\begin{proof}
By \ref{mapf}, there is a morphism of left $\SSkalg^\Int_\CA(\Sigma^\circ)$-modules
\[ f:\SSkalg^\Int_\CA(\Sigma^\circ)\otimes^{\mathbb{L}}_{b,\SSkalg^\Int_\CA(\Theta),\varepsilon}\mathbb{1} \to \SSk^\Int_\CA(H)  \]

By \ref{mapg}, we have an isomorphism of left $\SSkalg^\Int_\CA(\Theta)$-modules in $\widehat{\CA}$
\[ g:\SSkalg^\Int_\CA(\Theta)\xrightarrow{\sim} \SSk^\Int_\CA(H)   \]

The inclusion of $a$-cycles define a commutative diagram
\[ \xymatrix{& \SSkalg^\Int_\CA(\Theta)\ar[rr]^{g} \ar[rd]_a & & \SSk^\Int_\CA(H) \\ & & \SSkalg^\Int_\CA(\Sigma^\circ)\otimes^\mathbb{L}_{b,\SSkalg^\Int_\CA(\Theta),\varepsilon} \mathbb{1} \ar[ur]_f &  }   \]
Since $g,a$ are both isomorphisms, $f$ is an isomorphism.

The second isomorphism is due to the $b$-action is a free action by \ref{freea-action}.
\end{proof}

The following theorem builds a relation between internal derived skein module and internal skein module
\begin{prop}\label{LSkint=Skint}
Let $\CA$ be semisimple, and $H$ be a handlebody.
\[  \SSk^\Int_\CA(H)\cong \Sk^\Int_\CA(H)     \]
\end{prop}
\begin{proof}
This follows by \ref{SkintH=derivedtensorproduct} and \ref{semisimpleintskalg}.
\end{proof}

\section{Computable formula and results}\label{computeableformulaandresults}

\begin{prop}\label{InternalskeinM-B}
Let $M$ be closed and $M=H_1\bigcup_\Sigma H_2$, $\CA$ be a ribbon tensor category.
\[ \SSk^\Int_\CA(M-\mathbb{B})\cong \SSk^\Int_\CA(H_1)\otimes^\mathbb{L}_{\SSkalg^\Int_\CA(\Sigma^\circ)}\SSk^\Int_\CA(H_2)  \]
\end{prop}
\begin{proof}
It's enough to check it in generating objects. Let $V$ be an object in $\CC$. Then we have
\[ \SSk^\Int_\CA(M-\mathbb{B})(V):=  \SSk_\CA(M-\mathbb{B},\mathcal{P}(V))\cong \underline{\SSk}_\CA(M-\mathbb{B})(\mathcal{P}(V))    \]

On the other hand, if we unpack the following two derived relative tensor products,
\[ \SSk^\Int_\CA(H_1)\otimes^{\mathbb{L}}_{\SSkalg^\Int_\CA(\Sigma^\circ)}\SSk^\Int_\CA(H_2)(V)\cong \underline{\SSk}_\CA(H_1)\otimes^\mathbb{L}_{\SSkcat_\CA(\Sigma^\circ)}\underline{\SSk}_\CA(H_2)(\mathcal{P}(V))  \]
It is easy to see they all are equivalent to
\[ \Tot(\BBar(\SSk_\CA(H_1,\mathcal{P}(V)|_{\partial H_1},-),\SSk_\CA(\Sigma^\circ\times I,-,-),\SSk_\CA(H_2,-,\mathcal{P}(V)|_{\partial H_2}) )   \]

By excision property, we have 
\[ \underline{\SSk}_\CA(M-\mathbb{B})\cong \underline{\SSk}_\CA(H_1)\otimes^\mathbb{L}_{\SSkcat_\CA(\Sigma^\circ)}\underline{\SSk}_\CA(H_2)   \]

Thus we have 
\[\SSk^\Int_\CA(M-\mathbb{B})(V)\cong \SSk^\Int_\CA(H_1)\otimes^\mathbb{L}_{\SSkalg^\Int_\CA(\Sigma^\circ)}\SSk^\Int_\CA(H_2)(V) \]
\end{proof}

We have the following important corollary:
\begin{thm}\label{Sk^intGM-B}
Let $M=H_1\bigcup_{\Sigma}H_2$ be a Heegaard splitting, and $\CA=\rep^\fd_q(G)$ for not a root of unity $q$.
\[\SSk_G^\Int(M-\mathbb{B})\cong \Sk^\Int_G(H_1)\otimes^\mathbb{L}_{\Skalg^\Int_G(\Sigma^\circ)}\Sk^\Int_G(H_2)\]
\end{thm}

\begin{thm}\label{Computableexcision}
Let $M$ be closed and $M=N_1\bigcup_\Sigma N_2$, $\CA$ be a ribbon tensor category.
\[ \SSk_\CA(M)\cong \underline{\SSk}^\Int_\CA(M-\mathbb{B})\otimes^\mathbb{L}_{\SSkcat_\CA(S^2)}\underline{\SSk}_\CA(\mathbb{B})
\]
Here $\underline{\SSk}^\Int_\CA(M-\mathbb{B})$ is the corresponding skein bimodule functor of 
\[\SSk^\Int_\CA(M-\mathbb{B})\cong \SSk^\Int_\CA(N_1)\otimes^{\mathbb{L}}_{\SSkalg^\Int_\CA(\Sigma^\circ)}\SSk^\Int_\CA(N_2).\]
\end{thm}
\begin{proof}
Directly from excision property of derived skein module, and Proposition \ref{InternalskeinM-B}.
\end{proof}

The next corollary will be our main computation tool in this paper.
\begin{thm}\label{computableexcisionforrepqG}
Let $M$ be closed and $M=H_1\bigcup_\Sigma H_2$ is a Heegaard splitting, $\CA=\rep^\fd_q(G)$ for $q$ not a root of unity.
\begin{align*} 
\SSk_G(M)&\cong \Hom_{\Rep_q(G)}(\mathbb{1},\Sk^\Int_G(H_1)\otimes^{\mathbb{L}}_{\Skalg^\Int_G(\Sigma^\circ)}\Sk^\Int_G(H_2))\otimes^\mathbb{L}_{\rmH_\bullet(G)} \mathbb{k}\\
&\cong \Sk^\Int_G(H_1)\otimes^\mathbb{L}_{\Skalg^\Int_G(\Sigma^\circ)}\Sk_G^\Int(H_2)\otimes^{\mathbb{L}}_{\rmH_\bullet(G)}\mathbb{k}
\end{align*}
\end{thm}
\begin{proof}
Since $\SSkcat_G(S^2)$ is Morita equivalent to $B\SSkalg_G(S^2)$ see Lemma \ref{skcatS^2=Modskalg}, by Example \ref{exampleofgeneratedbyinvariants}, Lemma \ref{SkalgS^2=C(g)}, Prop \ref{G1End1F1=GCF}, Theorem \ref{LSkint=Skint} and Yoneda lemma, we get the formula.

The second $\cong$ is due to Corollary \ref{Hom(1,-)useless}
\end{proof}

\begin{expl}
For $q$ not a root of unity, we use Theorem \ref{computableexcisionforrepqG} to compute the homological dimension (which is also called Poincar\'e polynomial) of $\SSk_{\mathbb{G}_m}(S^3)$. The 3-dimensional Schoenflies theorem of Alexander implies that the 3-sphere has a unique Heegaard splitting of genus-0. This implies that 
\begin{align*}
\underline{\SSk}_{\mathbb{G}_m}(S^3-\mathbb{B}):\SSkcat_{\mathbb{G}_m}(S^2)&\to \VVect     \\
X &\mapsto \mathbb{k}
\end{align*}

By Lemma \ref{Gmskcatmorskalg}, we know that $\widehat{\SSkcat_{\mathbb{G}_m}(S^2)}\simeq \LMod_{\SSkalg_{\mathbb{G}_m}(S^2)}(\VVect)$. Thus
\begin{align*}
\SSk_{\mathbb{G}_m}(S^3)&\cong \underline{\SSk}_{\mathbb{G}_m}(S^3-\mathbb{B})\otimes_{\SSkcat_{\mathbb{G}_m}(S^2)}\underline{\SSk}_{\mathbb{G}_m}(\mathbb{B})  \\
&\cong \mathbb{k}\otimes^{\mathbb{L}}_{\SSkalg_{\mathbb{G}_m}(S^2)}\mathbb{k}\\
&\cong \mathbb{k}\otimes^{\mathbb{L}}_{\mathbb{k}[\varepsilon]/(\varepsilon^2)}\mathbb{k}
\end{align*}
Here $\varepsilon$ has homological degree 1. 

It's well-known that $\mathbb{k}[\varepsilon]/(\varepsilon^2)$ is a Koszul algebra, and there exists an infinite length projective resolution of $\mathbb{k}$ by free $\mathbb{k}[\varepsilon]/(\varepsilon^2)$-modules
\[   \cdots \xrightarrow{\cdot \varepsilon} \mathbb{k}[\varepsilon]/(\varepsilon^2)[2]\xrightarrow{\cdot \varepsilon}\mathbb{k}[\varepsilon]/(\varepsilon^2)[1]\xrightarrow{\cdot \varepsilon}\mathbb{k}[\varepsilon]/(\varepsilon^2)\to \mathbb{k} \]

Then $\SSk_{\mathbb{G}_m}(S^3)$ is the total complex of the following double complex
\[   \cdots \xrightarrow{0} \mathbb{k}[2]\xrightarrow{0}\mathbb{k}[1]\xrightarrow{0}\mathbb{k} \]

Therefore the homological dimension of $\SSk_{\mathbb{G}_m}(S^3)$ is
\[  \Sigma_{i\in \mathbb{N}^0} \dim \rmH_i(\SSk_{\mathbb{G}_m}(S^3))\cdot x^i= 1+x^2+x^4+x^6\cdots     \]
\end{expl}

\begin{expl}
For $q$ is not a root of unity, let's compute $\SSk_{\Sl_2}(S^3)$. We first compute $\SSkalg_{\Sl_2}(S^2)$. By Lemma \ref{SkalgS^2=C(g)},
\begin{equation}\label{Skalgsl2S^2}
\SSkalg_{\Sl_2}(S^2)\simeq \mathbb{k}[\varepsilon]/(\epsilon^2)    
\end{equation}
where $\varepsilon$ has homological degree 3.

Now $\mathbb{k}[\varepsilon]/(\varepsilon^2)$ is a Koszul algebra, and admits an infinite free resolution of $\mathbb{k}$:
\[   \cdots \xrightarrow{\cdot \varepsilon} \mathbb{k}[\varepsilon]/(\varepsilon^2)[6]\xrightarrow{\cdot \varepsilon}\mathbb{k}[\varepsilon]/(\varepsilon^2)[3]\xrightarrow{\cdot \varepsilon}\mathbb{k}[\varepsilon]/(\varepsilon^2)\to \mathbb{k} \]

Then $\SSk_{\Sl_2}(S^3)$ is the total complex of the following double complex:
\[   \cdots \xrightarrow{0} \mathbb{k}[6]\xrightarrow{0}\mathbb{k}[3]\xrightarrow{0}\mathbb{k}\to 0 \]

The homological dimension of $\SSk_{\Sl_2}(S^3)$ is
\[ \Sigma_{i\in\mathbb{N}^0}\dim \rmH_i(\SSk_{\Sl_2}(S^3))\cdot x^i=1+x^4+x^8+x^{12}+\cdots    \]
\end{expl}

\begin{rem}
Note that this is same as singular homology of the classifying space $B\Sl_2$ of $\Sl_2$, which we compute in the following:

There is a homological spectral sequence, the \emph{Serre spectral sequence}, with the $E^2$ page
\[ E^2_{p,q}=\rmH_p(B\Sl_2;\mathbb{Z})\otimes\rmH_q(S^3;\mathbb{Z})  \]
and which has the $E^\infty$ page (since total space, which is path space in this case, is contractible)
\[ E^\infty_{p,q}=\begin{cases} \mathbb{Z}, & p=q=0 \\ 0 & \text{otherwise} \end{cases}   \]

We use the fact that 
\[ \rmH_q(S^3;\mathbb{Z})=\begin{cases} \mathbb{Z} & q=0,3 \\ 0 & \text{otherwise}   \end{cases}   \]
compute $\rmH_p(B\Sl_2;\mathbb{Z})$.

Let's write the $E^2$ page. We denote $\rmH_p(B\Sl_2;\mathbb{Z})$ by $\rmH_p$.
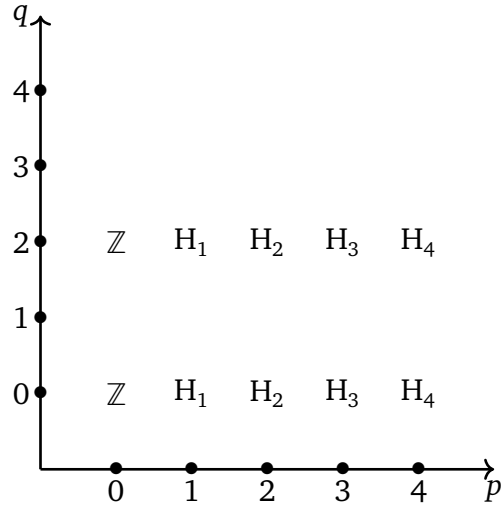
\begin{center}
\begin{tikzpicture}
\draw[line width=1,->] (0,0) -- (6,0);
\draw[line width=1,->] (0,0) -- (0,6);

\node[below] at (6,0) {$p$};
\node[left] at (0,6) {$q$};
\node[below] at (1,0) {0};
\node[] at (1,0) {$\bullet$};
\node[below] at (2,0) {1};
\node[] at (2,0) {$\bullet$};
\node[below] at (3,0) {2};
\node[] at (3,0) {$\bullet$};
\node[below] at (4,0) {3};
\node[] at (4,0) {$\bullet$};
\node[below] at (5,0) {4};
\node[] at (5,0) {$\bullet$};

\node[left] at (0,1) {0};
\node[] at (0,1) {$\bullet$};
\node[left] at (0,2) {1};
\node[] at (0,2) {$\bullet$};
\node[left] at (0,3) {2};
\node[] at (0,3) {$\bullet$};
\node[left] at (0,4) {3};
\node[] at (0,4) {$\bullet$};
\node[left] at (0,5) {4};
\node[] at (0,5) {$\bullet$};

\node[] at (1,1) {$\mathbb{Z}$};

\node[] at (2,1) {$\rmH_1$};
\node[] at (3,1) {$\rmH_2$};
\node[] at (4,1) {$\rmH_3$};
\node[] at (5,1) {$\rmH_4$};

\node[] at (1,3) {$\mathbb{Z}$};
\node[] at (2,3) {$\rmH_1$};
\node[] at (3,3) {$\rmH_2$};
\node[] at (4,3) {$\rmH_3$};
\node[] at (5,3) {$\rmH_4$};
\end{tikzpicture}
\end{center}

It is easy to see that $E^3$ page is the same with $E^2$ page. Nontrivial things happen on the $E^4$ page:
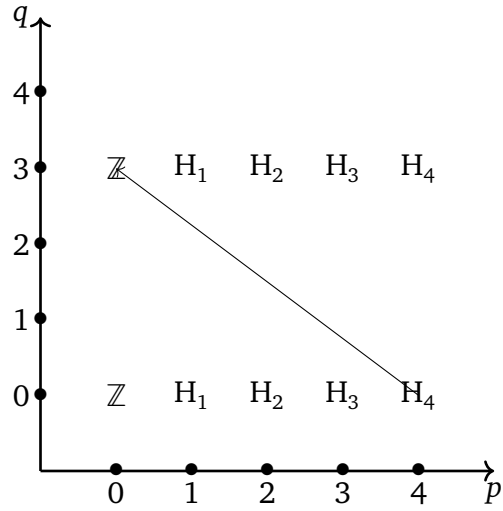
\begin{center}
\begin{tikzpicture}
\draw[line width=1,->] (0,0) -- (6,0);
\draw[line width=1,->] (0,0) -- (0,6);

\node[below] at (6,0) {$p$};
\node[left] at (0,6) {$q$};
\node[below] at (1,0) {0};
\node[] at (1,0) {$\bullet$};
\node[below] at (2,0) {1};
\node[] at (2,0) {$\bullet$};
\node[below] at (3,0) {2};
\node[] at (3,0) {$\bullet$};
\node[below] at (4,0) {3};
\node[] at (4,0) {$\bullet$};
\node[below] at (5,0) {4};
\node[] at (5,0) {$\bullet$};

\node[left] at (0,1) {0};
\node[] at (0,1) {$\bullet$};
\node[left] at (0,2) {1};
\node[] at (0,2) {$\bullet$};
\node[left] at (0,3) {2};
\node[] at (0,3) {$\bullet$};
\node[left] at (0,4) {3};
\node[] at (0,4) {$\bullet$};
\node[left] at (0,5) {4};
\node[] at (0,5) {$\bullet$};

\node[] at (1,1) {$\mathbb{Z}$};

\node[] at (2,1) {$\rmH_1$};
\node[] at (3,1) {$\rmH_2$};
\node[] at (4,1) {$\rmH_3$};
\node[] at (5,1) {$\rmH_4$};

\node[] at (1,4) {$\mathbb{Z}$};
\node[] at (2,4) {$\rmH_1$};
\node[] at (3,4) {$\rmH_2$};
\node[] at (4,4) {$\rmH_3$};
\node[] at (5,4) {$\rmH_4$};

\draw[->] (5,1) -- (1,4);

\end{tikzpicture}
\end{center}
According to the $E^\infty$ page, we need to vanish all terms except the term at $(0,0)$, which implies:
\begin{itemize}
\item the arrow in the figure is an isomorphism, i.e. $\rmH_4=\mathbb{Z}$. Iteratively, $\rmH_8,\rmH_{12},\cdots=\mathbb{Z}$.
\item $\rmH_1,\rmH_2,\rmH_3,=0$; Iteratively, $\rmH_5,\rmH_6,\rmH_7,\rmH_9,\cdots=0$.
\end{itemize}
All in all,
\[ \rmH_p(B\Sl_2;\mathbb{Z})=\begin{cases} \mathbb{Z} & p=0,4,8, 12\cdots \\ 0 & \text{otherwise}   \end{cases}   \]
Then by universal coefficient theorem for homology
\[ \rmH_p(B\Sl_2;\mathbb{k})=\begin{cases} \mathbb{k} & p=0,4,8, 12\cdots \\ 0 & \text{otherwise}   \end{cases}   \]

\end{rem}

\begin{thm}\label{skS^3}
For $q$ not a root of unity,
\[ \rmH_\bullet(\SSk_G(S^3))\cong \rmH_\bullet(BG;\mathbb{k})\cong \rmH^{G}_\bullet(\rm{pt})    \]
\end{thm}
\begin{proof}
By Koszul duality,
\[\rmH_\bullet(\mathbb{k}\otimes^\mathbb{L}_{\rmH_\bullet(G)}\mathbb{k}) \cong \rmH_\bullet(BG;\mathbb{k})    \]
\end{proof}

\begin{rem}
Since $\Loc_G(S^3)=[\bullet/G]=BG$ is smooth, the DT sheaf on it is the shifted constant sheaf $\underline{\mathbb{k}}[\dim G]$. Therefore, the above result is an evidence for the conjecture \ref{GSconj}.
\end{rem}

\begin{thm}\label{SkHH}
Let $\Sigma$ be closed oriented surface, and $\CA$ ribbon tensor category. 
\[ \SSk_\CA(\Sigma\times S^1)\simeq \HC(\SSkcat_\CA(\Sigma))    \]
\end{thm}
\begin{proof}
By the excision property of derived skein module,
\begin{align*}
\SSk_\CA(\Sigma\times S^1)&\simeq \int^{(X,Y)\in \SSkcat_\CA(\Sigma)\boxtimes \SSkcat_\CA(\Sigma)^\op} \SSk_\CA(\Sigma\times I,X,Y)\otimes^\mathbb{L}\SSk_\CA(\Sigma\times I,Y,X) \\
&\simeq \int^{X\in \SSkcat_\CA(\Sigma)}\SSk_\CA(\Sigma\times I,X,X)\\
&=\HC(\SSkcat_\CA(\Sigma))
\end{align*}

\end{proof}

\begin{thm}\label{SkGS^2timesS1}
For $q$ not a root of unity,
\[ \rmH_\bullet(\SSk_G(S^2\times S^1))\cong \rmH_\bullet(LBG)\cong \rmH^G_\bullet(G)\cong\rmH_\bullet(BG;\mathbb{k})\otimes \rmH_\bullet(G;\mathbb{k})\]
\end{thm}
\begin{proof}
The first $\simeq$ is due to \ref{SkHH} ,\ref{SkcatS2=DModBG} and combine with a classical result in string topology:
\[ \HH_\bullet(C_\bullet(G))\cong \rmH_\bullet(LBG)\] 

The second $\simeq$ is because it is a well-known folk theorem that there is a homotopy equivalence between $(G\times EG)/G$, where $G$ acts by conjugation, and the free loop space of $BG$.

The third $\cong$ can be obtained from Serre spectral sequence.
\end{proof}
\begin{rem}
Since $G$ is smooth, $\Loc_G(S^2\times S^1)=[G/G]$ is smooth. So the DT sheaf on it is a constant sheaf. Therefore, the above result is an evidence for the conjecture \ref{GSconj}.
\end{rem}

\begin{rem}\label{stableunderL}
Note that we have 
\[ \rmH^\bullet_G(X)\cong \rmH^\bullet_T(X)^W   \]
where $X$ is a smooth manifold, $T$ is a maximal torus of $G$, $W$ is Weyl group of $G$, and right hand side is $W$-invariant elements of equivariant cohomology. Under Langlands duality, maximal torus are identical (but not canonically), combined with universal coefficient theorem, we know the \ref{skS^3} and \ref{SkGS^2timesS1} are stable under Langlands duality.
\end{rem}

In the following, we will use Theorem \ref{SkHH} to prove that
\begin{thm}
Let $\CA=\rep^{\fd}_q(\mathbb{G}_m)$ for not a root of unity $q$. 
\[  \SSk_{\mathbb{G}_m}(\Sigma_g\times S^1)\cong C_\bullet(B\mathbb{G}_m)\otimes \Lambda^\bullet \mathbb{k}^{2g+1}     \]
\end{thm}

\begin{rem}\label{BR26}
Actually we know the general result for derived $\mathbb{G}_m$-skein module for any closed oriented 3-manifold and for any $q$. Joint in progress with Qiuyu Ren, we have a non-Blob type, combinatorial, global, simplicial model for derived $\mathbb{G}_m$-skein module and use that model, we can prove the following theorem: \cite{BR26} For any $q$ and any connected closed oriented 3-manifold $M$, 
\[  \SSk_{\mathbb{G}_m,q}(M)\cong C_\bullet(B\mathbb{G}_m;\mathbb{k})\otimes \mathbb{k}[\rmH_3(M)]\otimes^\mathbb{L}_{\mathbb{k}[\rmH_2(M)]}\mathbb{k}[\rmH_1(M)]     \]
Here $\mathbb{k}[\rmH_2(M)]$ acts on $\mathbb{k}[\rmH_3(M)]=\mathbb{k}$ trivially, and $\mathbb{k}[\rmH_2(M)]$ acts on $\mathbb{k}[\rmH_1(M)]$ by $(g,h)\mapsto q^{\langle 2g,h\rangle}h$. Here $\langle-,-\rangle$ is the intersection form of homology of $M$.

We are considering to generalize it to more general reductive group $G$.
\end{rem}

In order to use Theorem \ref{SkHH}, we need to prove a compact generation property for derived $\mathbb{G}_m$-skein category:
\begin{lem}\label{Gmskcatmorskalg}
Let $G=\mathbb{G}_m$, and $\Sigma$ a closed oriented surface. 
\[ \widehat{\SSkcat_{\mathbb{G}_m}(\Sigma)}\simeq \LMod_{\SSkalg_{\mathbb{G}_m}(\Sigma)}(\VVect)   \]
\end{lem}
\begin{proof}
Note that by Lemma \ref{1boxtimesVarecompactgenerators}, we have $\{\mathbb{1}_{\Sigma^\circ}\boxtimes_{\widehat{\SSkcat_\CA(\Ann)}}V\}_{V\in \rep^\fd_q(\mathbb{G}_m)}$ are compact generators of $\widehat{\SSkcat_{\mathbb{G}_m}(\Sigma)}$.

By \ref{QHR}, to prove the lemma, it's enough to show that $\mathcal{D}_q(\mathbb{G}_m)^{\otimes g}\otimes^\mathbb{L}_{\mathcal{O}_q(\mathbb{G}_m)}V^*\otimes V\simeq 0$ iff $V\notin\FZ_2(\rep_q(\mathbb{G}_m))=\Vect$. The $\mathcal{O}_q(\mathbb{G}_m)$-action on $\mathcal{D}_q(\mathbb{G}_m)^{\otimes g}$ is given by the quantum moment map
\begin{align*}  \mu:\mathcal{O}_q(\mathbb{G}_m)&\to \mathcal{D}_q(\mathbb{G}_m)^{\otimes g}  \\
x &\mapsto q^{-1}[\partial,K]=q^{-1}q=1
\end{align*}
Here the factor $q^{-1}$ is given by the ribbon structure. 

Although $V^*\otimes V\simeq\mathbb{1}$, the $\mathcal{O}_q(\mathbb{G}_m)$-action on $\mathbb{1}$ is non-trivial. It is again given by quantum moment map see Section. This quantum moment map is given by
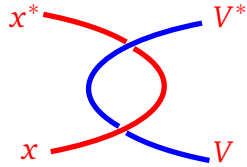
\[  \mathcal{O}_q(\mathbb{G}_m)\to \intend_{\Rep_q(\mathbb{G}_m)}(V)\cong V^*\otimes V       \]
\begin{center}
\begin{tikzpicture} 

\draw[line width=2,color=red] (0,0) arc (-60:30:3 and 1);
\draw[line width=2,color=red] (1,1.45) arc (30:60:3 and 1);

\draw[line width=2,color=blue] (2,1.7) arc (120:210:3 and 1);
\draw[line width=2,color=blue] (1,0.25) arc (210:240:3 and 1);

\node[color=red,left] at (0,0) {$x$};
\node[color=red,left] at (0,1.8) {$x^*$};

\node[color=red,right] at (2,0) {$V$};
\node[color=red,right] at (2,1.8) {$V^*$};
\end{tikzpicture}
\end{center}
Since $V\notin \FZ_2(\rep_q(G))$, the double braiding with defining representation $x$ is $\lambda\ne 1$. In other words, the $\mathcal{O}_q(\mathbb{G}_m)$-action on $V^*\otimes V\simeq \mathbb{1}$ is given by 
\begin{align*}
\mathcal{O}_q(\mathbb{G}_m)&\to \mathbb{1}\\
x &\mapsto \lambda
\end{align*}

The \emph{quantized Koszul complex} of $\mathbb{1}$ as $\mathcal{O}_q(\mathbb{G}_m)$-modules is given as
\[ 0\to \mathcal{O}_q(\mathbb{G}_m)\xrightarrow{\cdot (x-\lambda)} \mathcal{O}_q(\mathbb{G}_m)\to 0     \]
Then the derived tensor product $\mathcal{D}_q(\mathbb{G}_m)^{\otimes g}\otimes^\mathbb{L}_{\mathcal{O}_q(\mathbb{G}_m)}V^*\otimes V$ is the following complex
\[ 0\to \mathcal{D}_q(\mathbb{G}_m)^{\otimes g}\xrightarrow{\cdot(1-\lambda)}\mathcal{D}_q(\mathbb{G}_m)^{\otimes g}\to 0    \]
which is homotopic to 0.
\end{proof}
\begin{rem}
When $\Sigma=S^2$, this lemma coincides with Lemma \ref{skcatS^2=Modskalg}.
\end{rem}

\begin{cor}\label{GmSk}
Let $G=\mathbb{G}_m$, and $\Sigma_g$ a closed oriented surface of genus $g$. We have 
\[  \SSk_{\mathbb{G}_m}(\Sigma_g\times S^1)\simeq \HC(\SSkalg_{\mathbb{G}_m}(\Sigma_g))   \]
\end{cor}
\begin{proof}
Combine Lemma \ref{Gmskcatmorskalg} and Theorem \ref{SkHH}.
\end{proof}

By Theorem \ref{QHR}, Theorem \ref{semisimpleintskalg} and Example \ref{LskalgintG}, we know that 
\[\SSkalg_{\mathbb{G}_m}(\Sigma_g)\cong \Hom_{\Rep_q(G)}(\mathbb{k}, \mathcal{D}_q(\mathbb{G}_m)^{\otimes g}\otimes^\mathbb{L}_{\mathcal{O}_q(\mathbb{G}_m)}\mathbb{k}) \]   

Our goal now is to compute $\SSkalg_{\mathbb{G}_m}(\Sigma_g)$ first.

The quantum moment map which is different from the quantum moment map $\mathcal{O}_q(\mathbb{G}_m)\to \mathbb{1}, x\mapsto \lambda$ in the proof of \ref{Gmskcatmorskalg}  
\begin{align*}
\epsilon:\mathcal{O}_q(\mathbb{G}_m)&\to \mathbb{k}\\
x&\mapsto 1
\end{align*}
This gives $\mathbb{k}$ a $\mathcal{O}_q(\mathbb{G}_m)$-module structure.
There is a free resolution, called \emph{quantized Koszul complex}, of $\mathbb{k}$ as $\mathcal{O}_q(\mathbb{G}_m)$-modules:
\[ 0\to \mathcal{O}_q(\mathbb{G}_m)\xrightarrow{\cdot (x-1)}\mathcal{O}_q(\mathbb{G}_m)\to 0     \]

There is another quantum moment map which is the same as $\mu$ in the proof of \ref{Gmskcatmorskalg} 
\begin{align*} \mu:\mathcal{O}_q(\mathbb{G}_m)&\to \mathcal{D}_q(\mathbb{G}_m)^{\otimes g}\\  
x&\mapsto 1
\end{align*}

Thus $\SSkalg_{\mathbb{G}_m}(\Sigma^\circ_g)\otimes^{\mathbb{L}}_{\SSkalg^{\Int}_{\mathbb{G}_m}(\Ann)}\mathbb{k}$ is the following chain complex:
\[ 0\to \mathcal{D}_q(\mathbb{G}_m)^{\otimes g}\xrightarrow{\cdot (1-1)}\mathcal{D}_q(\mathbb{G}_m)^{\otimes g}\to 0     \]

Therefore, we have 
\begin{equation}\label{SkalgGm(Sigmag)}
\SSkalg_{\mathbb{G}_m}(\Sigma_g)\cong\Hom_{\Rep_q(\mathbb{G}_m)}(\mathbb{k},\SSkalg^\Int_{\mathbb{G}_m}(\Sigma^\circ_g)\otimes^{\mathbb{L}}_{\SSkalg^\Int_{\mathbb{G}_m}(\Ann)}\mathbb{k})\cong \mathcal{D}_q(\mathbb{G}_m)^{\otimes g}\otimes \mathbb{k}[\varepsilon]/(\varepsilon^2)\end{equation}
where $\varepsilon$ has homological degree 1.

Now we compute the homological dimension of $\SSk_{\mathbb{G}_m}(\Sigma_g\times S^1)$. By Corollary \ref{GmSk}, we need to compute Hochschild homology of $\SSkalg_{\mathbb{G}_m}(\Sigma_g)\cong \mathcal{D}_q(\mathbb{G}_m)^{\otimes g}\otimes \mathbb{k}[\varepsilon]/(\varepsilon^2)$.  

To compute it, we first need to compute the homological dimension of $\HC(\mathcal{D}_q(\mathbb{G}_m)^{\otimes g})$. By \cite{Wam97} Theorem 1.1, we have 
\[ \Sigma_{i\in\mathbb{N}^0} \dim \HH_i(\mathcal{D}_q(\mathbb{G}_m)^{\otimes g})\cdot x^i=(1+x)^{2g}      \]

Now let's compute $\HH(\mathbb{k}[\varepsilon]/(\varepsilon^2))$ where $\varepsilon$ has homological degree 1.
\begin{prop}
[Proposition 5.4.6, \cite{Lod13}] Let $(A,\delta)$ be a commutative DG algebra such that either $A=\Lambda V$ or $\mathbb{k}$ contains $\mathbb{Q}$ and $A$ is smooth. Then the Hochschild complex $\HC(A,\delta)$ is quasi-isomorphic to $\Omega_{(A,\delta)}$ and so
\[ \HH_\bullet(A,\delta)\cong \rmH_\bullet(\Omega_{(A,\delta)})  \]
where $\Omega^i_{(A,\delta)}:=\overline{\Lambda}_A \Omega^1_{(A,\delta)}$. Here $I$ is the kernel of multiplication map of $A$, and $\Omega^1_{(A,\delta)}:=I/I^2$.
\end{prop}
In our case, $(A,\delta)=(\mathbb{k}[\varepsilon]/(\varepsilon^2),0)$. 

Now we first compute $\Omega^1_{(A,0)}$. The multiplication map $\mu:A\otimes A\to A, a\otimes b\mapsto a\cdot b$. Then $I=\ker \mu$ is generated by $1\otimes \varepsilon-\varepsilon\otimes 1$ and $\varepsilon\otimes \varepsilon$. We have $I^2=0$. So $\Omega^1_{(A,0)}=I/I^2=I$, which is generated by $d\varepsilon$ and $\varepsilon d\varepsilon$ (note that $d\varepsilon$ has homological degree 1). So we can write $\Omega^1_{(A,0)}$ as $A\cdot d\varepsilon$. Then $\overline{\Lambda}_A \Omega^1_{(A,0)}=A\otimes \overline{\Lambda} \mathbb{k}\cdot d\varepsilon \cong A\otimes \Lambda \mathbb{k}\cdot d\varepsilon[1]\cong A\otimes \Sym \mathbb{k}[2]\cong A\otimes \mathbb{k}[u]$ where $u$ has homological degree 2. This leads
\[ \Sigma_{i\in\mathbb{N}^0}\dim \HH_i(\mathbb{k}[\varepsilon]/(\varepsilon^2))\cdot x^i= (1+x)(1+x^2+x^4+x^6+\cdots)    \]

Then by K\"unneth formula, the homological dimension of $\SSk_{\mathbb{G}_m}(\Sigma_g\times S^1)$ is
\[ \Sigma_{i\in \mathbb{N}^0} \dim \rmH_i(\SSk_{\mathbb{G}_m}(\Sigma_g\times S^1))\cdot x^i=(1+x)^{2g+1}(1+x^2+x^4+x^6+\cdots)    \]

\begin{rem}
In particular, $\rmH_0(\SSk_{\mathbb{G}_m}(\Sigma_g\times S^1))=\mathbb{k}$. This coincides with the torsion part of first homology $\rmH_1(\Sigma_g\times S^1,\mathbb{Z})$. Here $\rmH_1(\Sigma_g\times S^1,\mathbb{Z})=\mathbb{Z}^{2g+1}$ which is a free abelian group. So its torsion part is trivial. Thus the free vector space generated by torsion part is 1-dimensional.
\end{rem}

\begin{rem}
The above result also coincides with Remark \ref{BR26}.
\end{rem}

\begin{expl}
We compute $\SSk_{\Sl_2}(S^2\times S^1)$. We need to compute $\HH_\bullet(\SSkalg_{\Sl_2}(S^2))$. By Equation \ref{Skalgsl2S^2}, we need to compute $\HH_\bullet(\mathbb{k}[\varepsilon]/(\varepsilon^2))$ where $\varepsilon$ has homological degree 3.

By the same method above, the homological dimension of $\SSk_{\Sl_2}(S^2\times S^1)$ is
\[ \Sigma_{i\in \mathbb{N}^0} \dim \rmH_i(\SSk_{\Sl_2}(S^2\times S^1))\cdot x^i=(1+x^3)(1+x^4+x^8+x^{12}+\cdots)    \]
\end{expl}
\begin{rem}
Note that this is same with $\rmH_\bullet^G(G)\cong\rmH_\bullet(LBG)\cong \rmH_\bullet(BG)\otimes\rmH_\bullet(G)$\footnote{It is a well-known folk theorem that there is a homotopy equivalence between $(G\times EG)/G$, where $G$ acts by conjugation, and the free loop space of $BG$} for $G=\Sl_2$. (Theorem \ref{SkGS^2timesS1}).
\end{rem}

\begin{rem}\label{smomth}
All the examples in this paper are "smooth" examples. In order to get more non-trivial evidents for conjectures or have a more deep understanding for derived skein module, a computation for non-smooth example is necessary. 

The simplest non-smooth example should be $\SSk_{\Sl_2}(T^3)$ or $\SSk_{\Gl_2}(T^3)$. For ordinary skein module case, the dimension of $\Sk_{\Sl_N}(T^3)$ and $\Sk_{\Gl_N}(T^3)$ was computed by Gunningham-Jordan-Vazirani in \cite{GJV24}, also see \cite{Car17}\cite{Gil18} for a topological method when $G=\Sl_2$.  One can choose the well-known genus-3 Heegaard splitting of $T^3=H_3^\gamma\cup_{\Sigma_3} H_3$. The inter handlebody contractible loops on surface $\Sigma$ are $b$-loops $b_1,b_2,b_3$. There are several steps:
\begin{enumerate}
\item understand what the outer handlebody contractible loops on surface are. Once one figures this out, the $\Sk^\Int_G(H_3^\gamma)$ is just $\mathcal{D}_q(G)^{\otimes 3}$ quotient by those contractible loops;
\item find a projective resolution of $\mathcal{O}_q(G)$ as $\mathcal{D}_q(G)$-modules. One can construct this projective resolution by taking twisted tensor product of the Koszul resolution of $\mathbb{k}$ as $\mathcal{O}_q(G)$-modules with $\mathcal{O}_q(G)$. In \cite{JR25}, despite for a different motivation, they construct a Koszul complex of $\mathbb{k}$ as $\mathcal{O}_q(\Sl_2)$-modules, which will be useful here.

\item compute the projective resolution of $\mathcal{O}_q(G)^{\otimes 3}$ as $\mathcal{D}_q(G)^{\otimes 3}$-modules. 
\item compute $\SSk_G(T^3)$ by \ref{computableexcisionforrepqG}.
\end{enumerate}
\end{rem}

\section{Deformation quantization modules and finiteness property}\label{DQmoduleandfinite}

In this section, we are closely following \cite{GJS22} for the proof. We prove (Theorem \ref{finiteness1}) the derived relative tensor product 
$\Sk^\Int_G(H_1)\otimes^{\mathbb{L}}_{\Skalg^\Int_G(\Sigma^\circ)} \Sk^\Int_G(H_2)$ is bounded complex with finite-dimensional homologies for $q=e^{\hbar}$, where $\hbar$ is a formal parameter, and hence for generic $q$. And also $\SSk_G(M)$ is a complex with finite-dimensional homologies for $q=e^{\hbar}$ (Theorem \ref{finiteness}) by analytic and algebraic methods.

\begin{defn}
Let $M$ be a cochain complex of $\mathbb{C}[\![\hbar]\!]$-modules. It is \emph{cohomologically complete} if
\[ \RHom(\mathbb{C}(\!(\hbar)\!),M)=0.     \]
\end{defn}

\begin{lem}\label{cohocompleteislevelwise}
Let $M$ be a complex of $\mathbb{C}[\![\hbar]\!]$-modules. $M$ is cohomologically complete iff $\rmH^i(M)$ is cohomologically complete for each $i\in \mathbb{Z}$.
\end{lem}
\begin{proof}
This is well-known. See Proposition 6.6.2 (3) in \cite{DC}.
\end{proof}

\begin{lem}\label{truncation}
Let $M$ be a cochain complex of $\mathbb{C}[\![\hbar]\!]$-modules. If we have $\RHom(\mathbb{C}(\!(\hbar)\!),M)=0$, then we have $\RHom(\mathbb{C}(\!(\hbar)\!),\tau_{\le n}M)=0$ for every $n$. Here $\tau_{\le n}M$ denotes the standard cohomological truncation.
\end{lem}
\begin{proof}
Since $\rmH^i(\tau_{\le n}M)=\rmH^i(M)$ for $i\le n$ and $\rmH^i(\tau_{\le n}M)=0$ for $i> n$, by Lemma \ref{cohocompleteislevelwise}, we get the statement.
\end{proof}

Let $X$ be a smooth affine Poisson scheme and $L_1,L_2\subset X$ be smooth Lagrangian subschemes. Here, by a Lagrangian subscheme of a Poisson scheme we will mean a subscheme of an open symplectic leaf which is Lagrangian there. In addition, fix their deformation quantizations:
\begin{itemize}
\item Let $A$ be a $\hbar$-complete $\mathbb{C}[\![\hbar]\!]$-algebra without $\hbar$-torsion which is a deformation quantization of $\mathcal{O}(X)$.
\item Let $M_2$ be a cyclic left $A$-module without $\hbar$-torsion which is a deformation quantization of $\mathcal{O}(L_2)$.
\item Let $M_1$ be a cyclic right $A$-module without $\hbar$-torsion which is a deformation quantization of $\mathcal{O}(L_1)$.
\end{itemize}

Let the input ribbon category be $\CA=\rep^\fd_{\hbar}(G)$ the category of representations of the quantum group over $\mathbb{k}=\mathbb{C}[\![\hbar]\!]$, where each representation is a free $\mathbb{k}$-module of finite rank. In this section, from now on we will drop the subscript $\rep^{\fd}_{\hbar}(G)$ from our notations for skein modules and skein categories. Let $M=H_1\cup_{\Sigma} H_2$ be a Heegaard splitting of $M$.

\begin{prop}\label{DQSkalg}
\cite{GJS22} Let $X$ be the Poisson scheme $G^{2g}$ with the Fock-Rosly Poisson structure. $\Skalg^\Int(\Sigma^\circ)$ is a flat deformation quantization of $\mathcal{O}(X)$ which is $\hbar$-complete $\mathbb{C}[\![\hbar]\!]$-algebra without $\hbar$-torsion.
\end{prop}

\begin{prop}\label{DGSkN_2}
\cite{GJS22} $\Sk^\Int(H_2)$ is a deformation quantization of $L_2=G^g\subset G^{2g}$ which is a cyclic left $\Skalg^\Int(\Sigma^\circ)$-module which is $\hbar$-complete and without $\hbar$-torsion.
\end{prop}
\begin{prop}\label{DGSkN_1}
\cite{GJS22} $\Sk^\Int(H_1)$ is a deformation quantization of $L_1=G^g\subset G^{2g}$ which is a cyclic right $\Skalg^\Int(\Sigma^\circ)$-module which is $\hbar$-complete and without $\hbar$-torsion.
\end{prop}

We want to first prove the derived strong finiteness theorem \ref{finiteness1} which is the internal derived skein module $\SSk^\Int_G(M-\mathbb{B})$ is a bounded complex with finite-dimensional homologies over $\mathbb{Q}(A)$. In order to prove this, we would like to prove \ref{localfiniteness} first. Note that \ref{localfiniteness} is nothing about boundedness.

After choosing an isomorphism of vector spaces $A\cong \mathcal{O}(X)[\![\hbar]\!]$, the multiplication on $A$ is given by a power series of bidifferential operators. Because a differential operator can only lower the order of vanishing along a subvariety by a finite amount, this multiplication naturally extends to the formal completion $\mathcal{O}(\widehat{X}_{L_1})$ of the structure sheaf along $L_1$. The resulting completed algebra is denoted $\widehat{A}$. Thus we have an inclusion $A\subset \widehat{A}$, and $\widehat{A}$ is a deformation quantization of the completed symplectic variety $\widehat{X}_{L_1}$ (the formal neighborhood of $L_1$ in $X$).

\begin{lem}
$\widehat{A}$ is a flat $A$-module.
\end{lem}
\begin{proof}
Since $A$ is Noetherian, the ring map from $A$ to its completion $\widehat{A}$ with respect to any ideal $I$ of $A$ is flat.
\end{proof}

Recall that the $A$-module $M_1$ is cyclic, i.e. we have a surjection $A\to M_1$. In particular, the $\mathcal{O}(X)[\![\hbar]\!]$-module structure on $\mathcal{O}(L_1)[\![\hbar]\!]$ is also given by a bidifferential operator. Therefore, the $A$-module structure on $M_1$ extends to an $\widehat{A}$-module structure. Define $\widehat{M_2}=M_2\otimes^\mathbb{L}_A\widehat{A}\cong M_2\otimes_A \widehat{A}$ which is a finitely generated $\widehat{A}$-module. Then
\[ M_2\otimes^\mathbb{L}_A M_1\cong M_2\otimes^\mathbb{L}_A \widehat{A}\otimes^{\mathbb{L}}_{\widehat{A}} M_1=  \widehat{M_2}\otimes^{\mathbb{L}}_{\widehat{A}}M_1    \]

Following \cite{KS10} Section 7.1,  consider the composite map
\[  \hbar^{-1}\widehat{A}\xrightarrow{\hbar \cdot} \widehat{A}\twoheadrightarrow \widehat{A}/\hbar\widehat{A}\cong \mathcal{O}(\widehat{X}_{L_1})\twoheadrightarrow \mathcal{O}(L_1)   \]
where the first arrow is multiplication by $\hbar$ (so an element $a/\hbar$ is sent to $a$), the second is the quotient by $\hbar \widehat{A}$, and the last is the restriction map to $L_1$ (the zero-section of $\widehat{T}^*L_1\cong \widehat{X}_{L_1}$). Define $J$ to be the Kernel of the above map. Now $\widehat{A}_{L_1}$ is the $\mathbb{C}[\![\hbar]\!]$-subalgebra of $\widehat{A}[\hbar^{-1}]$ generated by $J$: $\widehat{A}_{L_1}=\mathbb{C}[\![\hbar]\!]\langle J\rangle\subset \widehat{A}[\hbar^{-1}]$.
\begin{prop}
(\cite{GJS22} Prop 3.7) The inclusion $\widehat{A}_{L_1}\subset \widehat{A}[\hbar^{-1}]$ induces an equality
\[ \widehat{A}_{L_1}[\hbar^{-1}]=\widehat{A}[\hbar^{-1}].   \]
\end{prop}

We will use the following "cohomologically complete Nakayama" theorem inductively to prove \ref{bounded1impliesbounded2}:
\begin{thm}\label{Nakayama}
(\cite{PSY13} Theorem 0.2; \cite{KS10} Theorem 1.6.4) Let $M$ be a cohomologically complete complex of $\mathbb{C}[\![\hbar]\!]$-modules, such that $\rmH^i(M)=0$ for $i>0$, and such that $\rmH^{0}(\mathbb{C}\otimes^{\mathbb{L}}_{\mathbb{C}[\![\hbar]\!]}M)$ is finitely generated over $\mathbb{C}$. Then $\rmH^{0}(M)$ is finitely generated as a $\mathbb{C}[\![\hbar]\!]$-module.
\end{thm}

\begin{thm}\label{bounded1impliesbounded2}
Suppose $N_1$ is finitely generated left $\widehat{A}_{L_1}$-module without $\hbar$-torsion (hence cohomologically complete) together with an isomorphism of left $\widehat{A}[\hbar^{-1}]$-modules $N_1[\hbar^{-1}]\cong M_1[\hbar^{-1}]$ and similarly for $N_2$. Suppose \[N_2/\hbar\otimes^{\mathbb{L}}_{\widehat{A}_{L_1}/\hbar}N_1/\hbar\] is a  complex with finite dimensional homologies over $\mathbb{C}$. Then \[(\widehat{M_2}\otimes^\mathbb{L}_{\widehat{A}}M_1)[\hbar^{-1}]\cong (\widehat{M}_2\otimes^\mathbb{L}_{\widehat{A}_{L_1}}M_1)[\hbar^{-1}]\] is a complex with finite-dimensional homologies over $\mathbb{C}(\!(\hbar)\!)$.
\end{thm}
\begin{proof}
By \cite{GJS22} Prop 3.7, we have 
\[(\widehat{M_2}\otimes^\mathbb{L}_{\widehat{A}}M_1)[\hbar^{-1}]\cong (\widehat{M}_2\otimes^\mathbb{L}_{\widehat{A}_{L_1}}M_1)[\hbar^{-1}]\]
We want to show that $(\widehat{M}_2\otimes^\mathbb{L}_{\widehat{A}}M_1)[\hbar^{-1}]$ is a  complex with finite-dimensional homologies over $\mathbb{C}(\!(\hbar)\!)$. Since we have
\[ (\widehat{M}_2\otimes^\mathbb{L}_{\widehat{A}}M_2)[\hbar^{-1}]\cong (\widehat{M_2}[\hbar^{-1}]\otimes^\mathbb{L}_{\widehat{A}[\hbar^{-1}]}M_1[\hbar^{-1}])\cong (N_2[\hbar^{-1}]\otimes^{\mathbb{L}}_{\widehat{A}[\hbar^{-1}]}(N_1[\hbar^{-1}]))\cong (N_2\otimes^\mathbb{L}_{\widehat{A}_{L_1}}N_1)[\hbar^{-1}]    \]
So it's enough to show that $N_2\otimes^{\mathbb{L}}_{\widehat{A}_{L_1}}N_1$ is a complex with finitely generated homologies over $\mathbb{C}[\![\hbar]\!]$. 

The derived tensor product $K:=N_2\otimes^\mathbb{L}_{\widehat{A}_{L_1}}N_1$ is cohomologically complete \cite{KS10}\cite{GJS22}. We also have $K_0:=\mathbb{C}\otimes^{\mathbb{L}}_{\mathbb{C}[\![\hbar]\!]}K= \mathbb{C}\otimes^\mathbb{L}_{\mathbb{C}[\![\hbar]\!]}(N_2\otimes^{\mathbb{L}}_{\widehat{A}_{L_1}}N_1)\cong (N_2/\hbar)\otimes^\mathbb{L}_{\widehat{A}_{L_1}/\hbar}(N_1/\hbar)$ is a  complex with finite-dimensional homologies over $\mathbb{C}$. Assume $\rmH^i(K_0)=0$ for $i< a\ \text{and  }  i>0$. Recall $\rmH^i(M[-n])=\rmH^{i+n}(M)$. $\tau_{\le m}M$ means the truncation killing all cohomology above degree $m$.

We have $\rmH^0(K_0)$ is finite-dimensional over $\mathbb{C}$. We also know $\rmH^i(K)=0$ for $i>0$, So by \ref{Nakayama}, we have $\rmH^0(K)$ is finitely generated over $\mathbb{C}[\![\hbar]\!]$. 

We want to show $\rmH^{-1}(K)=\rmH^0(K[1])=\rmH^0((\tau_{\le-1}K)[1] )$ is finitely generated over $\mathbb{C}[[\hbar]]$. Truncation (Lemma \ref{truncation}) and shift (obviously) preserves cohomologically completeness. So we have $\tau_{\le -1}K[1]$ is still a cohomologically complete complex. We have $\rmH^i(\tau_{\le -1}K[1])=0$ for all $i>0$. Since $\rmH^0(K_0[1])$ is finitely generated, $\rmH^0(\mathbb{C}\otimes^\mathbb{L}_{\mathbb{C}[\![\hbar]\!]}(\tau_{\le -1}K)[1])$, as a submodule of $\rmH^0(K_0[1])$, is still finitely generated. Thus $\rmH^0((\tau_{\le 1}K)[1])$ is finitely generated.

Inductively, we know $K$ is a complex with finitely generated homologies over $\mathbb{C}[\![\hbar]\!]$.
\end{proof}

Note that the above theorem is nothing about boundedness. The following theorem is also not.

\begin{thm}\label{localfiniteness}
The localizations
\[ (M_2\otimes^\mathbb{L}_A M_1)[\hbar^{-1}]   \]
is a complex with finite-dimensional homologies over $\mathbb{C}(\!(\hbar)\!)$.
\end{thm}
\begin{proof}
We know $M_2\otimes^\mathbb{L}_{A}M_1\cong \widehat{M_2}\otimes^\mathbb{L}_{\widehat{A}}M_1$. Thus it enough to show that $\widehat{M_2}\otimes^\mathbb{L}_{\widehat{A}}M_1[\hbar^{-1}]$ is a  complex with finite-dimensional homologies over $\mathbb{C}(\!(\hbar)\!)$.

Let $D(L_1)$ be the $\mathbb{C}$-algebra of differential operators. It admits a filtration given by the order of the differential operator. Consider the Rees algebra which is a graded $\mathbb{C}[\hbar]$-algebra and complete it in the $\hbar$-adic topology as well as with respect to the order filtration. We denote the completion by $\widehat{D}_{\hbar}(L_1)$. In the proof of Theorem 3.6 of \cite{GJS22}, they show that 
\[  N_2/\hbar\otimes^\mathbb{L}_{\widehat{A}_{L_1}/\hbar}N_1/\hbar \cong N_2/\hbar\otimes^\mathbb{L}_{D(L_1)}N_1/\hbar     \]
are (bounded) complex with finite-dimensional homologies over $\mathbb{C}$. Then by Theorem \ref{bounded1impliesbounded2}, we get $\widehat{M_2}\otimes^{\mathbb{L}}_{\widehat{A}}M_1[\hbar^{-1}]$ is a complex with finite-dimensional homologies over $\mathbb{C}(\!(\hbar)\!)$.
\end{proof}

\begin{thm}\label{finiteness1}
Let $\CA=\rep^\fd_q(G)$ for $q$ generic. Let $M$ be closed with a Heegaard splitting $M=H_1\cup_{\Sigma} H_2$. Then 
\[\SSk^\Int_G(M-\mathbb{B})\cong \Sk^\Int_G(H_1)\otimes^\mathbb{L}_{\Skalg^\Int_G(\Sigma^\circ)}\Sk^\Int_G(H_2)\]
is a bounded complex with finite-dimensional homologies over $\mathbb{Q}(q^{1/d})$.
\end{thm}
\begin{proof}
The homological dimension of the complex $\SSk^\Int_G(M-\mathbb{B})\otimes_{\mathbb{Z}[q^{1/d},q^{-1/d}]}\mathbb{Q}(q^{1/d})$ over $\mathbb{Q}(q^{1/d})$ coincides with the homological dimension of the $\SSk^\Int_G(M-\mathbb{B})\otimes_{\mathbb{Z}[q^{1/d},q^{-1/d}]}\mathbb{C}(\!(\hbar)\!)$ where $q=e^{\hbar}$. By Prop \ref{DQSkalg} \ref{DGSkN_2} \ref{DGSkN_1}, and \ref{localfiniteness}, we get $\SSk^\Int_G(M-\mathbb{B})$ is a complex with finite-dimensional homologies over $\mathbb{Q}(q^{1/d})$.

We still need to prove that $\SSk^\Int_G(M-\mathbb{B})\cong \SSk^\Int_G(H_1)\otimes^\mathbb{L}_{\mathcal{D}_q(G)^g}\mathcal{O}_q(G)^g$ is bounded. Since $\mathcal{O}_q(G)$ has finite global dimension, we have a bounded free resolution (i.e. quantum Koszul complex) $K$ of $\mathbb{k}$ as $\mathcal{O}_q(G)$-modules. So we also have a bounded free resolution $K^g$ of $\mathbb{k}^g$ as $\mathcal{O}_q(G)^g$-module. Note that $\mathcal{D}_q(G)\cong \mathcal{O}_q(G)\tilde{\otimes}\mathcal{O}_q(G)$. Here $\tilde{\otimes}$ is twisted tensor product. After taking the twisted tensor product with $\mathcal{O}_q(G)^g$, we still have a bounded resolution $\mathcal{O}_q(G)^g\tilde{\otimes} K$ of $\mathcal{O}_q(G)^g$ as $\mathcal{D}_q(G)^g$-modules. Therefore, we have $\SSk^\Int_G(M-\mathbb{B})$ is bounded.
\end{proof}

\begin{rem}
This implies that the property that the homology of derived skein module can have infinite length comes from filling the 3-handle $\mathbb{B}$ back.
\end{rem}

Now we want to prove a (weak) finiteness theorem for derived skein module:
\begin{thm}\label{finiteness}
Let $G$ be a reductive algebraic group, and $q$ generic. The derived $G$-skein module $\SSk_G(M)$ of closed oriented 3-manifold is a complex with  finite-dimensional homologies over $\mathbb{Q}(A)$.
\end{thm}
\begin{proof}
It's well-known that $\rmH_\bullet(G)$ is an exterior algebra on a finite set of odd-degree generators:
\[ \rmH_\bullet(G)\cong \Lambda^\bullet(x_1,\cdots, x_r),\ \ \ \deg x_i=2d_i-1\ \  (d_i\ge 1)    \]
Since $\Lambda^\bullet(x_1,\cdots,x_r)$ is Koszul algebra, a free resolution of $\mathbb{k}$ as  $\Lambda^\bullet(x_1,\cdots,x_r)$-modules can be given by the Koszul complex. Let $V=\langle x_1,\cdots,x_r \rangle$ free vector space. Then the Koszul resolution of $\mathbb{k}$ over $\Lambda(V)$ is
\[ P_\bullet=\Lambda^\bullet (V)\otimes S^\bullet(V[-1])   \]
where $S(V[-1])$ denotes the symmetric algebra on the shifted vector space $V[-1]$ (so its generators $y_i$ have degree $\deg y_i=\deg x_i-1$ which is even). This also means $S^\bullet(V[-1])$ is polynomial algebra with each homogeneous component $S^n(V[-1])$ finite-dimensional.

We want to use \ref{computableexcisionforrepqG} to prove. Due to theorem \ref{finiteness1}, $\Sk^\Int_G(H_1)\otimes^\mathbb{L}_{\Skalg^\Int_G(\Sigma^\circ)}\Sk^\Int_G(H_2)$ is (bounded) complex with finite dimensional homologies, then after tensoring with the Koszul complex above, we have $\SSk_G(M)\cong\Sk^\Int_G(H_1)\otimes^\mathbb{L}_{\Skalg^\Int_G(\Sigma^\circ)}\Sk^\Int_G(H_2)\otimes^{\mathbb{L}}_{\rmH_\bullet(G)}\mathbb{k}$ is a complex with finite-dimensional homologies.
\end{proof}
\begin{rem}
According to the computation results in Section \ref{computeableformulaandresults}, the homology of derived skein module can be infinite-length.
\end{rem}

\section{Conjectures and Expectations}\label{ConjandExpect}

\subsection{Derived skein module and cohomological Donaldson-Thomas invariants }\label{SkeinDT}

We denote the derived character stack, i.e. the derived stack of local systems on closed oriented 3-manifold $M$, to be $\LLoc_G(M)$. It admits a (-1)-shifted symplectic structure. So its classical truncation $\Loc_G(M)$ admits a d-critical stack structure by \cite{BBBj15} Theorem 3.12. Moreover, in \cite{NS23}, Naef-Safronov define orientation data on $\Loc_G(M)$. By \cite{BBBj15} Theorem 4.8, since $\Loc_G(M)$ is a d-critical stack with an orientation, then there is a perverse sheaf $\phi_{\Loc_G(M)}$ on it. Its hypercohomology is the so-called \emph{cohomological DT invariants} $\rmH^\bullet(\Loc_G(M),\phi_{\Loc_G(M)})$ of $\Loc_G(M)$. It was known that $\rmH^\bullet(\Loc_G(M),\phi_{\Loc_G(M)})$ gives a deformation quantization of $\Loc_G(M)$ by Pridham's work \cite{Pri19}.

Pick a point $x\in M$; we can define a map $\Loc_G(M)\to BG$,  then we call $R_G(M):=\Loc_G(M)\times_{BG}\pt$ the \emph{representation variety} which also carries a natural structure of d-critical loci. Define $\phi_{R_G(M)}$ to be the pullback of $\phi_{\Loc_G(M)}$ to $R_G(M)$. In \cite{AM19}, Abouzaid-Manolescu call the hypercohomology $\rmH^\bullet(R_G(M),\phi_{R_G(M)})$ of $\phi_{R_G(M)}$ \emph{framed complexified instanton Floer homology}. 

This object is closely related to $\SSk^\Int_G(M-\mathbb{B})$, see \ref{Sk^intGM-B}. The representation variety $R_G(M)$ is an intersection $L_1\cap L_2$ of two complex algebraic Lagrangian submanifolds $L_1, L_2$ (representation varieties of the handlebodies $H_1,H_2$) in a complex symplectic subvariety $\Loc_G(\Sigma\# T^2)$ of twisted flat $G$-connections on $\Sigma\# T^2$. In Section \ref{DQmoduleandfinite}, we have seen that $\SSk^\Int_G(M-\mathbb{B})$ is a derived relative tensor product of \emph{algebraic} deformation quantization modules $\Sk^\Int_G(H_i)$ quantizing $L_i\subset \Loc_G(\Sigma\# T^2)$ over the \emph{algebraic} deformation quantization algebra $\Skalg^\Int_G(\Sigma^\circ)$. In \cite{GS23} Theorem A, Gunningham-Safronov, who work in analytic setting, prove a similar derived relative tensor product $(\Sk^\Int_G(H_1)\otimes^\mathbb{L}_{\Skalg^\Int_G(\Sigma^\circ)}\Sk^\Int_G(H_2))^{\rm{analytic}}$ is isomorphic to an analytic perverse sheaf $\phi^{\rm{analytic}}_{R_G(M)}$. Thus conjecture \ref{GSconj1} naturally appears.
\begin{rem}
One would expect that $\Sk^\Int_G(H_1)\otimes^\mathbb{L}_{\Skalg^\Int_G(\Sigma^\circ)}\Sk^\Int_G(H_2)\cong \phi_{R_G(M)}$. Then as a corollary, the G-equivariant homology of $\SSk^\Int_G(M-\mathbb{B})$ is equivalent to cohomological DT invariants $\rmH^\bullet(\Loc_G(M),\phi_{\Loc_G(M)})$.
\end{rem}

A question is: do we have $\rmH_\bullet(\SSk_G(M))\cong \rmH^\bullet(R_G(M),\phi_{R_G(M)})$? According to Theorem \ref{computableexcisionforrepqG}, the answer is \text{No} since we still have to fill the 3-ball $\mathbb{B}$ back. 

It was shown by Bullock \cite{Bul97} and Przytycki and Sikora \cite{PS00} that the $A=-1$ specialization of the skein module $\Sk(M)$ is isomorphic to the algebra of functions $\mathcal{O}(\Loc_{\Sl_2}(M))$ on the character variety (equivalently, character stack). So derived skein module should be viewed as a deformation (actually BV) quantization of character stack rather than representation variety. 

Then we need to choose a good cohomology theory. Let $\CY$ be an algebraic stack. Denote by $\Shv(\CY)$ be the (compactly generated) DG category of ind-constructible sheaves on an algebraic stack $\CY$ \footnote{Note that this is different from ind-completion of the DG category of constructible sheaves since not all constructible sheaves are compact.}. For a morphism $f:\CY_1\to \CY_2$ between algebraic stack, the usual direct image functor $f_*:\Shv(\CY_1)\to \Shv(\CY_2)$ is not, in general, colimit-preserving. This is a significant issue because we are working in $\PPr$, where the default 1-morphisms are assumed to be colimit-preserving. In \cite{AGKRRV20} Section A.2.3, they define the \emph{renormalized functor of direct image}:
\begin{defn}
(\cite{AGKRRV20} A.2.3) The \emph{renormalized functor of direct image}, denoted
\[ f_\Delta:\Shv(\CY_1)\to \Shv(\CY_2)  \]
to be the unique colimit-preserving functor such that restricts to $f_*$ on $\Shv(\CY_1)^c$.
\end{defn}

Then we define renormalized cochain complex of sheaves:
\begin{defn}
(\cite{AGKRRV20} A.2.4) For $\CY_1=\CY$ and $\CY_2=\pt$, we obtain the renormalized cochain complex of sheaves, denoted
\[ C_\Delta^\bullet(\CY,-):\Shv(\CY)\to \VVect      \]
We call its cohomology $\rmH_{\Delta}^\bullet(\CY,-)$ as \emph{renormalized cohomology}. 
\end{defn}

From now on, all stacks are quasi-compact Verdier compatible algebraic stack of quotient type.
\begin{rem}
To understand $f_\Delta$ (a more strong definition), it is enough to know just renormalized chain complex (a more weak definition) as the following construction:
$f_\Delta=(f^!)^\vee$:
\begin{align*}
\Shv(\CY_1) \xrightarrow{\coev_{\CY_2} \otimes \id} \Shv(\CY_2)\boxtimes \Shv(\CY_2)\boxtimes \Shv(\CY_1)&\xrightarrow{\id\otimes f^!\otimes\id}\Shv(\CY_2)\boxtimes \Shv(\CY_1)\boxtimes\Shv(\CY_1)\\ 
\xrightarrow{\id\otimes \ev_{\CY_1}} \Shv(\CY_2).
\end{align*}
Here 
\begin{itemize}
\item since $\Shv(\CY)$ is compactly generated, it has dual $\Shv(\CY)^\vee$. By assumption, we can identify $\Shv(\CY)\simeq \Shv(\CY)^\vee$ by Verdier self-duality;
\item $\coev_{\CY_2}:\VVect\to \Shv(\CY_2)\boxtimes \Shv(\CY_2)$ is the unit of Verdier self-duality; 
\item $\ev_{\CY_1}=C^\Delta_\bullet(\CY_1,-\boxtimes^!-)$. Here $-\boxtimes^!-:\Shv(\CY_1)\boxtimes \Shv(\CY_1)\hookrightarrow \Shv(\CY_1\times \CY_1)\xrightarrow{\Delta^!_{\CY_1}}\Shv(\CY_1)$.
\end{itemize}
\end{rem}

\begin{lem}
(\cite{AGKRRV20} Lemma A.2.6) The functor $f_\Delta:\Shv(\CY_1)\to \Shv(\CY_2)$ is cohomologically right bounded.
\end{lem}
The above lemma implies that after shifting, renormalized chain complex can be shifted to a non-negative chain complex.

Combining the computation results in Section \ref{computeableformulaandresults}, some physical intuition, especially the hint from Conjecture \ref{GSconj1}, we have:
\begin{conj}\label{GSconj}
(Gunningham-Safronov) For $q$ generic, and $M$ be a connected closed oriented 3-manifold.
\[ \rmH_\bullet(\SSk_G(M))\cong \rmH^{-\bullet}_{\Delta}(\Loc_G(M),\phi_{\Loc_G(M)})    \]
\end{conj}

Questions: In the root of unity case, do we have a deformation quantization analogue? For the underived case, it was conjectured by Pavel Safronov \cite{Pav22} that 
if $q$ is a good root of unity, there is a line bundle $\CL_q$ on $\Loc_{G^\vee}(M)$ such that 
\[ \Sk_{G,q}(M)\cong \rmH^0(\Loc_{G^\vee}(M),\CL_q)    \]
\begin{rem}
It's not hard to see that 
\[\rmH_\Delta^0(\Loc_G(M),\phi_{\Loc_G(M)})\cong \rmH^0(R_G(M),\phi_{R_G(M)}).\]
\end{rem}
\begin{rem}
 Bierent-Jordan-Vancraeynest-Vazirani computes a very important and interesting example: $\Gl_N$-skein module of mapping tori recently \cite{BJVV25}.  The positive property of their skein partition function can be also viewed as an evidence for the conjecture \ref{GSconj}.
\end{rem}
\begin{rem}
In \cite{DKS25}, Detcherry-Kalfagianni-Sikora checks some $\Sl_2$ examples for the 0-homology part of the  conjecture \ref{GSconj}.
\end{rem}

\subsection{Compact generation conjecture for derived skein category}

According to the results like \ref{Gmskcatmorskalg} and \ref{skcatS^2=Modskalg}, one would expect that the categorical finiteness conjecture in \cite{GJV24} still holds in the derived case:
\begin{conj}\label{GJVconj}
(Gunningham-Jordan-Vazirani) Let $G$ be a reductive algebraic group, $q$ is generic and $\Sigma$ is a closed surface. Then the stable presentable $\infty$-category $\widehat{\SSkcat_G(\Sigma)}$ admits a compact generator.
\end{conj}
\begin{rem}
If your input category $\CA$ admits a compact generator, then the statement that $\widehat{\SSkcat_\CA(\Sigma)}$ admits a compact generator is not surprising. However, a nontrivial point in the above conjecture is that $\rep_q(G)$ is not finitely generated, but finitely $\otimes$-generated. Then it would be surprising that its global integration becomes finitely generated. There are some deep fusions between representation theory and topology. As we have seen in the proof of  \ref{skcatS^2=Modskalg} and \ref{Gmskcatmorskalg}, the nondegeneracy, which is an ordinary linear notion, of $\rep_q(G)$ plays a core role here. It seems that, for this conjecture, the derived version is essentially equivalent to the ordinary version even if here the surface is closed. 
\end{rem}
\begin{rem}
In  \cite{GJV24}, they checked this conjecture for $G=\Gl_n,\Sl_n$ and $\Sigma=T^2$ case. In \cite{DJ26}, Detcherry-Jordan have checked this conjecture for $\Sigma=\Sigma_g$ and $G=\Sl_2$ case.
\end{rem}

\subsection{Derived skein Langlands conjecture}\label{Langlands}

As we said in Section \ref{SkeinDT}, derived skein module is a BV quantization of functions on the character variety. This already puts skein modules close to Betti Geometric Langlands, because Betti Geometric Langlands is also about character varieties/ local systems. So it is natural to ask
\[ \text{What is the quantum/topological Langlands duality for skein modules?}   \]

A labelled knot or ribbon graph in $M$ defines a function on the character variety. The basic example is a Wilson loop. Suppose we have a representation $\rho:\pi_1(M)\to G$ and a loop $\gamma \subset M$. Then $\rho(\gamma)\in G$, and if $V$ is a representation of $G$, we can take $\Tr_V(\rho(\gamma))$. These trace functions are just classical Wilson loop observables. This gives a function $\Loc_G(M)\to \mathbb{C}$. Thus
\[ \text{classical Wilson loops correspond to $\mathcal{O}(\Loc_G(M))$ }  \]
while
\[  \text{quantum Wilson loops correspond to $\SSk_{G,q}(M)$}.   \]
This is why derived skein modules should be viewed as a quantum version of the character stack side of Langlands.

Of course, as one of the motivations of this paper, one expects:
\begin{conj}\label{TLconj}
(Ben-Zvi-Gunningham-Jordan-Safronov) Let $q$ be not a root of unity, and $M$ be a connected closed oriented 3-manifold.
\[  \SSk_G(M)\cong \SSk_{G^L}(M)  \]
\end{conj}

The deeper reason comes from physics. Kapustin-Witten explained geometric Langlands using 4D $N=4$ supersymmetric Yang-Mills theory. This theory has an $S$-duality:
\[ Z_{G,q}\simeq Z_{G^L,q^L}   \]
where $Z_{G,\Psi}$ is the topological field theory with gauge group $G$ and parameter $q$. If derived skein module are the state spaces of that quantum theory, then S-duality predicts the above conjecture. Remark \ref{stableunderL} gives an evidence for this conjecture.

However, recently Pei Du post his paper on skein module and gauge theory \cite{Pei26}. The above conjecture was doubted from a physical point of view. As we said above, the skein module only uses Wilson lines, but S-duality exchanges Wilsons lines with 't Hooft lines. Therefore, the skein module $\Sk_G(M)$ is not automatically S-dual to $\Sk_{G^L}(M)$. He proposed that the correct $S$-dual object should include Wilson, 't Hooft, and more generally dyonic line operators thus should be larger than skein module.

On the other hand, Gunningham-Jordan-Safronov gives the first non-trivial confirmation of Skein Langlands duality of type A:
\begin{thm}
(\cite{Jor23} Theorem 5.2) Suppose that $q$ is generic,
\[   \Sk_{\Sl_2}(\Sigma_g\times S^1)\cong \Sk_{\PGL_2}(\Sigma_g\times S^1)    \]
\end{thm}

However Pei claims that type A is special. For $\Sl_N \leftrightarrow \PGL_N$, he thinks the extra subspaces removed by ignoring 't Hooft/dyonic operators seem to match on both sides. So for type A group, skein module can accidentally still satisfy Langlands duality. But this is not guaranteed in general.

\appendix

\appendixpage

\section{The relation between $A_\infty$-categories and DG categories}\label{AinftyandDG}

As we said in Remark \ref{BlobcomplexAinfty}, the ordinary Blob complex actually gives us an $A_\infty$ model for derived skein category rather than DG category. In this Appendix, we compare $A_\infty$-category and DG category to show that their differences are not important for our categorical approach in this paper. We recommend to read \cite{Orn18} for more details.

\begin{defn}
An \emph{(unital) $A_\infty$-category} $\CC$ is a $\mathbb{k}$-linear category such that
\begin{itemize}
\item for all $X,Y$ in $\ob(\CC)$ the Hom-sets $\Hom_\CC(X,Y)$ are chain complexes;
\item for all objects $X_1,\cdots X_n$ in $\ob(\CC)$ there is a family of linear composition maps (the higher compositions) 
\[ m_n:\Hom_\CC(X_0,X_1)\otimes \Hom_\CC(X_1,X_2)\otimes\cdots\otimes \Hom_\CC(X_{n-1},X_n)\to \Hom_\CC(X_0,X_n)  \]
of degree $n-2$ (homological grading convention is used) for $n\ge 1$. Note that $m_1$ is the differential on the chain complex $\Hom_\CC(X,Y)$;
\item $m_n$ satisfy the Stasheff identities for all $n\ge 0$.
\item for every object $X$, there is a unit $e_x$ of $x$ to be a morphism of degree 0 such that $m_2(f,e_x)=f$, $m_2(e_x,g)=(-1)^{\deg (g)} g$ and $m_n(\cdots,e_x,\cdots)=0$ for all $n\ne 2$.
\end{itemize}
\end{defn}

We don't give the precise Stasheff identities here but we can give some explaination: The first few identities have the familiar meaning. The relation for $n=1$ says that $m^2_1=0$, so every morphism space is a chain complex. The relation for $n=2$ says that $m_1$ is a derivation with respect to $m_2$. The relation for $n=3$ says that $m_2$ is associative up to the specified homotopy $m_3$. Thus a unit $A_\infty$-category is a category enriched in chain complexes whose composition is associative up to coherent higher homotopies.

\begin{defn}
A \emph{DG category} is a $A_\infty$-category such that $m_n=0$, for every $n>2$.
\end{defn}
\begin{rem}
The relation for $n=3$ says that the failure of $m_2$ to be strictly associative is controlled by $m_3$. Hence the condition $m_n=0$ for $n\ge 3$ is exactly the condition that composition be strictly associative.
\end{rem}

\begin{expl}
Let $\VVect$ be the DG category of all (non-negative) complexes. It has a canonical t-structure whose heart is the abelian category $\Vect$ of vector spaces. 
\end{expl}

\begin{defn}
Let $\CC$ be an $A_\infty$-category. Its \emph{homology category} $\rmH_0(\CC)$ is the ordinary category with the same objects as $\CC$ and with morphism space
\[ \Hom_{\rmH_0(\CC)}(X,Y):=\rmH_0(\Hom_\CA(X,Y),m_1)      \]
where the composition is induced by $m_2$.
\end{defn}

\begin{defn}
We define the opposite category of $\CC$, denoted by $\CC^\op$, to be the category defined as the $A_\infty$-category with the same objects but 
\begin{itemize}
\item morphism spaces 
\[ \Hom_{\CC^\op}(X,Y):=\Hom_{\CC}(Y,X)  \]
\item higher compositions defined by reversing the order of the inputs:
\[  m^\op_n(a_n,\cdots,a_1)=(-1)^\dagger m_n(a_1,\cdots,a_n)    \]
where the sign $(-1)^\dagger$ is the Koszul sign coming from reversing the order of the homogeneous inputs.
\item the unit is same.
\end{itemize}
\end{defn}

\begin{defn}
Let $\CC,\CD$ be $A_\infty$-categories. An \emph{(unital) $A_\infty$ functor} 
\[  F:\CC\to \CD   \]
consists of a map on objects together with multilinear maps
\[  F_n:\Hom_\CC(X_{n-1},X_n)\otimes\cdots\otimes \Hom_\CA(X_0,X_1)\to \Hom_\CD(F(X_0),F(X_n))  \]
of degree $n-1$, satisfying the $A_\infty$-functor identities and the unit conditions
\[ F_1(e_X)=e_{F(X)}  \]
and
\[ F_n(a_n,\cdots,a_1)=0  \]
for $n\ne 1$ whenever one of the inputs is a unit.
\end{defn}

\begin{defn}
An $A_\infty$-functor
\[ F:\CC\to \CD   \]
is a quasi-equivalence if, for every pair of objects $X,Y$ the map
\[ F_1:\Hom_\CC(X,Y)\to \Hom_\CD(F(X),F(Y))   \]
is a quasi-isomorphism, and the induced functor
\[ \rmH_0(F):\rmH_0(\CC)\to \rmH_0(\CD)   \]
is essentially surjective.
\end{defn}

\begin{defn}
A DG functor is a $A_\infty$-functor $F$ such that 
\[ F_n=0, \ \ \ \forall n\ge 2   \]
\end{defn}

\begin{expl}
The identity $A_\infty$-functor $\id_\CC:\CC\to \CC$ is the $A_\infty$-functor whose object map is identity and whose unary component is
\[ (\id_\CC)_1=\id_{\Hom_\CC(X,Y)}  \]
on every hom complex, and whose higher components vanish:
\[ (\id_\CC)_n=0, \ \ \ \forall n\ge 2  \]
\end{expl}

\begin{defn}
Let $\CC$ and $\CD$ be $A_\infty$-categories Their \emph{tensor product} $\CC\boxtimes\CD$  is the $A_\infty$-category whose objects are pairs
\[ (X,Y), \ \ \ X\in \ob(\CC), Y\in \ob(\CD)    \]
whose hom complexes are
\[  \Hom_{\CC\boxtimes\CD}((X,Y),(X',Y'))=\Hom_\CC(X,X')\otimes \Hom_\CD(Y,Y')   \]
and whose higher composition maps are induced from those of $\CC$ and $\CB$ using the standard Koszul sign rule. The unit object at $(X,Y)$ is $e_{(X,Y)}:=e_X\otimes e_Y$.
\end{defn}
\begin{expl}
Let $\CC$ be a small $A_\infty$-category. Then $\CC\boxtimes \VVect=\CC$.
\end{expl}

\begin{defn}
The symmetric monoidal $(\infty,1)$-category $\ACat$ has:
\begin{itemize}
\item objects are small $A_\infty$-categories.
\item 1-morphisms from $\CC$ to $\CD$ are the $A_\infty$-functors $\CC\to \CD$.
\item tensor product is given by tensor product $\boxtimes$ of $A_\infty$-categories.
\end{itemize}
\end{defn}
\begin{thm}
There is an equivalence of symmetric monoidal stable $\infty$-categories
\[  \ACat\simeq \DGcat     \]
\end{thm}

\begin{defn}
Let $\CC\in \ACat$ be an $A_\infty$-category. A \emph{right $\CC$-module} is a $A_\infty$-functor $\CC^\op\to \VVect$. Similarly, \emph{left $\CC$-module} is an $A_\infty$-functor $\CC\to \VVect$. A $(\CC,\CD)$-bimodule is an $A_\infty$-functor $\CC\boxtimes \CD^\op\to \VVect$.
\end{defn}

\begin{defn}
Let $F,G:\CC\to \CD$ be two $A_\infty$-functors. A \emph{prenatural transformation} of degree $g$ from $F$ to $G$ is a sequence
\[ T=(T_0,T_1,T_2,\cdots)  \]
where for each object $x\in \CC$
\[ T_0(x)\in \Hom_\CD(Fx,Gx)_g.  \]
For every $d\ge 1$, every sequence of objects
\[ x_0,x_1,\cdots,x_d  \]
in $\CC$, there is a map of degree $g+d$
\[ T_d:\Hom_\CC(x_{d-1},x_d)\otimes\cdots\otimes\Hom_\CC(x_0,x_1) \to \Hom_\CD(Fx_0,Gx_d) \]
\end{defn}

Denote $\Fun(\CC,\CD)$ be the category of $A_\infty$-functors between two $A_\infty$-categories $\CC$ and $\CD$, together with prenatural transformations
\begin{lem}
$\Fun(\CC,\CD)$ is an $A_\infty$-category. Moreover if $\CD$ is a DG category, then $\Fun(\CC,\CD)$ is a DG category.
\end{lem}
\begin{expl}
Let $\CC$ be an $A_\infty$-category. Then $\Fun(\CC^\op,\VVect)$ is a (compactly generated pretriangulated) DG category which we call it \emph{free cocompletion of $\CC$}.
\end{expl}

\begin{defn}
The symmetric monoidal $(\infty,1)$-category $\ABim$:
\begin{itemize}
\item objects are small $A_\infty$-categories.
\item 1-morphisms from $\CC$ to $\CD$ are $(\CC,\CD)$-bimodules.
\item tensor product is given by tensor product $\boxtimes$ of $A_\infty$-categories.
\end{itemize}
\end{defn}
\begin{thm}
There is an equivalence of symmetric monoidal $(\infty,1)$-categories
\[ \ABim\simeq \BBim   \]
\end{thm}

\bibliography{Top.bib}

@Book{EGNO15,
  author    = {Pavel Etingof and Shlomo Gelaki and Dmitri Nikshych and Victor Ostrik},
  publisher = {American Math. Soc},
  title     = {Tensor categories},
  year      = {2015},
  address   = {Providence, RI},
  isbn      = {9781470434410},
}

@Book{Lur17,
  author    = {Lurie, Jacob},
  publisher = {Book online},
  title     = {Higher Algebra},
  year      = {2017},
  note      = {\url{http://www.math.harvard.edu/~lurie/papers/HA.pdf}},
}

@Article{Lur08,
  author    = {Jacob Lurie},
  journal   = {Current Developments in Mathematics},
  title     = {On the Classification of Topological Field Theories},
  year      = {2008},
  number    = {1},
  pages     = {129--280},
  volume    = {2008},
  doi       = {10.4310/cdm.2008.v2008.n1.a3},
  eprint    = {0905.0465},
  publisher = {International Press of Boston},
}

@InCollection{AF20,
  author    = {David Ayala and John Francis},
  booktitle = {Handbook of Homotopy Theory},
  publisher = {Chapman and Hall/{CRC}},
  title     = {A factorization homology primer},
  year      = {2020},
  month     = {jan},
  pages     = {39--101},
  doi       = {10.1201/9781351251624-2},
  eprint    = {1903.10961},
}

@Article{DSPS19,
  author    = {Christopher L. Douglas and Christopher Schommer-Pries and Noah Snyder},
  journal   = {Kyoto Journal of Mathematics},
  title     = {The balanced tensor product of module categories},
  year      = {2019},
  month     = {apr},
  number    = {1},
  volume    = {59},
  doi       = {10.1215/21562261-2018-0006},
  eprint    = {1406.4204},
  publisher = {Duke University Press},
}

@Article{AF15,
  author    = {David Ayala and John Francis},
  journal   = {Journal of Topology},
  title     = {Factorization homology of topological manifolds},
  year      = {2015},
  month     = {oct},
  number    = {4},
  pages     = {1045--1084},
  volume    = {8},
  doi       = {10.1112/jtopol/jtv028},
  eprint    = {1206.5522},
  publisher = {Wiley},
}

@Article{AFR18,
  author    = {David Ayala and John Francis and Nick Rozenblyum},
  journal   = {Advances in Mathematics},
  title     = {Factorization homology I: Higher categories},
  year      = {2018},
  month     = {jul},
  pages     = {1042--1177},
  volume    = {333},
  doi       = {10.1016/j.aim.2018.05.031},
  eprint    = {1504.04007},
  publisher = {Elsevier {BV}},
}

@Article{BZBJ18,
  author    = {David Ben-Zvi and Adrien Brochier and David Jordan},
  journal   = {Journal of Topology},
  title     = {Integrating quantum groups over surfaces},
  year      = {2018},
  month     = {aug},
  number    = {4},
  pages     = {874--917},
  volume    = {11},
  doi       = {10.1112/topo.12072},
  eprint    = {1501.04652},
  publisher = {Wiley},
}

@Article{BZBJ18a,
  author    = {David Ben-Zvi and Adrien Brochier and David Jordan},
  journal   = {Selecta Mathematica},
  title     = {Quantum character varieties and braided module categories},
  year      = {2018},
  month     = {jul},
  number    = {5},
  pages     = {4711--4748},
  volume    = {24},
  doi       = {10.1007/s00029-018-0426-y},
  eprint    = {1606.04769},
  publisher = {Springer Science and Business Media {LLC}},
}

@Article{CR16,
  author    = {Nils Carqueville and Ingo Runkel},
  journal   = {Quantum Topology},
  title     = {Orbifold completion of defect bicategories},
  year      = {2016},
  number    = {2},
  pages     = {203--279},
  volume    = {7},
  doi       = {10.4171/qt/76},
  eprint    = {1210.6363},
  publisher = {European Mathematical Society - {EMS} - Publishing House {GmbH}},
}

@Article{CRS19,
  author    = {Nils Carqueville and Ingo Runkel and Gregor Schaumann},
  journal   = {Geometry {\&} Topology},
  title     = {Orbifolds of n-dimensional defect {TQFTs}},
  year      = {2019},
  month     = {apr},
  number    = {2},
  pages     = {781--864},
  volume    = {23},
  doi       = {10.2140/gt.2019.23.781},
  eprint    = {1705.06085},
  publisher = {Mathematical Sciences Publishers},
}

@article{AMR24,
  title={Symmetries of the cyclic nerve},
  author={Ayala, David and Mazel-Gee, Aaron and Rozenblyum, Nick},
  journal={arXiv preprint arXiv:2405.03897},
  year={2024}
}

@book{GR19,
  title={A study in derived algebraic geometry: Volume I: correspondences and duality},
  author={Gaitsgory, Dennis and Rozenblyum, Nick},
  volume={221},
  year={2019},
  publisher={American Mathematical Society}
}

@article{BZN09,
  title={The character theory of a complex group},
  author={Ben-Zvi, David and Nadler, David},
  journal={arXiv preprint arXiv:0904.1247},
  year={2009}
}

@article{LO10,
  title={Uniqueness of enhancement for triangulated categories},
  author={Lunts, Valery and Orlov, Dmitri},
  journal={Journal of the American Mathematical Society},
  volume={23},
  number={3},
  pages={853--908},
  year={2010}
}

@article{SAG18,
  title={Spectral algebraic geometry},
  author={Lurie, Jacob},
  journal={preprint},
  year={2018}
}

@article{DAG2,
  title={Derived algebraic geometry II: Noncommutative algebra},
  author={Lurie, Jacob},
  journal={arXiv preprint math/0702299},
  year={2007}
}

@article{DG13,
  title={On some finiteness questions for algebraic stacks},
  author={Drinfeld, Vladimir and Gaitsgory, Dennis},
  journal={Geometric and Functional Analysis},
  volume={23},
  number={1},
  pages={149--294},
  year={2013},
  publisher={Springer}
}

@article{BGR22,
  title={PROJECTIVE GENERATION FOR EQUIVARIANT-MODULES},
  author={Bellamy, Gwyn and Gunningham, Sam and Raskin, Sam},
  journal={Transformation Groups},
  volume={27},
  number={3},
  pages={737--749},
  year={2022},
  publisher={Springer}
}

@article{Wam97,
  title={Hochschild and cyclic homology of the quantum multiparametric torus},
  author={Wambst, Marc},
  journal={Journal of Pure and Applied Algebra},
  volume={114},
  number={3},
  pages={321--329},
  year={1997},
  publisher={Elsevier}
}

@book{Lod13,
  title={Cyclic homology},
  author={Loday, Jean-Louis},
  volume={301},
  year={2013},
  publisher={Springer Science \& Business Media}
}

@article{Ara23,
  title={Homotopy limits and homotopy colimits of chain complexes},
  author={Arakawa, Kensuke},
  journal={arXiv preprint arXiv:2310.00201},
  year={2023}
}

@incollection{JF15,
  title={Heisenberg-picture quantum field theory},
  author={Johnson-Freyd, Theo},
  booktitle={Representation Theory, Mathematical Physics, and Integrable Systems: In Honor of Nicolai Reshetikhin},
  pages={371--409},
  year={2021},
  publisher={Springer}
}

@misc{Wal06,
  title={TQFTs},
  author={Walker, Kevin},
  year={2006}
}

@article{MW12,
  title={Blob homology},
  author={Morrison, Scott and Walker, Kevin},
  journal={Geometry \& Topology},
  volume={16},
  number={3},
  pages={1481--1607},
  year={2012},
  publisher={Mathematical Sciences Publishers}
}

@article{coo19,
  title={Excision of skein categories and factorisation homology},
  author={Cooke, Juliet},
  journal={Advances in Mathematics},
  volume={414},
  pages={108848},
  year={2023},
  publisher={Elsevier}
}

@article{BH24,
  title={Skein Categories in Non-semisimple Settings},
  author={Brown, Jennifer and Ha{\"\i}oun, Benjamin},
  journal={Selecta Mathematica},
  volume={32},
  number={2},
  pages={34},
  year={2026},
  publisher={Springer}
}

@article{MWW22,
  title={Invariants of 4-manifolds from Khovanov-Rozansky link homology},
  author={Morrison, Scott and Walker, Kevin and Wedrich, Paul},
  journal={Geom. Topol},
  volume={26},
  number={8},
  pages={3367--3420},
  year={2022}
}

@article{BFK97,
  title={Skein homology},
  author={Bullock, Doug and Frohman, Charles and Kania-Bartoszy{\'n}ska, Joanna},
  journal={Canadian Mathematical Bulletin},
  volume={41},
  number={2},
  pages={140--144},
  year={1998},
  publisher={Cambridge University Press}
}

@article{Pav22,
  title={Skein module and 4D TQFT},
  author={Pavel Safronov},
  journal={TQFT club talk,  https://tqft.math.tecnico.ulisboa.pt/seminars?year=2022},
  year={2022},
}

@article{SW21,
   title={Homotopy coherent mapping class group actions and excision for Hochschild complexes of modular categories},
   volume={386},
   ISSN={0001-8708},
   url={http://dx.doi.org/10.1016/j.aim.2021.107814},
   DOI={10.1016/j.aim.2021.107814},
   journal={Advances in Mathematics},
   publisher={Elsevier BV},
   author={Schweigert, Christoph and Woike, Lukas},
   year={2021},
   month=aug, pages={107814} }

@article{HRW24,
  title={Bordered invariants from Khovanov homology},
  author={Hogancamp, Matthew and Rose, David EV and Wedrich, Paul},
  journal={arXiv preprint arXiv:2404.06301},
  year={2024}
}

@article{GJS22,
  title={The finiteness conjecture for skein modules},
  author={Gunningham, Sam and Jordan, David and Safronov, Pavel},
  journal={Inventiones mathematicae},
  volume={232},
  number={1},
  pages={301--363},
  year={2023},
  publisher={Springer}
}

@article{GS23,
  title={Deformation quantization and perverse sheaves},
  author={Gunningham, Sam and Safronov, Pavel},
  journal={Duke Mathematical Journal},
  volume={175},
  number={6},
  pages={1067--1161},
  year={2026},
  publisher={Duke University Press}
}

@article{BFN10,
  title={Integral transforms and Drinfeld centers in derived algebraic geometry},
  author={Ben-Zvi, David and Francis, John and Nadler, David},
  journal={Journal of the American Mathematical Society},
  volume={23},
  number={4},
  pages={909--966},
  year={2010}
}

@article{AM19,
  title={A sheaf-theoretic model for $SL_2(\mathbb{C})$ Floer homology},
  author={Abouzaid, Mohammed and Manolescu, Ciprian},
  journal={J. Eur. Math. Soc.(JEMS)},
  volume={22},
  number={11},
  pages={3641--3695},
  year={2020}
}

@article{Ho45,
  title={On the cohomology groups of an associative algebra},
  author={Hochschild, Gerhard},
  journal={Annals of Mathematics},
  volume={46},
  number={1},
  pages={58--67},
  year={1945},
  publisher={JSTOR}
}

@article{Jor23,
  title={Langlands duality for skein modules of 3-manifolds},
  author={Jordan, David},
  journal={String-Math},
  volume={107},
  number={2024},
  pages={127},
  year={2022}
}

@misc{BR26, 
      author={Chun-Yu Bai and Qiuyu Ren},
      year={2026},
      journal={In progress}
}

@article{Prz91,
  title={Skein modules of 3-manifolds},
  author={Przytycki, J{\'o}zef H},
  journal={arXiv preprint},
  year={1991}
}

@article{Tur90,
  title={Conway and Kauffman modules of a solid torus},
  author={Turaev, Vladimir G},
  journal={Journal of Soviet Mathematics},
  volume={52},
  number={1},
  pages={2799--2805},
  year={1990},
  publisher={Kluwer Academic Publishers-Plenum Publishers New York}
}

@article{RW24,
  title={Khovanov homology and exotic $4 $-manifolds},
  author={Ren, Qiuyu and Willis, Michael},
  journal={arXiv preprint arXiv:2402.10452},
  year={2024}
}

@article{GJV24,
  title={Skeins on tori},
  author={Gunningham, Sam and Jordan, David and Vazirani, Monica},
  journal={arXiv preprint arXiv:2409.05613},
  year={2024}
}

@article{Car17,
  title={Nine generators of the skein space of the 3--torus},
  author={Carrega, Alessio},
  journal={Algebraic \& geometric topology},
  volume={17},
  number={6},
  pages={3449--3460},
  year={2017},
  publisher={Mathematical Sciences Publishers}
}

@article{Gil18,
  title={On the Kauffman bracket skein module of the 3-torus},
  author={Gilmer, Patrick M},
  journal={Indiana University Mathematics Journal},
  pages={993--998},
  year={2018},
  publisher={JSTOR}
}

@article{BJVV25,
  title={The skein partition function of the mapping torus},
  author={Bierent, Julia and Jordan, David and Vancraeynest, Matthias and Vazirani, Monica},
  journal={arXiv preprint arXiv:2511.19139},
  year={2025}
}

@article{DKS25,
   title={Kauffman bracket skein modules of small 3-manifolds},
   volume={467},
   ISSN={0001-8708},
   url={http://dx.doi.org/10.1016/j.aim.2025.110169},
   DOI={10.1016/j.aim.2025.110169},
   journal={Advances in Mathematics},
   publisher={Elsevier BV},
   author={Detcherry, Renaud and Kalfagianni, Efstratia and Sikora, Adam S.},
   year={2025},
   month=may, pages={110169} }

@article{Pei26,
  title={Gauge Theory and Skein Modules},
  author={Pei, Du},
  journal={arXiv preprint arXiv:2601.16213},
  year={2026}
}

@article{AFMR17,
  title={Factorization homology of enriched $\infty$-categories},
  author={Ayala, David and Francis, John and Mazel-Gee, Aaron and Rozenblyum, Nick},
  journal={arXiv preprint arXiv:1710.06414},
  year={2017}
}

@article{Orn18,
  title={COMPARISON RESULTS ABOUT DG-CATEGORIES, $A_\infty$-CATEGORIES, STABLE $\infty$-CATEGORIES AND NONCOMMUTATIVE MOTIVES},
  author={Ornaghi, Mattia},
journal={PhD thesis},
  year={2018}
}

@misc{DJ26, 
      author={Renaud Detcherry and David Jordan},
      journal={in preparation}
      }

@article{Saf19,
  title={A categorical approach to quantum moment maps},
  author={Safronov, Pavel},
  journal={arXiv preprint arXiv:1901.09031},
  year={2019}
}

@article{BCT09,
  title={Open-closed field theories, string topology, and Hochschild homology},
  author={Blumberg, Andrew J and Cohen, Ralph L and Teleman, Constantin},
  journal={arXiv preprint arXiv:0906.5198},
  year={2009}
}

@article{Kin24,
  title={Non-semisimple Crane-Yetter theory varying over the character stack},
  author={Kinnear, Patrick},
  journal={arXiv preprint arXiv:2404.19667},
  year={2024}
}

@article{AFRcor,
  title={Corrigendum to “Factorization homology I: Higher categories”[Adv. Math. 333 (2018) 1042--1177]},
  author={Ayala, David and Francis, John and Rozenblyum, Nick},
  journal={Advances in Mathematics},
  volume={370},
  pages={107217},
  year={2020},
  publisher={Elsevier}
}

@article{PSY13,
  title={Cohomologically cofinite complexes},
  author={Porta, Marco and Shaul, Liran and Yekutieli, Amnon},
  journal={Communications in Algebra},
  volume={43},
  number={2},
  pages={597--615},
  year={2015},
  publisher={Taylor \& Francis}
}

@article{KS10,
  title={Deformation quantization modules},
  author={Kashiwara, Masaki and Schapira, Pierre},
  journal={arXiv preprint arXiv:1003.3304},
  year={2010}
}

@article{NS23,
  title={Torsion volume forms},
  author={Naef, Florian and Safronov, Pavel},
  journal={arXiv preprint arXiv:2308.08369},
  year={2023}
}

@article{BBBj15,
   title={A ‘Darboux theorem’ for shifted symplectic structures on derived Artin stacks, with applications},
   volume={19},
   ISSN={1465-3060},
   url={http://dx.doi.org/10.2140/gt.2015.19.1287},
   DOI={10.2140/gt.2015.19.1287},
   number={3},
   journal={Geometry \& Topology},
   publisher={Mathematical Sciences Publishers},
   author={Ben-Bassat, Oren and Brav, Christopher and Bussi, Vittoria and Joyce, Dominic},
   year={2015},
   month=May, pages={1287–1359} }

@misc{Bul97,
      title={Rings of $SL_2({\mathbb C})$-Characters and the Kauffman Bracket Skein Module}, 
      author={Doug Bullock},
      year={1996},
      eprint={q-alg/9604014},
      archivePrefix={arXiv},
      primaryClass={q-alg},
      url={https://arxiv.org/abs/q-alg/9604014}, 
}

@article{PS00,
  title={On Skein Algebras And Sl\_2 (C)-Character Varieties},
  author={Przytycki, J{\'o}zef H and Sikora, Adam S},
  journal={arXiv preprint q-alg/9705011},
  year={1997}
}

@article{BJ25,
  title={Parabolic skein modules},
  author={Brown, Jennifer and Jordan, David},
  journal={arXiv preprint arXiv:2505.14836},
  year={2025}
}

@article{AF17,
  title={The cobordism hypothesis},
  author={Ayala, David and Francis, John},
  journal={arXiv preprint arXiv:1705.02240},
  year={2017}
}

@article{JLSS21,
  title={Quantum decorated character stacks},
  author={Jordan, David and Le, Ian and Schrader, Gus and Shapiro, Alexander},
  journal={arXiv preprint arXiv:2102.12283},
  year={2021}
}

@article{Dim13,
  title={Quantum Riemann surfaces in Chern--Simons theory},
  author={Dimofte, Tudor},
  journal={Adv. Theor. Math. Phys},
  volume={17},
  pages={479--599},
  year={2013}
}

@article{GY26,
  title={A quantum trace map for 3-manifolds},
  author={Garoufalidis, Stavros and Yu, Tao},
  journal={Advances in Mathematics},
  volume={486},
  pages={110735},
  year={2026},
  publisher={Elsevier}
}

@article{Gar04,
  title={On the characteristic and deformation varieties of a knot},
  author={Garoufalidis, Stavros},
  journal={Geom. Topol. Monogr},
  volume={7},
  pages={291--304},
  year={2004}
}

@article{DimGar13,
  title={The quantum content of the gluing equations},
  author={Dimofte, Tudor and Garoufalidis, Stavros},
  journal={Geometry \& Topology},
  volume={17},
  number={3},
  pages={1253--1315},
  year={2013},
  publisher={Mathematical Sciences Publishers}
}

@article{BW10a,
  title={Kauffman brackets, character varieties, and triangulations of surfaces},
  author={Bonahon, Francis and Wong, Helen},
  journal={arXiv preprint arXiv:1009.0084},
  year={2010}
}

@article{BW11,
  title={Quantum traces for representations of surface groups in $SL_2(\mathbb{C})$},
  author={Bonahon, Francis and Wong, Helen},
  journal={Geometry \& Topology},
  volume={15},
  number={3},
  pages={1569--1615},
  year={2011},
  publisher={Mathematical Sciences Publishers}
}

@article{BW16,
  title={Representations of the Kauffman bracket skein algebra I: invariants and miraculous cancellations},
  author={Bonahon, Francis and Wong, Helen},
  journal={Inventiones mathematicae},
  volume={204},
  number={1},
  pages={195--243},
  year={2016},
  publisher={Springer}
}

@article{BW17,
  title={Representations of the Kauffman bracket skein algebra, II: Punctured surfaces},
  author={Bonahon, Francis and Wong, Helen},
  journal={Algebraic \& geometric topology},
  volume={17},
  number={6},
  pages={3399--3434},
  year={2017},
  publisher={Mathematical Sciences Publishers}
}

@article{FSY23,
  title={String-net models for pivotal bicategories},
  author={Fuchs, J{\"u}rgen and Schweigert, Christoph and Yang, Yang},
  journal={arXiv preprint arXiv:2302.01468},
  year={2023}
}

@article{Walker22,
  title={Going from $m+\varepsilon$ to $n+1$ in non-semisimple oriented TQFTs.},
  author={Kevin Walker},
  journal={Talk at the UQSL seminar https://canyon23.net/math/talks/np1.pdf.},
  year={2022}
}

@article{CGBPM26,
  title={Skein (3+ 1)-TQFTs from non-semisimple ribbon categories},
  author={Costantino, Francesco and Geer, Nathan and Ha{\"\i}oun, Benjamin and Patureau Mirand, Bertrand and others},
  journal={SIGMA. Symmetry, Integrability and Geometry: Methods and Applications},
  volume={22},
  pages={034},
  year={2026},
  publisher={SIGMA. Symmetry, Integrability and Geometry: Methods and Applications}
}

@article{Hai24,
  title={Defining extended TQFTs via handle attachments},
  author={Ha{\"\i}oun, Benjamin},
  journal={arXiv preprint arXiv:2412.14649},
  year={2024}
}

@article{JR25,
  title={Finiteness and holonomicity of skein modules},
  author={Jordan, David and Romaidis, Iordanis},
  journal={arXiv preprint arXiv:2509.22313},
  year={2025}
}

@misc{Van, 
      title={Weaves, skeins and higher rank cluster charts. Upcoming},
      author={Vancraeynest,Matthias},
}

@misc{Bie, 
      title={Factorisation homology of wild character varieties. Upcoming},
      author={Julia Bierent},
}

@article{CN25,
  title={Cochain valued TQFTs from nonsemisimple modular tensor categories},
  author={Czenky, Agustina and Negron, Cris},
  journal={arXiv preprint arXiv:2507.17169},
  year={2025}
}

@article{AGKRRV20,
  title={Duality for automorphic sheaves with nilpotent singular support},
  author={Arinkin, Dima and Gaitsgory, Dennis and Kazhdan, David and Raskin, Sam and Rozenblyum, Nick and Varshavsky, Yasha},
  journal={arXiv preprint arXiv:2012.07665},
  year={2020}
}

@article{Pri19,
  title={Deformation quantisation for (-1)-shifted symplectic structures and vanishing cycles},
  author={Pridham, JP},
  journal={Algebraic Geometry},
  volume={6},
  number={6},
  pages={747--779},
  year={2019},
  publisher={European Mathematical Society}
}

@article{Pri26,
  title={Deformation quantisation of exact shifted symplectic structures, with an application to vanishing cycles},
  author={Pridham, JP},
  journal={arXiv preprint arXiv:2511.07602},
  year={2025}
}

@article{Toe07,
  title={The homotopy theory of dg-categories and derived Morita theory},
  author={To{\"e}n, Bertrand},
  journal={Inventiones mathematicae},
  volume={167},
  number={3},
  pages={615--667},
  year={2007},
  publisher={Springer}
}

@article{Tab07,
  title={A new Quillen model for the Morita homotopy theory of DG categories},
  author={Tabuada, Goncalo},
  journal={arXiv preprint math/0701205},
  year={2007}
}

@article{Prz98,
  title={Fundamentals of Kauffman bracket skein modules},
  author={Przytycki, J{\'o}zef H},
  journal={arXiv preprint math/9809113},
  year={1998}
}

@article{BD25,
  title={On torsion in the Kauffman bracket skein module of 3-manifolds},
  author={Belletti, Giulio and Detcherry, Renaud},
  journal={International Mathematics Research Notices},
  volume={2025},
  number={21},
  pages={rnaf333},
  year={2025},
  publisher={Oxford University Press}
}

@article{Kal25,
  title={Skein modules, character varieties and essential
surfaces of 3-manifolds},
  author={ E. Kalfagianni},
  journal={Quantum topology conference Bonn MPIM,  
  https://users.math.msu.edu/users/kalfagia/MPI.pdf,  year=2025  },
  year={2025},
}

@article{AF24,
  title={The tangle hypothesis: Dimension 1},
  author={Ayala, David and Francis, John},
  journal={arXiv preprint arXiv:2410.23965},
  year={2024}
}

@article{KW07,
  title={Electric-magnetic duality and the geometric Langlands program},
  author={Kapustin, Anton and Witten, Edward},
  journal={Communications in number theory and physics},
  volume={1},
  number={1},
  pages={1--236},
  year={2007},
  publisher={International Press of Boston, Inc. Somerville, MA 02143, USA}
}

@misc{DC, 
      title={ https://kskedlaya.org/prismatic/sec\_derived-complete.html},
      author={Kiran S. Kedlaya}
}

@article{GY24,
  title={The 3d-index of the 3d-skein module via the quantum trace map},
  author={Garoufalidis, Stavros and Yu, Tao},
  journal={arXiv preprint arXiv:2406.04918},
  year={2024}
}

@article{PP24,
  title={3d quantum trace map},
  author={Panitch, Samuel and Park, Sunghyuk},
  journal={arXiv preprint arXiv:2403.12850},
  year={2024}
}
\end{document}